\documentclass[leqno, 12pt]{imsart}
\usepackage[margin=2.7cm]{geometry}

\usepackage[numbers]{natbib}

\usepackage{multicol}

\usepackage[dvips]{graphics}
\DeclareGraphicsExtensions{.eps.gz,.eps,.epsi.gz,.epsi,.ps,.ps.gz}
\usepackage{epsfig}
\usepackage{color}
\graphicspath{{fig/}}

\usepackage[dvipsnames]{xcolor}

\usepackage{latexsym}
\usepackage{graphicx}
\usepackage{amsmath}
\usepackage{amsthm}
\usepackage{amsfonts}
\usepackage{amssymb}
\usepackage{mathtools}
\usepackage{mathrsfs}
\usepackage[lofdepth,lotdepth]{subfig}
\usepackage{multirow}
\usepackage{booktabs}
\usepackage{lscape}
\usepackage{constants}

\usepackage{hyperref}
\usepackage{cleveref}
\hypersetup{
    breaklinks=true,
    linkcolor=blue,
    citecolor=blue,
    colorlinks=true,
pdfborder={0 0 0}}

\usepackage{thmtools}
\usepackage{thm-restate}

\numberwithin{equation}{section}

\def\argmin{\mathop{\rm arg\, min}}

\newcommand{\bel}{\begin{eqnarray}\label}
\newcommand{\eel}{\end{eqnarray}}
\newcommand{\bes}{\begin{eqnarray*}}
\newcommand{\ees}{\end{eqnarray*}}
\newcommand{\bei}{\begin{itemize}}
\newcommand{\eei}{\end{itemize}}
\newcommand{\beiftnt}{\begin{itemize}\footnotesize}

\def\benu{\begin{enumerate}}
\def\eenu{\end{enumerate}}

\def\argmin{\mathop{\rm arg\, min}}
\def\real{{\mathbb{R}}}
\def\R{{\real}}
\def\E{{\mathbb{E}}}

\def\P{{\mathbb{P}}}

\def\complex{\mathop{{\rm I}\kern-.58em\hbox{\rm C}}\nolimits}
\def\pa{\partial}

\def\diag{\hbox{\rm diag}}

\def\rank{\hbox{\rm rank}}

\def\sgn{\hbox{\rm sgn}}

\DeclareMathOperator{\trace}{trace}

\def\Var{\hbox{\rm Var}}

\def\mathbold{\boldsymbol} 

\def\ba{\mathbold{a}}

\def\bA{\mathbold{A}}

\def\bb{\mathbold{b}}

\def\bfe{\mathbold{e}}

\def\bff{\mathbold{f}}
\def\tbf{{\widetilde{\bff}}}

\def\bg{\mathbold{g}}

\def\bG{\mathbold{G}}
\def\hbG{{\widehat{\bG}}}

\def\bh{\mathbold{h}}
\def\tbh{{\widetilde{\bh}}}
\def\bH{\mathbold{H}}\def\scrH{{\mathscr H}}
\def\hbH{{\widehat{\bH}}}

\def\bI{\mathbold{I}}

\def\bM{\mathbold{M}}

\def\bP{\mathbold{P}}

\def\bQ{\mathbold{Q}}

\def\bs{\mathbold{s}}

\def\bS{\mathbold{S}}
\def\Shat{{\widehat{S}}}\def\Sbar{{\overline S}}

\def\be{\mathbold{e}}

\def\bu{\mathbold{u}}

\def\bv{\mathbold{v}}

\def\bw{\mathbold{w}}

\def\bW{\mathbold{W}}\def\tbW{{\widetilde{\bW}}}

\def\bx{\mathbold{x}}

\def\bX{\mathbold{X}}\def\tbX{{\widetilde{\bX}}}\def\bXbar{{\overline \bX}}

\def\by{\mathbold{y}}
\def\bybar{\overline{\by}}
\def\tby{{\widetilde{\by}}}

\def\bz{\mathbold{z}}
\def\hbz{{\widehat{\bz}}}

\def\bbeta{\mathbold{\beta}}\def\hbeta{\widehat{\beta}}

\def\hbbeta{{\smash{\widehat{\bbeta}}}}
\def\tbbeta{{\smash{\widetilde{\bbeta}}}}
\def\bbetabar{{\overline\bbeta}}

\def\ep{\varepsilon}\def\eps{\epsilon}
\def\bep{ {\mathbold{\ep} }}
\def\bepbar{ \overline{{\mathbold{\ep} }} }

\def\tbep{{\widetilde{\bep}}}

\def\bfeta{\mathbold{\eta}}

\def\htheta{\widehat{\theta}}

\def\kappabar{{\overline{\kappa}}}

\def\lam{\lambda}

\def\bmu{\mathbold{\mu}}

\def\hbmu{{\widehat{\bmu}}}

\def\xibar{{\overline{\xi}}}

\def\bPi{\mathbold{\Pi}}

\def\bSigma{\mathbold{\Sigma}}

\def\bphi{\mathbold{\phi}}

\declaretheorem[name=Theorem,numberwithin=section]{theorem}
\declaretheorem[name=Proposition,sibling=theorem]{proposition}
\declaretheorem[name=Lemma,sibling=theorem]{lemma}

\declaretheorem[name=Assumption,numberwithin=section]{assumption}

\declaretheorem[name=Example,style=definition]{example}

\crefname{assumption}{assumption}{assumptions}

\usepackage{marginnote}

\usepackage{soul}

\usepackage[textsize=scriptsize,textwidth=1.1in]{todonotes}

\usepackage[color]{changebar} 
\cbcolor{Apricot}
\def\risk{{{R_*}}}

\def\df{{\widehat{\mathsf{df}}}{}}

\def\argmin{\mathop{\rm arg\, min}}

\def\defas{\stackrel{\text{\tiny def}}{=}}

\usepackage{enumitem}

\DeclareMathOperator{\dv}{div}

\def\DeBias{\hbbeta{}^{\text{\tiny (de-bias)} } }
\def\debias{\hbeta{}^{\text{\tiny (de-bias)} } }

\def\htheta{{\widehat\theta}}
\def\hbeta{{\widehat\beta}}

\def\bmubar{{\overline\bmu}}
\def\epsbar{{\overline\eps}}
\def\epsdoublebar{\overline{{\overline\eps}}}
\def\bAbar{{\overline\bA}}

\newcommand\restateIfEnabled[1]{ #1 } 

\begin{document}

\title{
De-biasing 
convex regularized estimators and interval estimation in linear models 
}
\runtitle{Asymptotic normality when $p/n\to\gamma$}
\author{Pierre C Bellec and Cun-Hui Zhang}
\runauthor{Bellec and Zhang}
\date{\today}
%
%
%
%
%
%
%

\begin{abstract}
    New 
    upper bounds are developed for the $L_2$ distance 
    between $\xi/\Var[\xi]^{1/2}$ and linear and quadratic functions of $\bz\sim N({\bf 0},\bI_n)$ 
    for random variables of the form 
    ${\xi}=\bz^\top f(\bz) - \dv f(\bz)$. The linear approximation 
    yields a central limit theorem 
    when the squared norm of $f(\bz)$ dominates the squared
    Frobenius norm of $\nabla f(\bz)$ in expectation. 

    Applications of this normal approximation
    are given for the asymptotic normality of de-biased
    estimators in linear regression with correlated design
    and convex penalty in the regime $p/n \le \gamma$
    for constant $\gamma\in(0,{\infty})$. 
    For the estimation of linear functions $\langle \ba_0,\bbeta\rangle$
    of the unknown coefficient vector $\bbeta$,
    this analysis leads to
    asymptotic normality of the de-biased estimate
    for most normalized directions $\ba_0$, where ``most'' is quantified in a precise
    sense.
    This asymptotic normality holds
    for any convex penalty if $\gamma<1$ and
    for any strongly convex
    penalty if $\gamma\ge 1$.
    In particular the penalty needs not be separable or permutation
    invariant.
        By allowing arbitrary regularizers, the results
        vastly broaden the scope of applicability of de-biasing
        methodologies to obtain confidence intervals
        in high-dimensions. 
        In the absence of strong convexity for $p>n$,
        asymptotic normality of the de-biased estimate is obtained
        for the Lasso and the group Lasso under additional conditions. 
        For general convex penalties, 
        our analysis also provides prediction and estimation error bounds of independent interest. 
\end{abstract}

\maketitle

\section{Introduction}
Consider the linear model
\bel{LM}
\by = \bX\bbeta + \bep
\eel
with an unknown coefficient vector $\bbeta\in \R^p$, a Gaussian noise vector 
$\bep\sim N({\bf 0},\sigma^2\bI_n)$, and a Gaussian design matrix $\bX\in \R^{n\times p}$ with iid 
$N({\bf 0},\bSigma)$ rows independent of $\bep$. 
We assume throughout the sequel that $\bSigma$ is invertible. 
The paper develops confidence intervals for
$\theta  = \langle \ba_0, \bbeta \rangle$
from a given regularized initial estimator $\hbbeta\in\R^p$,
using a technique referred to as \emph{de-biasing}:
a correction to the initial estimate $\langle \ba_0,\hbbeta\rangle$ in the
direction $\ba_0$ is constructed so that the 
\emph{``de-biased"} estimate
can be used for inference about 
$\theta=\langle\ba_0,\bbeta\rangle$. 


\subsection{Regularization induces bias}
\label{sec:intro-regularization-induces-bias}
If $\bX^\top\bX$ is invertible,
the unregulated least-squares estimate $\hbbeta^{ls} = (\bX^\top\bX)^{-1}\bX^\top\by$ is unbiased, that is, $\E[\hbbeta^{ls}-\bbeta | \bX] = \mathbf{0}$.
On the other hand, if the square loss is regularized with an additive penalty,
\begin{equation}
\hbbeta=\argmin_{\bb\in\R^p} \|\by-\bX\bb\|^2/(2n) + g(\bb)
\label{eq:intro-hbbeta-g}
\end{equation}
for penalty functions commonly used in high-dimensional statistics
such as $g(\bb)=\lambda\|\bb\|_1$ for $\lambda>0$ (Lasso)
or $g(\bb) = \mu\|\bb\|_2^2$ for $\mu>0$ (ridge regression), 
then $\hbbeta$ is biased. 

For ridge regression $\hbbeta = (\bX^\top\bX +  n\mu \bI_p)^{-1}\bX^\top\by$, 
this bias can be quantified explicitly when $\bSigma=\bI_p$ as a shrinkage to the origin. 
Let $\sum_{i=1}^r \bu_i s_i \bv_i^\top$ be the SVD of $\bX$ with $s_i>0$ and $r=\min(n,p)$. 
By rotational invariance, $\bv_i$ is independent of $s_i$ 
and uniformly distributed in the unit sphere in $\R^p$. 
Thus, with $G_{\gamma}$ being the Marchenko-Pastur law,
$$
\E\big[\hbbeta\big]
= \E\bigg[\sum_{i=1}^r \frac{s_i^2\bv_i\bv_i^\top\bbeta }{s_i^2 + n\mu}\bigg]
= \E\biggl[\sum_{i=1}^r \frac{p^{-1}s_i^2}{s_i^2 + n\mu}\bigg]\bbeta
\approx \bbeta \int \frac{(r/p)x}{x+(r/p)\mu}G_{\gamma}(dx)
\ \text{ as }\frac{p}{n}\to \gamma.
$$

The Lasso penalty $g(\bb) = \lambda \|\bb\|_1$ also introduces bias. 
For example, for deterministic orthonormal designs, 
the Lasso estimator of the coefficient $\beta_j$ is the soft-thresholding of 
$N(\beta_j,\sigma^2/n)$ which is again biased toward the origin. 
For Gaussian designs with $\bSigma = \bI_p$ and in an average sense, 
the Lasso is approximately the soft-thresholding of 
$N(\beta_j,\tau_*^2/n)$ with certain $\tau_*\ge \sigma$ under proper conditions 
\cite{bayati2012lasso}. 
Thus, with $s_1 = \#\{j:|\beta_j|>\lam\}$, 
the squared bias of the Lasso, $\|\bbeta - \E[\hbbeta]\|_2^2$,  
is expected to have no smaller order than the lower bound $s_1\lam^2$ for its $\ell_2$ risk 
\cite[Theorem 3.1]{bellec2018noise}. 
Alternative approaches were proposed to remove or reduce the bias of the Lasso for strong signals,
e.g.,
by using concave penalty functions (e.g., SCAD \cite{fan2001variable},
MCP \cite{zhang2010nearly}) 
or 
iterated hard thresholding algorithms
\cite{blumensath2009iterative}.
These approaches yield an error term
of the order $(\|\bbeta\|_0-s_1')\lambda^2 + s_1'\sigma^2/n$ where $s_1' = \{j=1,...,p: |\beta_j|>c \lambda\}$
for some constant $c>0$ \cite{feng2019sorted,ndaoud2020scaled}, 
alleviating the bias of the Lasso for large coefficients at typical penalty levels $\lam >\sigma/n^{1/2}$. 

\paragraph*{De-biasing the Lasso, asymptotic normality and confidence intervals}
If the goal is the estimation of a single scalar parameter
$\theta=\langle \ba_0,\bbeta\rangle$ in a predetermined direction $\ba_0$ 
instead of the full vector $\bbeta\in\R^p$, 
it is possible to correct
the bias of the Lasso 
and to construct confidence intervals for $\theta$: 
there is already
a vast literature on asymptotic normality of de-biased estimates
in sparse linear regression for the Lasso 
\cite[among others]{ZhangSteph14,GeerBR14,JavanmardM14a,JavanmardM14b,belloni2014high,javanmard2018debiasing,miolane2018distribution,bellec_zhang2019dof_lasso}. 
In this literature $\ba_0$ is usually the $j$-th canonical basis vector
and $\beta_j$ the scalar parameter of interest.
Given the Lasso $\hbbeta$ as an initial estimator of $\bbeta$,
the idea is to add a de-biasing term to achieve asymptotic normality  
which then yields confidence intervals 
for $\theta=\langle \ba_0,\bbeta\rangle$.
If $s_0 = \|\bbeta\|_0$ in 
\eqref{LM}, 
several de-biased estimators 
have been proposed and their asymptotic normality 
hold under certain rate conditions on $s_0,n,p$.
The earliest works on this topic 
\cite{ZhangSteph14,GeerBR14,JavanmardM14a,belloni2014high}
provide asymptotic normality results
in the regime $s_0 \log(p) / \sqrt n \to 0$. 
When $s_0 \log(p) / \sqrt n \to 0$ indeed holds, the  
de-biasing constructions in these papers are all first order equivalent to each other, 
and under normalization $\|\bSigma^{-1/2}\ba_0\|_2=1$ to 
\bel{eq:debiased-estiamte-lasso-z_0}
\qquad && \htheta =
\underbrace{\langle \ba_0,\hbbeta\rangle}_{
    \text{initial estimate}}
+  \underbrace{
\|\bz_0\|_2^{-2} \bz_0^\top(\by-\bX\hbbeta)}_{\text{de-biasing correction}},
\ \ 
\sqrt n (\htheta - \theta) =
\underbrace{\sqrt n \|\bz_0\|_2^{-2} \bz_0^\top \bep}_{\text{normal part}}
+ \underbrace{O_\P( R_n)}_{\text{remainder}}, 
\eel
where
$\bu_0 = \bSigma^{-1}\ba_0/\big\langle\ba_0,\bSigma^{-1}\ba_0\big\rangle$ 
and $\bz_0 = \bX\bu_0 \sim N(\mathbf{0},\bI_n)$. 
While these works do not assume $\bSigma$ known and construct an estimated score
vector $\hbz$ for 
$\bz_0$, the impact of using $\hbz$ can be absorbed into the remainder in 
\eqref{eq:debiased-estiamte-lasso-z_0} with $R_n = \sigma s_0\log(p) / \sqrt n$.
The direction $\bu_0$ and the de-biasing correction 
in \eqref{eq:debiased-estiamte-lasso-z_0}
have a natural semi-parametric interpretation
\cite{zhang2011statistical}.
Viewing $\theta:\R^p\to\R$ as the function $\theta(\bbeta)=\langle \ba_0,\bbeta\rangle$,
the Fischer information 
for the estimation of $\theta(\bbeta)$
in \eqref{LM}
is $F_\theta = 1/(\sigma^2 \langle \ba_0,\bSigma^{-1}\ba_0\rangle)$,
and the direction $\bu_0$ above is the only $\bu\in\R^p$ with 
\begin{equation}
\langle \nabla \theta(\hbbeta), \bu\rangle
= \langle \ba_0,\bu\rangle 
= 1
\label{eq:normalization-u-intro}
\end{equation}
such that $F_\theta$ is also the Fischer information in the one-dimensional
submodel $\{\hbbeta + t \bu, t\in \R\}$.
For this reason the line $\{\hbbeta + {t\bu_0,} t\in \R\}$ is referred to
as the least-favorable one-dimensional submodel for the estimation of $\theta$.
The normalization \eqref{eq:normalization-u-intro} ensures that 
$\theta(\hbbeta+t\bu) = \theta(\hbbeta) + t$ 
and $\widehat{\theta} = \theta(\hbbeta+\widehat{t} \bu_0)$
with $\widehat{t}=\|\bz_0\|_{2}^{-2}\bz_0^\top(\by-\bX\hbbeta)$, 
so that \eqref{eq:debiased-estiamte-lasso-z_0} replaces 
the initial $\hbbeta$ 
with its one-step correction 
$\hbbeta + \widehat{t} \bu_0$, where
$\widehat{t}$ maximizes the likelihood in the least-favorable submodel. 
We refer to \cite{BickelKRW98} for a systematic study of
this semi-parametric perspective.

If $s_0\log(p)/\sqrt n \to +\infty$
and $\bSigma$ is unknown with bounded spectrum,
the minimax estimation error of the form
$\sqrt{n}(\htheta - \theta)$
diverges for any estimator $\htheta$ \cite{cai2017confidence}.
This rules out asymptotic normally results
at the $\sqrt n$ adjusted rate if $s_0\log(p)/\sqrt n\to +\infty$
and no further assumption is made on $\bSigma$.
However, if $\bz_0$ is known, 
\eqref{eq:debiased-estiamte-lasso-z_0} holds with $R_n' = \sqrt{s_0\log(p/s_0)/n}(1+s_0/\sqrt n)$, 
providing asymptotic normality for sparsity levels
$s_0 \lesssim n^{2/3}$ up to logarithmic factors, 
cf. \cite[Corollary 3.3]{bellec_zhang2019dof_lasso}.
Similarly,
\cite[Theorem 3.8]{javanmard2018debiasing} provides 
\eqref{eq:debiased-estiamte-lasso-z_0} 
with $\ba_0=\be_j\in\R^p$ a canonical basis vector
and $R_n' = \log(p) \sqrt{s_0/n}
{\max_j\|\bSigma^{-1} \be_j\|_1}$.
Already in the regime $\sqrt n \lll s_0 \lll n^{2/3}$,
the arguments of \cite{javanmard2018debiasing,bellec_zhang2019dof_lasso}
differ significantly from
the $\ell_1$-$\ell_\infty$ H\"older inequality argument of
\cite{ZhangSteph14,GeerBR14,JavanmardM14a,belloni2014high}:
while these earlier works prove asymptotic normality with a remainder term
of order $O_\P(s_0\log(p)/\sqrt n)$,
\cite{javanmard2018debiasing,bellec_zhang2019dof_lasso}
analyze explicitly the smaller order terms hidden
in this $O_\P(s_0\log(p)/\sqrt n)$ remainder.

For $s_0\ggg n^{2/3}$, the de-biasing correction 
in \eqref{eq:debiased-estiamte-lasso-z_0} 
needs to be modified: 
\begin{align}
\label{eq:debiased-estiamte-lasso}
\htheta
= \langle \ba_0,\hbbeta\rangle + 
(n-|\Shat|)^{-1} \bz_0^\top(\by-\bX\hbbeta),
\ \ 
{\sqrt{n}(\htheta - \theta)
= \sqrt n \|\bz_0\|_2^{-2} \bz_0^\top \bep + O_\P(R_n')}
\end{align}
with $\Shat=\{j\in[p]:\hbeta_j\ne 0\}$ and
${R_n'} =\sigma (s_0 \log(p/s_0)/n)^{1/2}$,
cf. \cite[Theorem 3.1]{bellec_zhang2019dof_lasso}. 
For $\|\bSigma^{-1/2}\ba_0\|=1$
the difference from \eqref{eq:debiased-estiamte-lasso-z_0}
is the replacement of
$\|\bz_0\|_2^{-2}\approx n^{-1}$ in the de-biasing correction with
$(n-|\Shat|)^{-1}$  
to amplify it by a factor $(1-|\Shat|/n)^{-1}$.
This modification is required as soon as $s_0\ggg n^{2/3}$ 
up to logarithmic factors 
\cite[Section 3]{bellec_zhang2019dof_lasso}.
These asymptotic results for $s_0\ggg \sqrt n$ are amenable
to the lack of knowledge of $\bSigma$: in this case
estimation of $\bz_0$ is possible when $\bSigma^{-1}\ba_0$
is sufficiently sparse, see \cite{javanmard2018debiasing}
if the direction of interest $\ba_0$ is canonical basis vector
and \cite[Section 2.2]{bellec_zhang2019dof_lasso}
for arbitrary direction $\ba_0$.
These results \cite{javanmard2018debiasing,bellec_zhang2019dof_lasso} for 
$s_0\ggg \sqrt n$ and correlated $\bSigma$ are so far restricted to
random Gaussian designs.


\paragraph*{Inflated asymptotic variance for non-vanishing prediction error}
In the results discussed so far for the Lasso, 
$s_0\log(p/s_0)/n\to 0$ or stronger conditions are required for asymptotic normality, 
and 
the asymptotic variance of $\sqrt n (\htheta - \theta)$
is $\sigma^2$.
The condition $s_0\log(p/s_0)/n\to 0$ implies
the consistency of the Lasso in prediction and estimation
thanks to error bounds of the form
$\|\bSigma^{1/2}(\hbbeta-\bbeta)\|_2^2 \lesssim s_0\log(p/s_0)/n$
\cite{ZhangH08,sun2012scaled,bellec2016slope,bellec_zhang2019dof_lasso}.
It turns out that the asymptotic variance of $\sqrt{n}(\htheta - \theta)$
is larger than $\sigma^2$ 
if $\|\bSigma^{1/2}(\hbbeta-\bbeta)\|_2^2$ does not vanish;
this is the situation studied in the present work.
The literature on asymptotic normality of de-biased estimates
in the regime 
\begin{equation}
    p/n\to \gamma\in (0,+\infty),\qquad s_0/n\to \kappa\in(0,1)
    \label{intermediate-regime}
\end{equation}
for constants $\gamma,\kappa>0$ is more scarce.
In this regime where $p,n$ and $s_0$ are all of the same order, \cite{JavanmardM14b,miolane2018distribution}
provide asymptotic normality results 
for the de-biased Lasso \eqref{eq:debiased-estiamte-lasso} 
in the estimation of $\beta_j$ (canonical $\ba_0=\be_j$) 
in the isotropic Gaussian design. In these works,
the asymptotic variance of $\sqrt n (\htheta - \theta)$
equals a constant $\tau_*^2$ 
satisfying the 
system of two nonlinear 
equations in \cite{bayati2012lasso} and \cite[Proposition 3.1,Theorem 3.1]{miolane2018distribution}.
The constant $\tau_*^2$ is related to the residual sum of squares
\cite[Corollary 4.1]{miolane2018distribution}
and out-of-sample error 
\cite[Theorem 3.2]{miolane2018distribution} as in 
$$
(1-|\Shat|/n)^{-2}\|\by-\bX\hbbeta\|_2^2/n \to^\P \tau_*^2,
\qquad
\sigma^2 + \|\bSigma^{1/2}(\hbbeta-\bbeta)\|_2^2 \to^\P \tau_*^2
$$
where $\to^\P$ denotes convergence in probability. 
These results for $\bSigma=\bI_p$ highlight that
the  asymptotic variance is strictly larger than $\sigma^2$
when $p, n$ are of the same order as in \eqref{intermediate-regime}.
This phenomenon in the regime \eqref{intermediate-regime} is generic:
for instance the asymptotic variance is also larger than $\sigma^2$
for all permutation-invariant penalty functions
\cite[Proposition 4.3]{celentano2019fundamental}.

In this regime where $n$ and $p$ are of the same order,
\cite{el_karoui2013robust,donoho2016high}
proved asymptotic normality and characterized the variance
for unregularized $M$-estimators.
For $M$-estimators a de-biasing correction is unnecessary
due to the absence of regularization,
and a rotational invariance argument reduces
the problem of correlated designs to a corresponding uncorrelated one
\cite[Lemma 1]{el_karoui2013robust}.
However, this rotational invariance is lost in the presence of a penalty such as the $\ell_1$-norm.
New techniques are called for to analyse the asymptotic behavior, in the regime \eqref{intermediate-regime}
and under correlated designs, of estimators that are not rotational invariant.
More recently, the Approximate Message Passing techniques used 
in \cite{JavanmardM14b,donoho2016high}
were used to obtain similar results in logistic regression
\cite{sur2018modern}; but again, these techniques cannot handle
the Lasso penalty for correlated design.
A more detailed comparison with these works is made 
in \Cref{sec:relaxing-strong-convex-lasso-GL}.
To our knowledge, there is no previous asymptotic normality result
for de-biased
estimates in the regime \eqref{intermediate-regime} for correlated designs
in the presence of a penalty not depending on $\bSigma$ 
(i.e., in situations where rotational invariance does not hold). 
A main goal of the paper is to fill this gap.
Available techniques that tackle the regime \eqref{intermediate-regime}
assume, in addition to uncorrelated design, that the penalty is invariant
under permutations of the $p$ coefficients
\cite{bayati2012lasso,miolane2018distribution,celentano2019fundamental,bu2019algorithmic}
and that the empirical distribution of the true $\{\sqrt{n}\beta_j, j\le p\}$ converges to
some prior distribution.
A second goal of the present paper is to show that asymptotic normality
of de-biased estimates can be obtained beyond the Lasso and beyond
permutation-invariant penalty functions,
without imposing the convergence of the empirical distribution of the normalized coefficients 
$\{\sqrt n\beta_j, j\le p\}$.

\subsection{A general construction of de-biased estimators}
\label{sec:setup-de-biasing}
This section describes a general approach  to systematically construct de-biased
estimates in the linear model \eqref{LM} where $\bX$ has iid $N({\bf 0},
\bSigma)$ rows.
Our goal is to construct confidence intervals for the one-dimensional parameter $\theta=\langle \ba_0, \bbeta\rangle$.
Consider an initial estimator $\hbbeta$,
viewed as a function
of $(\by,\bX)$, i.e.,
$\hbbeta:\R^{n\times (1+p)}\to\R^p$
and assume that this function $\hbbeta$ is Fr\'echet\footnote{
Although the Fr\'echet derivative is the usual definition of derivative
in finite dimension, we write Fr\'echet to emphasize that
the derivative is linear. Linearity may fail for weaker 
notions such as Gateaux differentiability.}
differentiable.
For a given observed data $(\by,\bX)$ from the linear model \eqref{LM}
and a $\hbbeta$ Fr\'echet differentiable at $(\by,\bX)$,
there exist uniquely matrices $\hbH\in\R^{n\times n}$ and $\hbG\in\R^{n\times p}$ such that
\begin{align}
\bX\hbbeta(\by+\bfeta,\bX)
-
\bX\hbbeta(\by,\bX)
&=
\hbH{}^\top\bfeta
+ o(\|\bfeta\|),
\label{eq:def-H-intro}
\\
\hbbeta(\by,\bX+\bfeta \ba_0^\top)
-
\hbbeta(\by,\bX)
&=
\hbG{}^\top\bfeta
+ o(\|\bfeta\|)
\nonumber
\end{align}
for all $\bfeta\in\R^n$.
With $\bX= (x_{ij})_{i\in[n],j\in[p]}$,
if the partial derivatives of $\hbbeta(\by,\bX)$
at the observed data $(\by,\bX)$
are $(\partial/\partial x_{ij})\hbbeta(\by,\bX)$
and $(\partial/\partial y_i)\hbbeta(\by,\bX)$
then \eqref{eq:def-H-intro} implies
$\hbH{}^\top \be_i = \bX(\partial/\partial y_i) \hbbeta(\by,\bX)$
and $\hbG{}^\top \be_i = \sum_{j=1}^p \langle \ba_0,\be_j\rangle
(\partial/\partial x_{ij}) \hbbeta(\by,\bX)$
for canonical basis vectors $\be_i\in\R^n$ and $\be_j\in\R^p$. 
The derivatives of $\smash{\hbbeta}$
and the matrices $\smash{\hbH}$ and $\smash{\hbG}$ can be computed
by only looking at the observed data $(\by,\bX)$, 
for instance by finite difference schemes.

Next, consider the function $\bphi$ defined as
$$\bphi:\R^{n\times (1+p)}\to\R^n,\qquad
(\by,\bX)\mapsto \bphi(\by,\bX) = \bX\hbbeta(\by,\bX) - \by
.$$
If $\hbbeta$ is differentiable
at $(\by,\bX)$ then $\bphi$ is differentiable as well.
By the product and 
chain rules
\begin{equation}
\begin{split}
\bphi(\by +\bfeta,\bX)
-
\bphi(\by,\bX)
&=
[\hbH - \bI_n]^\top
\bfeta + o(\|\bfeta\|),
\\
\bphi(\by,\bX+\bfeta\ba_0^\top)
-
\bphi(\by,\bX)
&=
[
\langle \ba_0, \hbbeta\rangle \bI_n
+  \hbG{} \bX^\top
]^\top
\bfeta + o(\|\bfeta\|).
\end{split}
\label{eq:gradient-phi-intro}
\end{equation}
If the partial derivatives of $\bphi$
are $(\partial/\partial y_i)\bphi$ and
$(\partial/\partial x_{ij})\bphi$, the second line of
the previous display
is equivalently rewritten as
$$\hbox{$\sum$}_{j=1}^p \langle \ba_0,\be_j\rangle
(\partial/\partial x_{ij})\bphi(\by,\bX)
=[\langle \ba_0,\hbbeta\rangle\bI_n + \hbG{}\bX^\top
]^\top\be_i
$$ for each canonical basis vector $\bfeta = \be_i  \in\R^n$.

Observe that the arguments $(\by,\bX)$ of $\bphi$
are centered and jointly normal random variables
and their correlations are computed explicitly, e.g.,
$\E[x_{ij}y_l] = \be_j^\top\bSigma\bbeta I_{\{i=\ell\}}$, 
with basis vectors $\be_j\in\R^p$. 
One version of Stein's formula, also known as Gaussian integration by parts,
is $\E[G  h(Z_1,...,Z_q)] = \sum_{k=1}^q \E[GZ_k] \E[(\partial/\partial z_k) h(Z_1,...,Z_q)]$ provided that the function $h(z_1,...,z_q)$ is differentiable
and that $G,Z_1,...,Z_q$ are centered jointly normal random variables, 
provided the 
existence of the expectations
\cite[Appendix A.4]{talagrand2010mean}.
We leverage this version of Stein's formula
to obtain an unbiased estimating equation involving
only one unknown parameter, the scalar
$\theta=\langle \ba_0,\bbeta\rangle$ of interest.
For $G_i = \be_i^\top\bX\bSigma^{-1}\ba_0$ we find
$\E[G_i y_k] = \E[G_i x_{kj}] = 0$ if $i\ne k$ while
$\E[G_i x_{ij}] = \langle \ba_0, \be_j\rangle$
and
$\E[G_i y_i] = \langle \ba_0, \bbeta\rangle$
so that by reading the partial derivatives in \eqref{eq:gradient-phi-intro},
\bel{eq:stein-intro}
\E\Bigl[G_i\phi_i(\by,\bX)\Bigr]
&=&
\langle \ba_0,\bbeta\rangle 
\E\Bigl[\frac{\partial\phi_i}{\partial y_{i}}(\by,\bX)\Bigr]
+
\sum_{j=1}^p \langle \ba_0, \be_j\rangle
\E\Bigl[\frac{\partial\phi_i}{\partial x_{ij}}(\by,\bX)\Bigr]
\cr&=&
\E\Bigl[
    \langle \ba_0,\bbeta\rangle(\hbH{}_{ii} - 1)
\Bigr]
+
\E\Bigl[
    \langle \ba_0,\hbbeta \rangle
    +
    \be_i^\top\hbG\bX^\top\be_{i}
\Bigr].
\eel
Summing over $i=1,...,n$ and
using $\bphi(\by,\bX) = \bX\hbbeta-\by$, we find that 
$$
\E\bigl[
    \langle \bX\bSigma^{-1}\ba_0,\bX\hbbeta-\by\rangle
\bigr]
= \E\bigl[
    -\langle \ba_0,\bbeta\rangle\trace[\bI_n-\hbH]
    +
    \langle \ba_0,\hbbeta \rangle n
    +
    \trace[\bX^\top\hbG]
\bigr].
$$
To transform this equation into a form representative
of the results of the paper, define
the scalars $\df$ and ${\widehat{A}}$
by
\begin{equation}
\df = \trace [ \hbH ],
\quad
{\widehat{A}} = 
\trace[\bX^\top\hbG] 
+
\langle \ba_0, \hbbeta\rangle \df
.
\label{def-hat-df-and-B}
\end{equation}
The notation $\df$ underlines that $\trace[\hbH]$ has the interpretation
of degrees-of-freedom of the estimator $\hbbeta$ in 
Stein's Unbiased Risk Estimate (SURE) \cite{stein1981estimation}:
regarding $\hbmu=\bX\hbbeta$ as an estimate of $\bmu=\bX\bbeta$ in the 
Gaussian sequence model with observation $\by=\bmu+\bep$, the quantity
$\widehat{\textsc{sure}}=\|\by-\hbmu\|^2+2\sigma^2\df - \sigma^2n$ is an unbiased
estimate of the in-sample error $\|\hbmu-\bmu\|^2 =\|\bX(\hbbeta-\bbeta)\|^2$.
With this notation, we obtain the \emph{unbiased estimating equation}
\begin{equation}
0 = \E\bigl[
\langle \bX\bSigma^{-1}\ba_0, \by - \bX\hbbeta\rangle 
+ (n - \df)(\langle \ba_0, \hbbeta \rangle - \theta)
+ {\widehat{A}}
\bigr]
\label{eq:unbiased-estimating-equation}
\end{equation}
where the only unobserved quantity inside the expectation
is $\theta=\langle \ba_0,\bbeta\rangle$, the scalar parameter
we wish to estimate.
In the above application
of Stein's formula, $G_i=\be_i^\top\bX\bSigma^{-1}\ba_0$ was chosen
on purpose so that $\bbeta$ appears in \eqref{eq:stein-intro}
only through $\langle \ba_0,\bbeta\rangle$
thanks to $\E[G_iy_i] = \langle \ba_0,\bbeta\rangle$.
Note that 
replacing $G_i$ in \eqref{eq:stein-intro} 
by $\be_i^\top\bX\bu$ for any $\bu\in\R^p$ not proportional
to $\bSigma^{-1}\ba_0$ brings a scalar projection of $\bbeta$
different from $\langle \ba_0,\bbeta\rangle$:
this shows the unique role of the random vector
$\bX\bSigma^{-1}\ba_0$ to derive an unbiased estimating equation
for $\theta=\langle \ba_0,\bbeta\rangle$. It is notable that
the direction $\bSigma^{-1}\ba_0$ coincides
with the least-favorable
direction described around \eqref{eq:normalization-u-intro}.
Equation \eqref{eq:unbiased-estimating-equation} is obtained for an arbitrary
initial estimator $\hbbeta$ provided that its derivatives with respect to $(\by,\bX)$ exist 
and the integrability conditions hold to ensure
existence of the expectations involved. 
From \eqref{eq:unbiased-estimating-equation},
the method of moments suggests to estimate $\theta$ with 
$\htheta =
\langle \ba_0, \hbbeta\rangle
+ (n-\df)^{-1}\bigl(\langle \bX\bSigma^{-1}\ba_0, \by - \bX\hbbeta\rangle + {\widehat{A}}\bigr)$
which resembles \eqref{eq:debiased-estiamte-lasso}
for the Lasso for $\df=|\Shat|$ and $\bX\bSigma^{-1}\ba_0 = \bz_0$
under the normalization $\langle \ba_0,\bSigma^{-1}\ba_0\rangle = 1$.

It is useful at this point to specialize the above derivation
to an estimator for which all derivatives can be computed explicitly.
For Ridge regression with penalty $g(\bb)=\mu\|\bb\|_2^2$
for some $\mu>0$, $\hbbeta(\by,\bX) = (\bX^\top\bX + n\mu\bI_p)^{-1}\bX^\top\by$
and
\begin{align}
    \label{eq:ridge-calculation-intro}
    \hbH{}^\top &= \bX(\bX^\top\bX + n\mu\bI_p)^{-1}\bX^\top,
\cr
    \hbG{}^\top &= (\bX^\top\bX + n\mu\bI_p)^{-1}\bigl[\ba_0(\by-\bX\hbbeta)^\top-
\bX^\top \langle \ba_0,\hbbeta\rangle
\bigr].
\end{align}
Indeed, the derivatives of $\hbbeta(\by,\bX)$ exist
as it is the composition of elementary differentiable functions.
Differentiation with respect to $\by$ is straightforward
as $\hbbeta$ is linear in $\by$, while
in order to compute $\hbG$
we proceed by setting
$\bb(t)=\hbbeta(\by,\bX(t))$ with $\bX(t) = \bX + t\bfeta
\ba_0^\top$. Differentiation of the KKT
conditions $\bX(t)^\top(\by - \bX(t)\bb(t)) = n \mu \bb(t)$
at $t=0$ provides the directional derivative $(d/dt)\bb(t)|_{t=0}=\hbG{}^\top\bfeta$. 
This gives \eqref{eq:ridge-calculation-intro}.
It follows from \eqref{eq:ridge-calculation-intro} 
that $\df = \trace[\bX(\bX^\top\bX + n\mu\bI_p)^{-1}\bX^\top]$
and $\widehat{A} = (\by-\bX\hbbeta)^\top \bX (\bX^\top\bX + n\mu\bI_p)^{-1}\ba_0$
for the quantities in \eqref{def-hat-df-and-B}
(for $\widehat{A}$, note the fortuitous cancellation of the term
$\langle \ba_0,\hbbeta\rangle\df$).
For the Lasso similar differentiability formulae are derived
in \cite{bellec_zhang2019dof_lasso}. It is however,
unclear how to obtain closed form formulae for the derivatives of $\hbbeta$
for an arbitrary convex penalty $g$ in \eqref{eq:intro-hbbeta-g}.

We now set up some notation that will be useful for the rest
of the paper, and derive again the unbiased estimating
equation  \eqref{eq:unbiased-estimating-equation} using this new
notation.
Define
\bel{u_0}
\qquad\bu_0 = \bSigma^{-1}\ba_0/\big\langle\ba_0,\bSigma^{-1}\ba_0\big\rangle,
\quad
\bz_0 = \bX\bu_0,
\quad
\bQ_0 = \bI_{p\times p} - \bu_0\ba_0^\top
.
\eel
The normalizing constant in $\bu_0$ is such that $\langle \ba_0,\bu_0\rangle=1$
holds so that
the expression \eqref{u_0} for $\bu_0$
coincides with the direction of the least-favorable
submodel discussed around \eqref{eq:normalization-u-intro}.
The vector $\bz_0$ is independent of $\bX \bQ_0$
by construction as $(\bz_0,\bX\bQ_0)$ are jointly normal
and uncorrelated. 
This follows  by noting that
$\bX\bSigma^{-1/2}$ has iid $N(0,1)$ entries and
$$
\bz_0 = \bX\bSigma^{-1/2}\bv /\|\bSigma^{-1/2}\ba_0\|,
\quad
\bX\bQ_0 = \bX\bSigma^{-1/2}(\bI_p-\bv\bv^\top)\bSigma^{1/2}
$$
for the unit vector $\bv = \bSigma^{-1/2}\ba_0/\|\bSigma^{-1/2}\ba_0\|$
as by construction of $\bQ_0$ matrix
$\bI_p-\bv\bv^\top=\bSigma^{1/2}\bQ_0\bSigma^{-1/2}$
is the orthogonal projection onto $\{\bv\}^\perp$.
We summarize this as
\begin{equation}
    \label{eq:ref-X-Q0}
    \bX = \bX\bQ_0 + \bz_0 \ba_0^\top
    \ \text{ with }\ 
    \bz_0\sim N(\mathbf{0},\|\bSigma^{-1/2}\ba_0\|^{-2}\bI_n)
    \text{ independent of }
    \bX\bQ_0 
    .
\end{equation}
For brevity, we assume in the sequel and without loss of generality
that the direction of interest $\ba_0$ is normalized such that 
\begin{equation}
    \label{normalization-a_0}
    \|\bSigma^{-1/2} \ba_0\|^2 = \langle \ba_0, \bSigma^{-1} \ba_0 \rangle = 1.
\end{equation}
By definition of $\bu_0$ and $\bz_0$,
the normalization \eqref{normalization-a_0}
gives $\bz_0 \sim N({\bf 0},\bI_n)$. 

Conditionally on $(\bX\bQ_0,\bep)$,
define the function
$f_{(\bX\bQ_0,\bep)}:\R^n\to\R^n$ by
\begin{equation}
    \label{definition-f-intro}
    f_{(\bX\bQ_0,\bep)}(\bz_0) = \bX\hbbeta-\by.
\end{equation}
By \eqref{eq:ref-X-Q0} and the independence of $\bep$ and $\bX$, the 
conditional expectation
given $(\bX\bQ_0,\bep)$ can be written as integrals
against the Gaussian measure of $\bz_0$, e.g.,
$$
\E\Bigl[\bz_0^\top f_{(\bX\bQ_0,\bep)}(\bz_0) ~\Big|~ (\bX\bQ_0,\bep)\Bigr]
=
\int\Bigl(\bz^\top f_{(\bX\bQ_0,\bep)}(\bz) \Bigr)
e^{-\|\bz\|_2^2/2}(\sqrt{2\pi})^{-n} d\bz
$$
since $\bz_0\sim N(\mathbf{0},\bI_n)$. 
As we argue conditionally on $(\bX\bQ_0,\bep)$,
we omit the dependence on  $(\bX\bQ_0,\bep)$
and write simply $f:\R^n\to\R^n$.
Since $\by=\bep + \bX\hbbeta$ and
$\bX=\bX\bQ_0 + \bz_0\ba_0^\top$,
$$f(\bz_0) = \bX\hbbeta\bigl(
    \bep + \bX\bQ_0\bbeta + \bz_0 \ba_0^\top\bbeta,
    \bX\bQ_0 + \bz_0\ba_0^\top
\bigr)
-\bX\bbeta - \bep.$$
The gradient $\nabla f$ with respect to $\bz_0$,
holding $(\bX\bQ_0,\bep)$ fixed, can be computed by the product rule and the chain rule 
via \eqref{eq:gradient-phi-intro}:
\begin{equation}
\nabla f(\bz_0)^\top
= \bI_n \langle \ba_0, \hbbeta-\bbeta\rangle
+
\bigl[
    \langle \ba_0,\bbeta\rangle
    \hbH{}^\top
+
\bX\hbG{}^\top
\bigr].
    \label{nabla-f-introduction-chain-rule}
\end{equation}
We adopt the usual convention that the gradient of a vector valued
function is the transpose of its Jacobian.  
Computing the directional derivative
of $f$ in a direction $\bfeta$ requires considering difference
of an expression at $(\bep,\bX\bQ_0,\bz_0 + t\bfeta)$
minus the same expression at $(\bep,\bX\bQ_0,\bz_0)$,
dividing by $t$ and taking the limit as $t\to0$;
this is equivalent to considering the difference
of an expression at $(\bep,\bX(t))$ with $\bX(t)=\bX+t\bfeta \ba_0^\top$
minus the same expression at $(\bep,\bX)$,
dividing by $t$ and taking the limit as $t\to0$.

Taking the trace of \eqref{nabla-f-introduction-chain-rule} and by definition
of $\df$ and $\widehat{A}$ in \eqref{def-hat-df-and-B}, the identity
\bel{intro-decomposition}
\nonumber
-{\xi_0} 
   & \defas &
\dv f(\bz_0) - \bz_0^\top f(\bz_0)
\\ &=&
   (n - \df)(\langle \ba_0, \hbbeta \rangle - \theta)
+ \langle \bz_0, \by - \bX\hbbeta\rangle + {\widehat{A}}
\eel
holds
where $\dv f(\bz_0) = \trace[\nabla f(\bz_0)]$.
Since $\E[\dv f(\bz_0) - \bz_0^\top f(\bz_0) | (\bX\bQ_0,\bep)]=0$
by Stein's formula \cite{stein1981estimation},
this provides the unbiased estimating equation
\eqref{eq:unbiased-estimating-equation}.
Reasoning conditionally on $(\bep,\bX\bQ_0)$,
using Stein formulae with respect to $\bz_0$
involving conditional expectations given $(\bep,\bX\bQ_0)$
and gradients of the form $\nabla f(\bz_0)$ 
holding $(\bep,\bX\bQ_0)$ fixed will be a recurring
theme throughout the paper.
In this context, the function $f$ itself depends
on $(\bep,\bX\bQ_0)$ as in \eqref{definition-f-intro},
although the dependence on $(\bep,\bX\bQ_0)$ is omitted 
for brevity.

In order to construct confidence intervals using the unbiased estimating
equation \eqref{eq:unbiased-estimating-equation},
one may hope that
the quantity \eqref{intro-decomposition} above is well behaved---ideally,
approximately normal with mean zero and a 
variance that can be consistently estimated from the observed data.
By the Second Order Stein's formula in
\Cref{prop:2nd-order-stein-Bellec-Zhang} below, 
which was already known to \citet[(8.6)]{stein1981estimation}
in a different form, 
the conditional variance of \eqref{intro-decomposition} given $(\bep,\bX\bQ_0)$
is 
\bel{variance-f-second-order-stein-V-star-intro}
\Var_0[\xi_0] &=&
\E_0\left[ \|f(\bz_0)\|^2
+  \trace[\{ \nabla f(\bz_0)\}^2 ]
\right]
\\&=&
\E_0\left[ V^*(\theta)]
\right] 
\quad
\text{ for } 
\ 
V^*(\theta)
=
\|\by-\bX\hbbeta\|^2
+  \trace[\{ \nabla f(\bz_0)\}^2] 
\nonumber
\eel
where $\E_0 = \E[ \cdot | \bep,\bX\bQ_0]$ denotes the conditional expectation
with respect to $\bz_0$ given $(\bep,\bX\bQ_0)$
and $\Var_0$ denotes the conditional variance given $(\bep,\bX\bQ_0)$. 
The gradient $\nabla f(\bz_0)$ in \eqref{nabla-f-introduction-chain-rule}
and the unbiased estimate $V^*(\theta)$ of $\Var_0[\xi_0]$ 
only depend on the unknown parameter of interest $\theta$
and observable quantities, and $V^*(\theta)$ is quadratic in $\theta$. 

Assume now we are in an ideal situation in the sense that both conditions below
are satisfied:
(i)
The quantity \eqref{intro-decomposition} is approximately 
normally distributed conditionally on $(\bep,\bX\bQ_0)$, and
(ii)
$V^*(\theta)$ is a consistent estimator of
\eqref{variance-f-second-order-stein-V-star-intro}, the conditional
variance of the random variable \eqref{intro-decomposition}.
Then the set of $\theta$ for which the inequality 
\begin{equation}
    \label{quadratic-polynomial-equation-confidence-interval}
\big[
    (n-\df)\big(\langle \ba_0, \hbbeta \rangle - \theta\big)
    + \langle \bz_0, \by - \bX\hbbeta\rangle + {\widehat{A}}
\big]^2
    -
    V^*(\theta) 
    z_{\alpha/2}^2
     \le 0
\end{equation} 
is satisfied is an $(1-\alpha)$-confidence interval, 
where $\P(|N(0,1)|>z_{\alpha/2})=1-\alpha$. 
Solving the corresponding quadratic equality gives up to two solutions 
$\Theta_1(z_{\alpha/2}) \le \Theta_2(z_{\alpha/2})$
that are such that \eqref{quadratic-polynomial-equation-confidence-interval}
holds with equality. These two solutions implicitly depend on the observables 
$$\langle \by -\bX\hbbeta, \bz_0\rangle,\quad 
\|\by - \bX\hbbeta\|^2,\quad
\df,\quad 
{\widehat{A}},\quad
\ba_0^\top\hbbeta$$
and the derivatives of $\hbbeta$. 
If the coefficient of $\theta^2$ in the left hand side of
    \eqref{quadratic-polynomial-equation-confidence-interval}
    is positive,
    (i.e., if the leading coefficient of
    \eqref{quadratic-polynomial-equation-confidence-interval},
    seen as a polynomial in $\theta$ with data-driven coefficients,
    is positive), 
a $(1-\alpha)$ confidence interval for $\theta=\ba_0^\top\bbeta$ is then given by
\bel{CI-introduction}
\widehat{CI} = \big[\Theta_1(z_{\alpha/2}),\Theta_2(z_{\alpha/2})\big].
\eel
We will show in the discussion surrounding \eqref{eq:event-CI-sec3}
below that the dominant coefficient is positive
and that the confidence interval is indeed of the above form 
if $\hbbeta$ is a convex penalized estimator.
Although a variant of the above construction
was briefly presented in \cite[Section 6]{bellec_zhang2018second_order_stein}
(there, the function $\bz_0\to\bX\bQ_0(\hbbeta-\bbeta) -\bep$ is used),
important questions remain unanswered
to prove the validity of the general confidence interval in 
\eqref{CI-introduction} and its applicability to commonly used regularized estimators. 



\subsection{The rest of the paper is organized as follows} 
\Cref{sec:CLT} develops an $L_2$ bound between $\xi/\Var[\xi]^{1/2}$ and $N(0,1)$ 
for random variables of the form ${\xi} = \bz^\top f(\bz) - \dv f(\bz)$
where $\bz\sim N({\bf 0}, \bI_n)$.
\Cref{sec:application-de-biasing} uses this normal approximation 
to show the asymptotic normality of \eqref{intro-decomposition} 
and proves the consistency of the variance estimate $V^*(\theta)$ in 
\eqref{variance-f-second-order-stein-V-star-intro} 
in the regime where $p$ and $n$ 
are of the same order in the linear model \eqref{LM} with correlated design.
\Cref{sec:examples} provides closed-form formulas to apply the results 
in \Cref{sec:application-de-biasing}
to the Lasso, the group Lasso and twice continuously differentiable penalty functions.
\Cref{sec:proofs} contains the proofs of the results in
\Cref{sec:application-de-biasing}. 
\Cref{appendix:integrability-edelman}
provides a technical lemma on the integrability of smallest eigenvalue of Wishart matrices,
\Cref{sec:p-larger-n-without-strong-convexity}
provides the proofs of the asymptotic normality results 
for the Lasso and group Lasso when $p>n$, and
\Cref{sec:proof-gradient-GL} contains
the proofs of the derivative
formulae for the group Lasso.

\subsection{Notation}
For two reals $\{a,b\}$, let
$a\wedge b = \min\{a,b\}$, $a\vee b=\max\{a,b\}$
and $a_+= a \vee 0$.
Let $\bI_d$ be the identity matrix of size $d\times d$, e.g. $d=n,p$. 
For any $p\ge 1$, let $[p]$ be the set $\{1,...,p\}$.
Let $\|\cdot\|$ be the Euclidean norm
and $\|\cdot\|_q$ the $\ell_q$ norm of vectors for any $q\ge 1$, so that $\|\cdot\| = \|\cdot\|_2$. 
Let $\|\cdot\|_{op}$ be the operator norm 
of matrices and 
$\|\cdot\|_F$ the Frobenius norm.
Let $\phi_{\min}(\bS)$ be
the smallest eigenvalue of a symmetric matrix $\bS$.
We use the notation $\langle\cdot,\cdot\rangle$ for the canonical scalar product of vectors in $\R^n$ or $\R^p$,
i.e., $\langle \ba,\bb\rangle  = \ba^\top \bb$ for two vectors $\ba,\bb$ of the same dimension.
For any event $\Omega$, denote by $I_\Omega$ its indicator function. 
The unit sphere is $S^{p-1}=\{\bx\in\R^p:\|\bx\|=1\}$.
Convergence in distribution is denoted by $\to^d$
and convergence in probability by $\to^\P$.
Throughout the paper, $C_0, C_1,...$ denote positive absolute constants,
$C_k(\gamma)$ positive constants depending on $\gamma$ only, 
and $C_k(\gamma,\mu)$   
on $\{\gamma,\mu\}$ only. 

For any vector $\bv=(v_1,...,v_p)^\top \in\R^p$ and set $A\subset [p]$,
the vector $\bv_A\in\R^{|A|}$ is the restriction $(v_j)_{j\in A}$.
For any $n\times p$ matrix $\bM$ with columns $(\bM_1,\ldots,\bM_p)$
and any subset $A\subset [p]$, let $\bM_A = (\bM_j, j\in A)$ be 
the matrix composed of columns of $\bM$ indexed by $A$. 
If $\bM$ is a symmetric matrix of size $p\times p$ and $A\subset [p]$,
then $\bM_{A,A}$ denotes the sub-matrix of $\bM$ with rows and columns in $A$,
and $\bM_{A,A}^{-1}$ is the inverse of $\bM_{A,A}$. 
For any square matrix $\bM$, let $\bM^s = (\bM+\bM^\top)/2$ 
be its symmetrization giving the same quadratic form. 

For a vector valued map
$h:\R^n\to \R^q$ with coordinates $h_1,...,h_q:\R^n\to\R$, the gradient 
$\nabla h\in\R^{n\times q}$ is the matrix with columns $\nabla h_1, ....,\nabla h_q$. 
Thus, $\nabla h$ is the transpose of the Jacobian of $h$ and 
$h(\bx+\bfeta) = h(\bx) + \nabla h(\bx)^\top \bfeta + o(\|\bfeta\|)$ if each 
coordinate $h_i$ is Fr\'echet differentiable at $\bx$. 
For deterministic matrices $\bA\in \R^{m\times q}$, 
$\nabla(\bA h) = (\nabla h)\bA^\top \in \R^{n\times m}$.
For $f$ in  \eqref{stein-unbiased-random-variable},
$\nabla f(\bx)\in\R^{n\times n}$ and
the divergence is $\dv f(\bx) = \trace[\nabla f(\bx)]$.  

\section{Normal approximation in Stein's formula} 
\label{sec:CLT}
We develop in this section normal approximations for random variables
of the form 
\bel{stein-unbiased-random-variable}
{\xi} = \bz^\top f(\bz) - \dv f(\bz), 
\eel
for which Stein's formula \cite{stein1981estimation} states $\E[\xi]=0$,  
where $\bz\sim N({\bf 0},\bI_n)$ is standard normal and $f:\R^n\to \R^n$. 
We establish $L_2$ bounds for the linear and quadratic  
approximations of $\xi$ and construct consistent variance estimates in the related CLT. 

Throughout this paper, the $i$-th coordinate $f_i$ of $f$ is a function
$f_i:\R^n\to\R$ and its weak gradient is denoted by $\nabla f_i$.
Similarly, the weak derivative of $g(\bz)$ is denoted by $\nabla g$.  
We refer to \cite[Section 1.5]{bogachev1998gaussian}
for definitions of weak differentiability.
For the application to asymptotic normality of de-biased 
estimates in \Cref{sec:application-de-biasing}, the functions
we will consider are locally Lipschitz. 
By Rademacher's theorem,
locally Lipschitz functions 
are Fr\'echet differentiable almost everywhere,
which is stronger than the existence of directional derivatives
in all directions. In this case
the weak derivatives agree with the classical partial derivatives
almost everywhere. As far as the application in
\Cref{sec:application-de-biasing} is concerned, the reader unfamiliar with
weak differentiability may consider the additional assumption
that $f$ is locally Lipschitz in the following results
and replace weak derivatives with classical derivatives.
The variance of \eqref{stein-unbiased-random-variable} is given
by the following proposition.

\begin{proposition}
    \label{prop:2nd-order-stein-Bellec-Zhang}
    [Second Order Stein formula,
    \cite[Eq. (8.6)]{stein1981estimation}
    \cite{bellec_zhang2018second_order_stein}]
Let $\bz\sim N({\bf 0}, \bI_n)$ and $f:\R^n\to\R^n$ be a function 
with each coordinate $f_i$ being squared integrable
and weakly differentiable with squared integrable
gradient, i.e., $\E[f_i(\bz)^2] + \E[\|\nabla f_i(\bz)\|^2]<+\infty$.
Then
\bel{eq:variance}
\qquad\E[(\bz^\top f(\bz) - \dv f(\bz))^2 ]
=
\E[\|f(\bz)\|^2 ] + \E   \trace[ \{\nabla f(\bz)\}^2 ].
\eel
\end{proposition}

The above result, in the twice differentiable case,
was known to Stein \cite[Eq. (8.6)]{stein1981estimation}.
If $f$ is twice differentiable, the result follows by 
a sequence of integration by parts. The differentiability
requirement was relaxed to only once weakly differentiable $f$
in \cite{bellec_zhang2018second_order_stein} where statistical
applications of this formula to such once differentiable $f$ are discussed. 

\subsection{Linear approximation}\label{subsection:linear-approx}
The goal of the present section is to derive
normal approximations and CLT 
for the random variable \eqref{stein-unbiased-random-variable}.
The intuition is as follow. 
We are looking for linear approximation of the random variable
\eqref{stein-unbiased-random-variable}, of the form $\bz^\top \bmu\sim N(0, \|\bmu\|^2)$
for some deterministic $\bmu\in\R^n$.
We rewrite \eqref{stein-unbiased-random-variable} as
\bel{decomposition-linear-part-remainder}
\bz^\top f(\bz) - \dv f(\bz)
&=& \underbrace{\; \bz^\top \bmu \;}_{\text{linear part}} \quad+\quad 
\underbrace{\bz^\top(f(\bz) - \bmu) - \dv f(\bz)}_{\text{remainder}}.
\eel
The remainder term above is mean-zero with second moment equal to
$\E[ \|f(\bz)-\bmu\|^2] + \E\trace[\{ \nabla f(\bz)\}^2]$
by \Cref{prop:2nd-order-stein-Bellec-Zhang}.
This second moment is minimized for $\bmu= \E[f(\bz)]$, hence 
{$\bz^\top \E[f(\bz)]$ gives} the best linear approximation of $\xi$ 
in \eqref{stein-unbiased-random-variable}. 
The following result provides conditions on $f$ under which the
remainder term {is negligible in \eqref{decomposition-linear-part-remainder}.} 

\begin{theorem}
    \label{thm:L2-distance-from-normal}
Let $\bz\sim N({\bf 0}, \bI_n)$ and $f$ be a function $f:\R^n\to\R^n$,
with each coordinate $f_i$ being squared integrable
and weakly differentiable with squared integrable
gradient, i.e. $\E[f_i(\bz)^2] + \E[\|\nabla f_i(\bz)\|^2]<+\infty$.
    Then ${\xi}= \bz^\top f(\bz) - \dv f(\bz)$ satisfies 
\begin{equation}
        \label{second-order-poincare-eq1} 
        \E\big[\big(\xi/\Var[{\xi}]^{1/2} - Z\big)^2\big] 
        =  {\eps_{1}^2} + \bigl(1- (1-{\eps_{1}^2})^{1/2}\bigr)^2 
        = {\eps_{1}^2} + c_{1}\eps_1^4        
        \end{equation}
        with $Z=\bz^\top \E[f(\bz)] / \|\E[f(\bz)]\|\sim N(0,1)$, 
        deterministic real $1/4 \le c_{1} \le 1$ and  
\bel{identity-eps-n-norm-of-E-f}
\eps_{1}^2 \defas 1 - \frac{\|\E[f(\bz)] \|^2}{\Var[{\xi}]} \le
\epsbar_1^2
\le
\frac
{2\E[\|\nabla f(\bz)\|_F^2]}
{\E[\|f(\bz)\|^2]+\E[\|\nabla f(\bz)\|_F^2]}, 
\eel
where $\epsbar_1^2 \defas
{2\E[\|\{\nabla f(\bz)\}^s\|_F^2]}\big/
\{\|\E[f(\bz)]\|^2+2\E[\|\{\nabla f(\bz)\}^s\|_F^2]\}$. 
Consequently, $\sup_{t\in\R}
    \big|
    \P(\xi/\Var[\xi]^{1/2}\le t)
    - \P(Z\le t)
    \big|
    \le 
    C
    (\eps_1^2+c_1\eps_1^4)^{1/3}
$
for $C=1+(2\pi)^{-1/2}$. 
\end{theorem}

A direct consequence of \Cref{thm:L2-distance-from-normal}
is $\eps_{1}^2 \le \eqref{second-order-poincare-eq1} \le 2\eps_{1}^2 \le 2\epsbar_1^2$.
Inequality \eqref{second-order-poincare-eq1} provides
an upper bound on the 2-Wasserstein 
distance between $\xi/\Var[\xi]^{1/2}$ and $Z\sim N(0,1)$. 
When $\eps_{1}^2\to 0$, it gives a stronger $L_2$ form of the CLT  
$\xi/\Var[\xi]^{1/2}\to^d N(0,1)$ 
in addition to the Kolmogorov distance bound in \Cref{thm:L2-distance-from-normal}. 
The theorem follows from \Cref{prop:2nd-order-stein-Bellec-Zhang}
and an application of 
the Gaussian Poincar\'e inequality. 

\begin{proof}[Proof of \Cref{thm:L2-distance-from-normal}]
    Define $Z=\bz^\top \E[f(\bz)] / \|\E[f(\bz)] \|$ then
    $Z\sim N(0,1)$ and
    \begin{equation*}
    {\xi} - \Var[{\xi}]^{1/2} Z
    =
    \bz^\top g(\bz) - \dv g(\bz)
    \end{equation*}
    where $g(\bz) = f(\bz) - r \E f(\bz)$ and $r=(\Var[\xi]^{1/2}/\|\E f(\bz)\|)$. 
    By Proposition \ref{prop:2nd-order-stein-Bellec-Zhang} applied to $g$
    and a bias-variance decomposition,
    \bes
    && \E[({\xi} - \Var[{\xi}]^{1/2} Z)^2] 
    \cr &=& \E \| f(\bz) - r\E[f(\bz)] \|^2 + \E \trace[ \{\nabla f(\bz)\}^2 ] 
    \cr &=& \E \| f(\bz) - \E[f(\bz)] \|^2 + \E \trace[ \{\nabla f(\bz)\}^2 ] 
    + \{ \Var[{\xi}]^{1/2} - \| \E f(\bz) \|  \}^2
    \cr &=& 
        \Var[\xi] - \|\E f(\bz)\|^2
        + \{ \Var[{\xi}]^{1/2} - \| \E f(\bz) \|  \}^2
    \ees
    thanks to $(r-1)\|\E f(\bz)\| = \Var[\xi]^{1/2}-\|\E f(\bz)\|$.
    Thus, \eqref{second-order-poincare-eq1} follows from 
    the definition of $\eps_{1} ^2$ in \eqref{identity-eps-n-norm-of-E-f}. 
    Moreover, $\E[\| f(\bz) - \E[f(\bz)] \|^2] \le \E[\|\nabla f(\bz)\|_F^2]$ 
    by the Gaussian Poincar\'e inequality and 
    $\|\bM\|_F^2+\trace(\bM^2)=2\|\bM^s\|_F^2$ for $\bM\in\R^{n\times n}$.
    Hence, with $a=\|\E f(\bz)\|^2$, $b=\E \|f(\bz)-\E f(\bz)\|^2$,
    $c = \E\trace[ \{\nabla f(\bz)\}^2]$ and $d = \E[\|\{\nabla f(\bz)\}^s\|_F^2]$ we have
    $$
    \eps_{1} ^2 = \frac{b+c}{a+b+c} \le \frac{2d}{a+2d}
    = \epsbar^2_{1}
        \le
        \frac{2\E[\|\nabla f(\bz)\|_F^2]}
        {a+2\E[\|\nabla f(\bz)\|_F^2]}
        \le
        \frac{2\E[\|\nabla f(\bz)\|_F^2]}
        {\E[\|f(\bz)\|^2 + \|\nabla f(\bz)\|_F^2]}
    $$
    thanks to another Gaussian Poincar\'e inequality for the last inequality. 
    Finally, $x^2/4 \le (1-\sqrt{1-x})^2\le x^2$ holds
    for all $x\in[0,1]$ which proves $c_{1} \in[1/4,1]$. 
    
For any $\delta>0$, by Markov's inequality
$\P({\xi}/{\Var[\xi]^{1/2}}\le t) - \P(Z\le t)
\le \P(|\xi/\Var[\xi]^{1/2}-Z|>\delta) + \P(Z\in[t,t+\delta])
\le (\eps_1^2+c_1\eps_1^4)/\delta^2 + \delta(2\pi)^{-1/2}$ 
since the standard normal pdf is uniformly bounded by $(2\pi)^{-1/2}$. 
Hence, with $\delta=(\eps_1^2+c_1\eps_1^4)^{1/3}$, 
the above and a similar argument on $[t-\delta,t]$ provide 
the Kolmogorov distance bound. 
\end{proof}

Normal approximation results such as \Cref{thm:L2-distance-from-normal} are
flexible tools as they let us derive asymptotic normality results by
mechanically computing gradients: By \Cref{thm:L2-distance-from-normal} 
it suffices to show that the expectation of $\|\nabla f(\bz)\|_F^2$ is
negligible compared with that of $\|f(\bz)\|^2$ to obtain
$\xi/\Var[\xi]^{1/2}\to^d N(0,1)$.
Normal approximations involving derivatives have been studied
for random variables with the more general form $W=g(\bz)$
for differentiable functions $g:\R^n\to\R$.
The Second Order Poincar\'e inequality of \cite{chatterjee2009fluctuations}
bounds the total variation distance $d_{TV}$ of $g(\bz)$ to the Gaussian distribution
using the first and second derivatives of $g$:
\cite[Theorem 2.2]{chatterjee2009fluctuations} specialized
to $W=g(\bz)$ with $\bz\sim N(\mathbf{0},\bI_n)$
states that 
\begin{equation}
    \label{chatterjee-second-order-Poincare}
    d_{TV}\{W,\; N(\mu_0,\sigma_0^2)\} 
    \le (2\sqrt 5/{\sigma_0^2})\E[ \|\nabla g(\bz)\|^4]^{1/4}\E[ \|\nabla^2 g(\bz)\|_{op}^4]^{1/4}
\end{equation}
where $W=g(\bz)$, $\bz\sim N({\bf 0},\bI_n)$,
$\mu_0= \E[W]$ and $\sigma_0^2=\Var[W]$.
Above, $\nabla g,\nabla^2 g$ denote the gradient and 
Hessian matrix of $g$. 
Inequality \eqref{chatterjee-second-order-Poincare} provides
a CLT for $g(\bz)$
provided that the moments of the derivatives
$\E[ \|\nabla g(\bz)\|^4]^{1/4}$ and $\E[ \|\nabla^2 g(\bz)\|_{op}^4]^{1/4}$
are negligible compared to the variance $\sigma_0^2=\Var[g(\bz)]$. 
Inequality \eqref{chatterjee-second-order-Poincare} has been successfully applied
to derive asymptotic normality of unregularized $M$-estimators when $p/n\to\gamma < 1$
and the $M$-estimation loss is twice differentiable
\cite{lei2018asymptotics}.
However, the \eqref{chatterjee-second-order-Poincare}-based
approach is not applicable for regularized estimators such as the Lasso and
group Lasso that are only once differentiable functions of $(\bX,\by)$. 
In fact, by \Cref{proposition:lasso-general} below, 
the Lasso is not twice differentiable as 
$\trace[(\partial/\partial\by)\bX\hbbeta(\by,\bX)]$ is integer-valued.  
In \Cref{thm:L2-distance-from-normal},
while ${\xi} = \bz^\top f(\bz) - \dv f(\bz)$ already involves the derivatives
of $f$ through the divergence, the ratio ${\epsbar_1^2}$ that appears in the upper bound \eqref{identity-eps-n-norm-of-E-f} only involves $f$
and its gradient $\nabla f$; the second derivatives of
$f$ need not exist. \Cref{sec:application-de-biasing} uses
\Cref{thm:L2-distance-from-normal} to provide
asymptotic normality for de-biasing estimators that are 
only once differentiable. 

\paragraph{Variance estimate}
It follows from \Cref{thm:L2-distance-from-normal} that 
random variables $\xi$ 
of the form \eqref{stein-unbiased-random-variable} 
are asymptotically normal 
under the condition $1 - \|\E[f(\bz)]\|^2/\Var[\xi] \to 0$, 
or under a somewhat stronger but more explicit 
condition $\epsbar_1^2\to 0$ as in \eqref{identity-eps-n-norm-of-E-f}. 
The following theorem  provide consistent estimates of $\Var[\xi]$.

\begin{theorem}\label{thm:consistency-variance} 
Let $f, \bz, \xi, \eps_{1} ^2$ and $c_{1} \in [1/4,1]$
be as in \Cref{thm:L2-distance-from-normal}. Then, 
\bel{eq:consistency-variance}
    \E\bigl[\bigl(\|f(\bz)\|/\Var[\xi]^{1/2} - 1\bigr)^2\bigr]
    &\le& 
    \eps_{1} ^2 - \E\big[\trace(\{\nabla f(\bz)\}^2)]/\Var[\xi] +c_{1} \eps_{1} ^4
    \cr &\le& 
    (1- \eps_{1} ^2)\epsdoublebar_1^2/(2- 2\epsdoublebar_1^2)+c_{1} \eps_{1} ^4  
\eel
with $\epsdoublebar_1^2 \defas 2\E\big[\|\nabla f(\bz)\|_F^2\big]
\big/\{\|\E[f(\bz)]\|^2 + 2\E\big[\|\nabla f(\bz)\|_F^2\big]\}\ge\eps_{1} ^2$. 
Consequently, 
\bel{eq:consistency-variance-2}
\|f(\bz)\|^2/\Var[\xi]\to^\P 1
\quad
\hbox{ and }
\quad 
\xi/\|f(\bz)\| \to^d N(0,1). 
\eel
when $\eps_{1} ^2+\epsdoublebar_1^2I\{\E\big[\trace(\{\nabla f(\bz)\}^2)]<0\} \to 0$. 
\end{theorem} 


\begin{proof}[Proof of \Cref{thm:consistency-variance}]
It follows from the Jensen inequality and \eqref{second-order-poincare-eq1} that  
\bes
    \E\Big[\big(\|f(\bz)\|/\Var[\xi]^{1/2}-1\big)^2\Big]
    &\le & \E\big[\|f(\bz)\|^2\big]/\Var[\xi] +1 - 2\|\bmubar\|/\Var[\xi]^{1/2} 
    \cr &=& \eps_{1} ^2 - \E\big[\trace(\{\nabla f(\bz)\}^2)]/\Var[\xi] +c_{1} \eps_{1} ^4
    \ees
    with $\bmubar=\E[\nabla f(\bz)]$ due to $\eps_{1} ^2=1-\|\bmubar\|^2/\Var[\xi]$. 
    For the second inequality in \eqref{eq:consistency-variance}, 
    \bes
    \eps_{1} ^2 - \E\big[\trace(\{\nabla f(\bz)\}^2)]/\Var[\xi] 
    = \E\big[\|f(\bz)-\bmubar\|^2\big]/\Var[\xi]
    \le \E\big[\|\nabla f(\bz)\|_F^2\big]/\Var[\xi], 
    \ees
    thanks to the Gaussian Poincar\'e inequality, and 
    $(1- \eps_{1} ^2)\epsdoublebar_1^2/(2- 2\epsdoublebar_1^2)$ equals to the right-hand side above by 
    the definition of $\epsdoublebar_1^2$. 
\end{proof}

\subsection{Quadratic approximation}
\label{subsection:quadratic-approx} 
The decomposition \eqref{decomposition-linear-part-remainder}
is especially useful if the linear part $\bz^\top \bmu$ with $\bmu=\E[f(\bz)]$
is a good approximation for $\xi=\bz^\top f(\bz) - \dv f(\bz)$.
In some cases, e.g.,
if $f(\bz) = \bA \bz$ for some square deterministic matrix $\bA$,
the decomposition
\eqref{decomposition-linear-part-remainder} is uninformative.
It is then natural to look for the best quadratic
approximation of $\xi$ 
in the sense of the $L_2$ orthogonal projection to 
\bes
\scrH_{1,2} = \big\{\xi_{\bmu,\bA} = \bmu^\top \bz +\bz^\top \bA \bz - \trace[\bA]: 
\bmu\in\R^n, \bA\in\R^{n\times n}\big\} = \scrH_1\oplus\scrH_2, 
\ees
where $\scrH_1=\{\bmu^\top \bz:\bmu\in\R^n\}$ and 
$\scrH_2=\big\{\bz^\top \bA \bz - \trace[\bA]: \bA\in\R^{n\times n}\big\}$ 
are $L_2$ subspaces orthogonal to each other. 

The calculation in \eqref{second-order-poincare-eq1} for $\scrH_1$ is generic in
the following sense. 
If $\xibar$ is the $L_2$ projection
of a random variable $\xi$ in $L_2$
then the sine of the $L_2$-angle
between $\xibar$ and $\xi$ 
is
$\eps   = (\E[(\xi - \xibar)^2]/\E[\xi^2])^{1/2} = (1- \E[\xibar^2]/\E[\xi^2])^{1/2}$  and
    \bel{quadratic-approximation-2}
          && \E\big[\big(\xi/\Var[{\xi}]^{1/2} - \xibar/\Var[\xibar]^{1/2}\big)^2\big] 
        = 2\big(1-\sqrt{1-\eps  ^2}\big)
        = \eps  ^2 + c  \eps  ^4        
    \eel
    holds for some deterministic real $1/4 \le c   \le 1$.
    Indeed, take $\eps  =\sin\alpha$ with $\alpha$ being the 
    $L_2$-angle between $\xi$ and $\xibar$, so that \eqref{quadratic-approximation-2} becomes 
    $(2\sin(\alpha/2))^2=2(1-\cos(\alpha)) 
    = 
    \eps^2 + c  \eps  ^4$ 
    as in the proof of \Cref{thm:L2-distance-from-normal}. 

The next result 
extends \Cref{thm:L2-distance-from-normal} to the $L_2$ quadratic projections  
to $\scrH_2$ and $\scrH_{1,2}$, and also gives \Cref{thm:L2-distance-from-normal} 
the interpretation as the $L_2$ projection to $\scrH_1$. 

\begin{theorem}
    \label{thm:quadratic-approximation}
    Let $\bz\sim N({\bf 0},\bI_n)$,
    $f:\R^n\to\R^n$ satisfy the assumption of \Cref{thm:L2-distance-from-normal},  
    and $\xi=\bz^\top f(\bz) - \dv f(\bz)$. 
    For $\bmu\in\R^n$ and $\bA\in\R^{n\times n}$ let 
    $\xi_{\bmu,\bA} = \bz^\top(\bmu+\bA \bz) - \trace \bA$. 
    Let $\overline\bmu=\E[f(\bz)]$ and $\overline\bA=\E[\nabla f(\bz)]$. 
    Then, $\xi_{\bmubar,\bAbar}$ is the $L_2$ projection of $\xi$ to $\scrH_{1,2}$ and 
    \bel{quadratic-approximation-1}
    \E[(\xi - \xi_{\bmubar,\bAbar})^2]
    &=& \E [\|f(\bz)-\bmubar\|^2 - \|\bAbar\|_F^2] +\E\trace[ \{\nabla f(\bz)-\bAbar \}^2 ]
    \cr &\le & 2\E[\|\{\nabla f(\bz)-\bAbar\}^s\|_F^2]. 
    \eel
    Consequently, $\xi_{\bmubar,{\bf 0}}$ is the projection of $\xi$ and $\xi_{\bmubar,\bAbar}$ to 
    $\scrH_1$ with $\E[(\xi_{\bmubar,\bAbar}-\xi_{\bmubar,{\bf 0}})^2] = 2\|\bAbar^s\|_F^2$ and 
    $\xi_{{\bf 0},\bAbar}$ is the projection of $\xi$ and $\xi_{\bmubar,\bAbar}$ to 
    $\scrH_2$ with $\E[(\xi_{\bmubar,\bAbar}-\xi_{{\bf 0},\bAbar})^2] = \|\bmubar\|^2$. 

    For the projection $\xi_{\bmubar,\bAbar}$ of $\xi$ to $\scrH_{1,2}$ ,
    $\eps_{1,2}^2\defas 1-\E[\xi_{\bmubar,\bAbar}^2]/\E[\xi^2]$
    satisfies
    $\eps_{1,2}^2 \le \epsbar_{1,2}^2 
    \defas 2\E[\|\{\nabla f(\bz)-\bAbar\}^s\|_F^2]/\{\|\bmubar\|^2+2\E[\|\{\nabla f(\bz)\}^s\|_F^2]\}$ and
    under the condition $\eps_{1,2}^2=o(1)$,
    \bel{quadratic-approximation-3}
    \|\bAbar^s\|^{2}_{op}/(\|\bmubar\|_2^2 + \|\bAbar^s\|_F^2)\to 0
    \ \ \Leftrightarrow\ \ \xi/\Var[{\xi}]^{1/2} \to^d N(0,1).
    \eel

    For the projection $\xi_{\mathbf{0},\bAbar}$ of $\xi$ to $\scrH_2$,
    $\eps_2^2=1-\E[\xi_{\mathbf{0},\bAbar}^2]/\E[\xi^2]$ 
    satisfies
    $
    \eps_{2}^2\le \epsbar_2^2 
\defas \{\|\bmubar\|^2+2\E[\|\{\nabla f(\bz)-\bAbar\}^s\|_F^2]\}
/\{\|\bmubar\|^2+2\E[\|\{\nabla f(\bz)\}^s\|_F^2]\} 
$
    and under the condition $\eps_{2}^2=o(1)$,
    \bel{quadratic-approximation-H_2}
    \|\bAbar^s\|_{op}^2/\|\bAbar^s\|_F^2\to 0
    \ \ \Leftrightarrow\ \
    \xi/\Var[\xi]^{1/2} \to^d N(0,1).
    \eel
\end{theorem}


\begin{proof}[Proof of \Cref{thm:quadratic-approximation}]
    The function $g(\bz) = f(\bz) - \bmu - \bA^\top \bz$ has
    gradient $\nabla g = \nabla f - \bA$.
    Application of the Second Order Stein's formula 
    in \Cref{prop:2nd-order-stein-Bellec-Zhang} to $g$ yields
    \bes
    \E[(\xi - \xi_{\bmu,\bA})^2] = 
    \E[\| f(\bz) - \bmu - \bA^\top\bz\|^2 ] + \E\trace[\{\nabla f(\bz)-\bA\}^2]
    \ {\defas}\ 
    I + II.
    \ees
    The first term is 
    $I = \E[\|f(\bz)-\bmu\|^2] + \|\bA\|_F^2 - 2 \E[ \bz^\top \bA(f(\bz)-\bmu) ]
    $.
    By Stein's formula and the linearity of the trace, we have 
    \bes
    \|\bA\|_F^2 - 2\E[ \bz^\top \bA(f(\bz)-\bmu) ] 
    &=& 
    \|\bA\|_F^2 - 2\E\trace(\nabla f(\bz) \bA^\top)
    \\&=&
    \|\bA\|_F^2 - 2\trace[\bA^\top \bAbar] 
    \\&=& -\|\bAbar\|_F^2+\|\bA-\bAbar\|_F^2.
    \ees
    We also have
    $\E[\|\nabla f(\bz)-\bAbar\|_F^2] = \E[\|\nabla f(\bz)\|_F^2] - \|\bAbar\|_F^2$
    so that 
    $$I = \E[\|f(\bz)-\bmu\|^2 - \|\nabla f(\bz)\|_F^2]
    + \E[ \|\nabla f(\bz) - \bAbar\|_F^2] + \|\bA-\bAbar\|_F^2.$$
    For the second term, using that $\E[\nabla f(\bz) - \bAbar] =0$ we get
    \bes
    II
    = \E\trace[\{\nabla f(\bz)-\bA\}^2]
    =
    \E \trace[ \{\nabla f(\bz)-\bAbar\}^2] + \trace[ \{\bAbar - \bA\}^2].
    \ees
    Due to $\|\bM\|_F^2+\trace(\bM^2)=2\|\bM^s\|_F^2$ for $\bM\in\R^{n\times n}$, 
    it follows that 
    \bes
    \E[(\xi - \xi_{\bmu,\bA})^2]
    &= &\E [\|f(\bz)-\bmubar\|^2 - \|\bAbar\|_F^2] +\E\trace[ \{\nabla f(\bz)-\bAbar \}^2 ]
    \cr && +\|\bmu-\bmubar\|^2 + 2\|(\bA-\bAbar)^s\|_F^2
    \ees
    The optimality of $\bmu=\bmubar$ and $\bA=\bAbar$ follows, so that 
    $\xi_{\bmubar,\bAbar}$ is the $L_2$ projection of $\xi$ to $\scrH_{1,2}$. 
    Also, the first line above gives the formula of $\E[(\xi - \xi_{\bmubar,\bAbar})^2]$ 
    in \eqref{quadratic-approximation-1}, and the second line gives the formulas 
    for the variances of $\xi_{\bmubar,\bf 0}$ and $\xi_{\bf 0,\bAbar}$. 
    The upper bound in \eqref{quadratic-approximation-1} follows from 
    $\E [\|f(\bz)-\bmubar\|^2 - \|\bAbar\|_F^2] 
    \le \E [\|\nabla f(\bz)\|_F^2] - \|\bAbar\|_F^2 
    = \E [\|\nabla f(\bz)-\bAbar\|_F^2]$ 
    thanks to the Gaussian Poincar\'e inequality. 
    Inequality \eqref{quadratic-approximation-1} is equivalent to
    $$\E[\xi^2] = \E[\xi_{\bmubar,\bAbar}^2] + \E[(\xi-\xi_{\bmubar,\bAbar})^2] 
    \le\|\bmubar\|^2 + 2\|\bAbar^s\|_F^2 + 2\E[ \|\{\nabla f(\bz)-\bAbar\}^s\|_F^2]$$ which provides
    $\eps_{1,2}^2\le \epsbar_{1,2}^2$ and $\eps_2^2\le \epsbar_{2}^2$
    by bounding from above the denominator in $\eps_{1,2}^2=1-\E[\xi^2_{\bmubar,\bAbar}]/\E[\xi^2]$
    and $\eps_2^2 = 1-\E[\xi^2_{\mathbf{0},\bAbar}]/\E[\xi^2]$. 

    For \eqref{quadratic-approximation-3}, we write 
    $\xi_{\bmubar,\bAbar} = \sum_{j=1}^n\{a_j G_j+ b_j(G_j^2-1)\}$ with iid $G_j\sim N(0,1)$, 
    where $a_j = \bu_j^\top\bmubar$ and $G_j=\bu_j^\top \bz$ with the 
    eigenvalue decomposition $\bAbar^s = \sum_{j=1}^n b_j\bu_j\bu_j^\top$. 
    Assume without loss of generality that $\Var(a_j G_j+ b_j(G_j^2-1)) = a_j^2+2b_j^2$ is non-increasing in $j$
    and that $\Var[\xi_{\bmubar,\bAbar}]$ satisfies
    $\sum_{j=1}^n(a_j^2+2b_j^2) = 1$.
    The condition on the left-hand side of \eqref{quadratic-approximation-3}
        implies that the integer $k_n\defas\lceil\|\bAbar^s\|_{op}^{-1}\rceil
        =\lceil(\max_j b_j)^{-1} \rceil
        $ satisfies $k_n\to+\infty$
        and 
    $\sum_{j=1}^{k_n}b_j^2 \ll 1 = \Var[\xi_{\bmubar,\bAbar}]$, so that 
    $$\textstyle
    \xi_{\bmubar,\bAbar} = \sum_{j=1}^{k_n} a_j G_j
    +\sum_{j=k_n+1}^n\{a_j G_j+ b_j(G_j^2-1)\} + o_{\P}(1).$$
    Assuming that
        $\sum_{j=1}^{k_n}a_j^2\to c$ for some $c\in[0,1]$ by extracting
    a subsequence if necessary,
    $k_n\to+\infty$ implies 
    $\max_{j>k_n}(a_j^2+2b_j^2) =a_{k_n+1}^2 + 2b_{k_n+1}^2  \to0$
    so that
    the second term above is independent of the first and approximately $N(0,1-c)$ by the Lyapunov CLT
    when $\sum_{j=1}^{k_n}a_j^2\to c \le 1$. 
    This proves that the LHS of \eqref{quadratic-approximation-3} implies
    the RHS. 
    Conversely, assume the asymptotic normality on the RHS so that 
    $\sum_{j=1}^n\{a_j G_j+ b_j(G_j^2-1)\}\to N(0,1)$. 
    Let $W_j = a_j G_j+ b_j(G_j^2-1)$ and $j_n\le n$. 
    As $W_{j_n}$  is an independent component of the sum, 
    for any $(a_{j_n},b_{j_n}) \to (a,b)$ along a subsequence with $a^2+b^2>0$, 
    we must have $b=0$ because 
    $W_{j_n}\to^d N(0,a^2+2b^2)$ by the Cram\'er-L\'evi theorem and 
    $W_{j_n}\to^d aG+b(G^2+1)$ for some $G\sim N(0,1)$. 
    As $j_n\le n$ are arbitrary, this gives 
    $\|\bAbar^s\|_{op} = \max_{j=1.,..,n} b_j^2\to 0$.
    \end{proof}

\paragraph{Variance estimate: quadratic case}
\Cref{thm:quadratic-approximation} provides the quadratic 
normal approximation of $\xi$ under the 
condition $\epsbar_{1,2}^2\to 0$ with
\bel{eps_1,2}
\epsbar_{1,2}^2 \ge \eps_{1,2}^2 = 1 - (\|\E[f(\bz)]\|^2+2\|\{\E[\nabla f(\bz)]\}^s\|_F^2)/{\Var[\xi]} 
\eel
where $\epsbar_{1,2}^2$ is defined using the upper bound 
$\|\E f(\bz)\|^2+2\E[\|\{\nabla f(\bz)\}^s\|_F^2]\ge \Var[\xi]$ 
established in \eqref{quadratic-approximation-1} 
in the denominator on the right-hand side 
of \eqref{eps_1,2}.

\begin{theorem}\label{thm:consistency-variance-two-terms} 
Let $f, \bz, \xi, \bmubar = \E[f(\bz)]$ and $\bAbar = \E[\nabla f(\bz)]$ 
be as in \Cref{thm:quadratic-approximation}  
and $\widehat{\Var[\xi]} =\|f(\bz)\|^2+\trace[\{\nabla f(\bz)\}^2]$. Then, 
\bel{thm:consistency-variance-two-terms-1}
\quad && \E\Big[\Big|\widehat{\Var[\xi]}/\Var[\xi]  - 1 \Big|\Big]
\le 2\epsdoublebar_{1,2}^2 + 2\epsdoublebar_{1,2} C_0 + C_0\|\bAbar\|_{op}/\Var[\xi]^{1/2}
\eel
with $\epsdoublebar_{1,2} \defas 
\{2\E[\|\nabla f(\bz)-\bAbar\|_F^2]/\Var[\xi]\}^{1/2}$ and 
$C_0 \defas \{(\|\bmubar\|^2+2\|\bAbar\|_F^2)/\Var[\xi]\}^{1/2}$. Consequently, 
under the conditions $\big\{\E[\|\nabla f(\bz)-\bAbar\|_F^2]+\|\bAbar\|_{op}\big\}
\big/(\|\bmubar\|^2+\|\bAbar^s\|_F^2) = o(1)$ 
and $\|\bAbar\|_F^2/(\|\bmubar\|^2+\|\bAbar^s\|_F^2)=O(1)$,  
\bel{thm:consistency-variance-two-terms-2}
\widehat{\Var[\xi]}/\Var[\xi]\to^\P 1\ \hbox{ and }\ \big(\widehat{\Var[\xi]}\big)^{-1/2}\xi \to^d N(0,1). 
\eel
\end{theorem} 

It follows from the Second Order Stein formula in \Cref{prop:2nd-order-stein-Bellec-Zhang} that 
$\widehat{\Var[\xi]}$ is an unbiased estimator of $\Var[\xi]$. 
Moreover, when $\E[\|\nabla f(\bz)-\bAbar\|_F^2]$, $\|\bAbar\|_F$ and $\|\bAbar\|_{op}$ are all equivalent to their symmetric counterparts, $\E[\|\nabla f(\bz)-\bAbar\|_F^2]\asymp \E[\|\{\nabla f(\bz)-\bAbar\}^s\|_F^2]$, $\|\bAbar\|_F\asymp \|\bAbar^s\|_F$ 
and $\|\bAbar\|_{op}\asymp \|\bAbar^s\|_{op}$, the condition for 
\eqref{thm:consistency-variance-two-terms-2} holds if and only if 
$\eps_{1,2}^2+\|\bAbar^s\|_{op}^2/(\|\bmubar\|^2+\|\bAbar^s\|_F^2)=o(1)$ 
for the quantities in \eqref{eps_1,2} and \eqref{quadratic-approximation-3}. 
The proof is given in \Cref{appendix:proof-variance-quadratic-case}.

\section{De-biasing general convex regularizers}
\label{sec:application-de-biasing}
Our main application of the normal approximation
in \Cref{thm:L2-distance-from-normal} concerns
de-biasing regularized estimators of the form
\begin{equation}
    \label{penalized-hbbeta-general}
    \hbbeta = \argmin_{\bb\in\R^p}\left\{\|\by - \bX\bb \|^2/(2n) + g(\bb) \right\} 
\end{equation}
for
convex $g:\R^p\to\R$ in the linear model \eqref{LM}.
Throughout, let $\bh=\hbbeta-\bbeta$ be the error vector, 
$\ba_0\in\R^p$ be a direction of interest, $\theta=\langle \ba_0,\bbeta\rangle$ 
be the target of statistical inference, and
$\bu_0$, $\bz_0$ and $\bQ_0$ be as in \eqref{u_0} so that
\eqref{eq:ref-X-Q0} holds.

\subsection{Assumption}
We say that $g$ is $\mu$-strongly convex
with respect to the norm $\bb\to\|\bSigma^{1/2}\bb\|$ if its symmetric Bregman divergence 
is bounded from below as
\begin{equation}
    \label{strong-convex}
    (\tilde\bb - \bb)^\top\big((\pa g)(\tilde\bb) - (\pa g)(\bb)\big) 
    \ge \mu\|\bSigma^{1/2}(\tilde\bb - \bb)\|^2
\end{equation}
for some $\mu > 0$. 
Here the interpretation of \eqref{strong-convex} is 
its validity for all choices in the sub-differential $(\pa g)(\tilde\bb)$ and $(\pa g)(\bb)$. 
Condition \eqref{strong-convex} holds for any convex $g$ for $\mu=0$.
If $g$ is twice differentiable, \eqref{strong-convex} holds if and only if
$\mu\bSigma$ is a lower 
bound for the Hessian of $g$. However, \eqref{strong-convex} may also hold for 
non-differentiable $g$, e.g. the Elastic-Net penalty with $\bSigma=\bI_p$.
Our results require the following assumption.

\begin{assumption}
    \label{assum:main}
    (i) Let $\gamma>0,\mu\in[0,\frac12]$ be constants
    such that $\mu+(1-\gamma)_+ > 0$,
    that is, either $\mu>0$ or $\gamma<1$ must hold.
    Consider a sequence of regression problems \eqref{LM}
    with $n,p\to+\infty$ and $p/n\le\gamma$. 
    The penalty $g:\R^p\to\R$ in \eqref{penalized-hbbeta-general} 
    is convex and 
    {\eqref{strong-convex}} holds. 
    The rows of $\bX$ are iid $N({\bf 0},\bSigma)$
    with invertible $\bSigma$ and  
    the noise $\bep\sim N({\bf 0},\sigma^2\bI_n)$ is independent of $\bX$. 
    (ii) $\ba_0\in\R^p$ is a sequence of vectors normalized
    with $\|\bSigma^{-1/2}\ba_0\|=1$. 
\end{assumption}

Note that if \eqref{strong-convex} holds for $\mu\ge 0$
it also holds for $\mu'=\min(\frac 12, \mu)$ and we may thus assume
$\mu\in[0,\frac12]$ without loss of generality.
Strongly convex objective functions admit
unique minimizers.
Since
$\gamma<1$ implies $\P(\phi_{\min}(\bSigma^{-1/2}\bX^\top\bX\bSigma^{-1/2})>0)=1$ (cf. \Cref{appendix:integrability-edelman})
and
the objective function of the optimization problem
\eqref{penalized-hbbeta-general}
is $(\phi_{\min}(\bSigma^{-1/2}\bX^\top\bX\bSigma^{-1/2}/n) + \mu)$-strongly convex,
\Cref{assum:main} grants almost surely the
existence and uniqueness of the minimizer
\eqref{penalized-hbbeta-general}.

\subsection{Gradient with respect to $\by$ and effective degrees-of-freedom}

Consider a penalized estimator \eqref{penalized-hbbeta-general}
viewed as a function $\hbbeta=\hbbeta(\by,\bX)$.
For every $\bX\in\R^{n\times p}$,
the map $\by\mapsto\bX\hbbeta(\by,\bX)$
is 1-Lipschitz (cf. \Cref{prop:hat-bH}).
By Rademacher's theorem,
for almost every $\by$ there exists a unique matrix $\hbH\in\R^{n\times n}$
such that
\begin{equation}
    \bX\hbbeta(\by+\bfeta,\bX)
    =
    \bX\hbbeta(\by,\bX)
    +
    \hbH{}^\top\bfeta + o(\|\bfeta\|),
    \label{def-matrix-H}, 
\end{equation}
as in \eqref{eq:def-H-intro}, 
i.e., $\hbH$ is the gradient of the map $\by\mapsto \bX\hbbeta(\by,\bX)$.
Furthermore $\hbH$ is symmetric with eigenvalues in $[0,1]$; 
See 
\Cref{prop:hat-bH} for the existence of $\hbH$ and its properties.
While existence of $\hbH$ 
was assumed in \eqref{eq:def-H-intro} in the introduction,
for penalized estimators \eqref{penalized-hbbeta-general}
the matrix $\hbH$ provably exists for almost every $\by$ by 
\Cref{prop:hat-bH}.

\Cref{table:hbH} provides closed-form expressions
of $\hbH$ for specific penalty functions $g$. The proofs of these
closed-form expressions will be given in \Cref{sec:examples}.
An advantage of defining $\hbH$ as the Fr\'echet derivative of the Lipschitz
map $\by\mapsto\bX\hbbeta(\by,\bX)$
is that this definition applies to any convex penalty $g$, 
even though for arbitrary penalty $g$ we are unable to provide a
closed-form expression for $\hbH$.
Finally, 
define the effective degrees-of-freedom $\df$ of $\hbbeta$ by
\begin{equation}
    \label{eq:def-trace-H}
    \df = \trace[\hbH] 
\end{equation}
as in \eqref{def-hat-df-and-B}. 
Because $\hbH$ is symmetric with eigenvalues in $[0,1]$
(cf. \Cref{prop:hat-bH}),
$0\le \df \le n$ holds almost surely.
The matrix $\hbH$ and the scalar $\df$
play a major role in our analysis. 

\begin{table}[ht]
\begin{tabular}{@{}|p{5cm}|l|l|l|@{}}
\toprule
Penalty & $\hbH\in\R^{n\times n}$ & Justification  \\ \midrule
$g(\bb)=\lambda\|\bb\|_1$ (Lasso) &
${\bX_{\Shat}}(\bX_{\Shat}^\top\bX_{\Shat})^{-1}\bX_{\Shat}^\top$
                                  &
                                          \cite{tibshirani2013lasso},
                                  \Cref{proposition:lasso-general}
                                  \\ \midrule
$g(\bb)=\mu\|\bb\|_2^2$ (Ridge) &
${\bX}(\bX^\top\bX + n\mu\bI_p)^{-1}\bX^\top$
                                & \eqref{eq:ridge-calculation-intro},
                                \Cref{sec:computation-twice-differentiation-g}  \\ \midrule
$g(\bb)=\lambda\|\bb\|_1 + \mu\|\bb\|_2^2$ (Elastic-Net)&
${\bX_{\Shat}}(\bX_{\Shat}^\top\bX_{\Shat} + n\mu \bI_{|\Shat|})^{-1}\bX_{\Shat}^\top$
                                          & 
                                          \cite[(28)]{tibshirani2013lasso},
                                          \cite[\S 3.5.3]{bellec_zhang2018second_order_stein}
                                          \\ \midrule
$g(\bb)=\|\bb\|_{GL}=\sum_{k=1}^K\lambda_k\|\bb_{G_k}\|_2$
(group Lasso \eqref{group-lasso})&
${\bX_{\Shat}}(\bX_{\Shat}^\top\bX_{\Shat} + \bM)^{-1}\bX_{\Shat}^\top$
                                                                                            &                                                                              \cite{vaiter2012degrees}, \Cref{lemma:gradient-group-lasso} \\ \midrule
$g(\bb)$ twice continuously differentiable &
${\bX}(\bX^\top\bX + n\nabla^2g(\hbbeta))^{-1}\bX^\top$
                                           &  \Cref{sec:computation-twice-differentiation-g}
                                          \\ \midrule
$g(\bb)$ arbitrary convex function &
symmetric with eigenvalues in $[0,1]$
                                          &
\Cref{prop:hat-bH}
                                          \\ \bottomrule
\end{tabular}
\caption{
    Closed-form expressions $\hbH$
    from \Cref{def-matrix-H}
    for specific convex penalty functions $g:\R^p\to\R$.
    For the Lasso and Elastic-Net,
    $\Shat = \{j\in[p]:\hbeta_j\ne 0\}$.
    For the group Lasso, $\Shat$ 
    and $\bM$ are given in 
    \Cref{sec:de-biasing-group-lasso}. 
}
\label{table:hbH}
\end{table}

\subsection{Approximation for 
\texorpdfstring{$\xi_0=\bz_0^\top f(\bz_0) - \dv f(\bz_0)$}{xi0}
and the de-biased vector 
\texorpdfstring{ $\DeBias$ }{hbbeta-debiased}}
Consider, for a fixed value of                 
$(\bX\bQ_0,\bep)$ the function $f_{(\bX\bQ_0,\bep)}:\R^n\to\R^n$ given by
\begin{equation}
    \label{def-f-xi}
    f_{(\bX\bQ_0,\bep)}(\bz_0)
    =
    f(\bz_0) = \bX\hbbeta - \by.
\end{equation}
For brevity we will often 
omit the dependence on $(\bX\bQ_0,\bep)$ of $f$ as
discussed after \eqref{definition-f-intro}.
The Fr\'echet gradient $\nabla f(\bz_0)$, where it exists,
is uniquely defined by
\begin{equation}
    f_{(\bX\bQ_0,\bep)}(\bz_0+\bfeta)
    -
    f_{(\bX\bQ_0,\bep)}(\bz_0)
    =
    [\nabla f(\bz_0)]^\top\bfeta + o(\|\bfeta\|)
    \label{eq:gradient-f-bXQ0-bep}
\end{equation}
and the divergence by
$\dv f(\bz_0) = \trace[\nabla f(\bz_0)]$.
If 
$\tbbeta=\argmin_{\bb\in\R^p}
\bigl(\|\bep - \tbX(\bb-\bbeta)\|^2/(2n) + g(\bb)\bigr)$
with $\tbX=\bX + \bfeta \ba_0^\top$, then
\eqref{eq:gradient-f-bXQ0-bep} is equivalent to
\begin{equation}
\label{gradient-f-bz0-explicit}
(\tbX(\tbbeta-\bbeta) - \bep)
-
(\bX\hbbeta - \by)
=
[\nabla f(\bz_0)]^\top
\bfeta
+ o(\|\bfeta\|).
\end{equation}
By Stein's formula,
we have conditionally on $(\bX\bQ_0,\bep)$ that almost surely
\begin{equation}
    \label{xi0-mean-zero-conditionally}
    \E[\xi_0 ~|~ (\bX\bQ_0,\bep)] = 0
    \qquad
    \text{ for }
    \quad
    \xi_0 = \bz_0^\top f(\bz_0) - \dv f(\bz_0).
\end{equation}
As in \eqref{intro-decomposition} for the general case discussed
in the introduction,
\eqref{xi0-mean-zero-conditionally} gives an unbiased
estimating equation for $\theta=\langle \ba_0,\bbeta\rangle$.
The next lemma provides an expression for $\nabla f(\bz_0)$.

\begin{restatable}{lemma}{lemmaExistenceWZero}
    \label{lemma:existence-bw0}
    Let \Cref{assum:main}(i) be fulfilled, 
    $\ba_0\in\R^p$ and $\hbH$ be as in \eqref{def-matrix-H}. 
    Then, 
    \begin{equation}
        \label{eq:identity-bw0}
    \nabla f(\bz_0)^\top = 
    (\bI_n-\hbH)^\top\langle \ba_0,\bh\rangle
    +  \bw_0 (\by-\bX\hbbeta)^\top
    \end{equation}
    satisfies \eqref{eq:gradient-f-bXQ0-bep}
    for some random $\bw_0\in\R^n$ almost surely.
    If additionally $\|\bSigma^{-1/2}\ba_0\|=1$ then 
    \begin{equation}
    \|\bw_0\|^2 \le
    n^{-1}
    \min\bigl\{
        (4 \mu)^{-1},
        \phi_{\min}
        (
            \bSigma^{-1/2}\bX^\top\bX\bSigma^{-1/2}/n
        )^{-1}
    \bigr\}.
    \label{eq:bound-norm-bw0}
    \end{equation}
\end{restatable}

\begin{table}[ht]
\begin{tabular}{@{}|p{5cm}|l|l|l|@{}}
\toprule
Penalty & Vector $\bw_0\in\R^n$ in \Cref{lemma:existence-bw0} & Justification  \\ \midrule
$g(\bb)=\lambda\|\bb\|_1$ (Lasso) &
${\bX_{\Shat}}(\bX_{\Shat}^\top\bX_{\Shat})^{-1} (\ba_0)_{\Shat}$
                                  & 
                                  \cite{bellec_zhang2019dof_lasso},
                                  \Cref{proposition:lasso-general}
                                  \\ \midrule
$g(\bb)=\mu\|\bb\|_2^2$ (Ridge) &
${\bX}(\bX^\top\bX + n\mu\bI_p)^{-1}\ba_0$
                                &
                                \Cref{sec:computation-twice-differentiation-g} 
                                \\ \midrule
$g(\bb)=\|\bb\|_{GL}=\sum_{k=1}^K\lambda_k\|\bb_{G_k}\|_2$ 
(group Lasso \eqref{group-lasso})&
${\bX_{\Shat}}(\bX_{\Shat}^\top\bX_{\Shat} + \bM)^{-1} (\ba_0)_{\Shat}$
                                          & 
\Cref{lemma:gradient-group-lasso}
                                          \\ \midrule
$g(\bb)$ twice continuously differentiable &
${\bX}(\bX^\top\bX + n\nabla^2g(\hbbeta))^{-1}\ba_0$
                                &  \Cref{sec:computation-twice-differentiation-g} 
                                         \\ \bottomrule
\end{tabular}
\caption{Closed-form expressions for $\bw_0\in\R^n$
    in \Cref{lemma:existence-bw0}
    for specific convex penalties $g:\R^p\to\R$.
}
\label{table:bw0}
\end{table}

\Cref{lemma:existence-bw0} is proved in \Cref{sec:proof-lipschitzness}.
Although we do not use this fact in any results,
we mention here in passing that
vector $\bw_0$ in \eqref{eq:identity-bw0} is linear in $\ba_0$
in the sense that $\bw_0$ can be chosen
of the form $\bW\bSigma^{-1/2}\ba_0$ for some matrix $\bW\in\R^{n\times p}$.
Indeed, the proof of \Cref{lemma:existence-bw0}
shows that the map $(\bep,\bX)\mapsto \bX(\hbbeta-\bbeta)- \bep$
is Fr\'echet differentiable at almost every point
by Rademacher's theorem.
At such a point,
with $\ba_1,\ba_2\in\R^p$, $t\in \R$,
the linear combination $\ba_3=\ba_1+ t \ba_2$
and the perturbed design matrix $\tbX = \bX + \bfeta(\ba_1 + t \ba_2)^\top$,
linearity of the Fr\'echet derivative implies that
\bes
&& (\tbX(\tbbeta-\bbeta) - \bep)
-
(\bX\hbbeta - \by)
\cr &=&
\bigl(
\langle \ba_1 + t\ba_2,\bh\rangle (\bI_n-\hbH)^\top
+  (\bw_1 + t \bw_2) (\by-\bX\hbbeta)^\top
\bigr)
\bfeta
+ o(\|\bfeta\|)
\ees
where $\bw_1$ and $\bw_2$ denote the $\bw_0$ from \eqref{eq:identity-bw0}
for $\ba_0=\ba_1$ and $\ba_0=\ba_2$ respectively.
Hence with $\bw_3 = \bw_1  + t\bw_2$,
\eqref{eq:identity-bw0} holds for $(\ba_0,\bw_0)=(\ba_3,\bw_3)$.
This proves that $\bw_0$ is linear in $\ba_0$,
i.e., it is of the form
$\bw_0 = \bW\bSigma^{-1/2}\ba_0$ for some matrix $\bW\in\R^{n\times p}$.
One way to construct $\bW$ explicitly
is the following: define the $j$-th column of 
$\bW$ as the vector $\bw_0$ corresponding to $\ba_0 = \bSigma^{1/2}\be_j$
where $\be_j$ is the $j$-th canonical basis vector. The linearity proved above for any linear combination $\ba_3$ then implies that
\eqref{eq:identity-bw0} holds for $\bw_0=\bW\bSigma^{-1/2}\ba_0$ for any $\ba_0\in\R^p$.
Inequality \eqref{eq:bound-norm-bw0} 
provides an upper bound
on the operator norm of $\bW$.
Linearity of $\bw_0$ with respect to $\ba_0$
and explicit matrices $\bW$ can be seen for some penalty functions
in \Cref{table:bw0}.

By taking the trace of \eqref{eq:identity-bw0}, we obtain
almost surely under \Cref{assum:main}
\bel{-xi0}
-\xi_0 
&=&
\dv f(\bz_0) - \langle \bz_0, f(\bz_0)\rangle
\cr&=& (n-\df)({\langle \ba_0,\hbbeta\rangle} - \theta) + \langle \bz_0 + \bw_0,\by-\bX\hbbeta\rangle
\cr&=&
(n-\df)(\htheta - \theta)
\eel
for
\begin{equation}
    \label{def-hat-theta}
\htheta \defas \langle \ba_0,\hbbeta\rangle + 
(n-\df)^{-1}\langle \bz_0 + \bw_0, \by-\bX\hbbeta\rangle.
\end{equation}
In the present context of the regularized
estimator $\hbbeta$ in \eqref{penalized-hbbeta-general}
with its effective degrees-of-freedom $\df$ defined in
\eqref{eq:def-trace-H}
and under \Cref{assum:main},
the quantities $\xi_0$, $\df$, $\htheta$ in the previous display
coincide with the random variables with the same name
in \eqref{intro-decomposition}.
By \eqref{xi0-mean-zero-conditionally}, equality $\E[\xi_0] = 0$ holds
so that 
\bel{eq:unbiased-estimating-equation-2}
0 = \E[-\xi_0]
   &=&
    \E[(n-\df)(\htheta-\theta)]
 \cr&=&
    \E[
    (n-\df)(\langle \ba_0,\hbbeta\rangle - \theta) + \langle \bz_0 + \bw_0,\by-\bX\hbbeta\rangle
    ]
\eel
by taking expectations in \eqref{-xi0}.
This provides a first evidence that the correction
$(n-\df)^{-1}\langle \bz_0+\bw_0,\by-\bX\hbbeta\rangle$
indeed removes the bias, at least after multiplication of $(\htheta-\theta)$
by $(n-\df)$.
Since $\bz_0=\bX\bSigma^{-1}\ba_0$ under the normalization
\eqref{normalization-a_0},
the unbiased estimating equation 
\eqref{eq:unbiased-estimating-equation-2}
is the specialization of
\eqref{eq:unbiased-estimating-equation} from the introduction
to the penalized estimator \eqref{penalized-hbbeta-general},
for which we have the gradient expression \eqref{eq:identity-bw0}
in terms of $\bw_0$ and $\hbH$.
If $\hbbeta$ is given by \eqref{penalized-hbbeta-general},
by identifying the terms in
\eqref{intro-decomposition}
and \eqref{-xi0} we see that 
the random variable $\widehat{A}$ defined
in \eqref{def-hat-df-and-B} of the introduction 
is here $\widehat{A}=\langle \bw_0, \by-\bX\hbbeta\rangle$
with $\bw_0$ given by \Cref{lemma:existence-bw0}.
The following lemma shows that this term is negligible.

\begin{restatable}{lemma}{lemmaNegligibleTerm}
    \label{lemma:negligible-term}
    Under \Cref{assum:main} there exists
    $\Omega_n$ with $\P({\Omega_n^c})\le C_0(\gamma,\mu)n^{-1/2}$
    and
    \begin{equation}
    \E\bigl[
        I_{\Omega_n}
        \langle \bw_0, \by-\bX\hbbeta\rangle^2
        \big/
        \Var_0[\xi_0]
    \bigr]
    \le {\C(\gamma,\mu)}n^{-1}.
    \label{eq:extra-term-to-0}
    \end{equation}
\end{restatable}
The proof is given in \Cref{sec:proof-variance}.
Since $\P(\Omega_n)\to 1$, inequality \eqref{eq:extra-term-to-0}
implies
$\langle \bw_0,\by-\bX\hbbeta\rangle^2/\Var_0[\xi_0]\to^\P0$. 
This
motivates the definition 
\begin{equation}
    \label{eq:De-Bias-hbbeta}
    \DeBias
    \defas \hbbeta
    + (n-\df)^{-1}\bSigma^{-1}\bX^\top( \by-\bX\hbbeta ).
\end{equation}
The de-biased estimate $\langle \ba_0,\DeBias\rangle$
in direction $\ba_0$ is obtained
from $\htheta$ in \eqref{def-hat-theta} by dropping the smaller order term
$(n-\df)^{-1}\langle \bw_0,\by-\bX\hbbeta\rangle$.
By Slutsky's theorem, \eqref{eq:extra-term-to-0} implies
\begin{equation}
\frac{\xi_0}{
    \Var_0[\xi_0]^{1/2}
}\to^d F
\qquad\text{if and only if}\qquad
\frac{(n-\df)\langle\ba_0,  \DeBias-\bbeta\rangle}{
    \Var_0[\xi_0]^{1/2}
}
\to^d F
\label{eq:limiting-F}
\end{equation}
for any candidate limiting distribution $F$.
As $\E[\xi_0]=0$, this suggests that the simpler correction
in \eqref{eq:De-Bias-hbbeta} 
also corrects the bias
of $\hbbeta$. By Prohorov's theorem, there exists a subsequence
and limiting distribution $F$ such that \eqref{eq:limiting-F} holds
in this subsequence. While
$F$ is mean-zero as $\xi_0/\Var_0[\xi_0]$ has mean zero and variance one, 
$F$ has variance at most one by Fatou's lemma. 
However, the normality of $F$ is unclear at this point. 

To obtain more precise information on the 
limiting distribution and the deviations of $\xi_0$,
the next subsections build estimate of its variance
and derive asymptotic normality results by showing that
$F=N(0,1)$ for most directions $\ba_0$.
The next result provides a loose data-driven
upper bound on the error
$\langle\ba_0, \DeBias\rangle-\theta$.

\begin{restatable}{theorem}{theoremConstantTimesResiduals}
    \label{thm:De-Bias-within-constant}
    Under \Cref{assum:main} there exists  $\Omega_n$ with
    $\P(\Omega_n^c)\le C_0(\gamma,\mu) n^{-1/2}$ and 
    \begin{equation}
        \label{eq:DeBias-within-constant-factor}
    \E\bigl[
    I_{\Omega_n}
    (n-\df)^2 \langle \ba_0,\DeBias - \bbeta\rangle^2 \big/
    \|\by-\bX\hbbeta\|^2
    \bigr] 
    \le \C(\gamma,\mu).
    \end{equation}
    Furthermore
    $|\langle \ba_0,\DeBias-\bbeta\rangle|
    = O_\P(1) \|\by-\bX\hbbeta\|/(n-\df)
    = O_\P(1) \|\by-\bX\hbbeta\|/n$.
\end{restatable}
\Cref{thm:De-Bias-within-constant} is proved in
\Cref{sec:proof-asymptotic-normality}.
If $\|\bX(\hbbeta-\bbeta)\|^2/n = O_\P(\sigma^2)$ then
$\|\by-\bX\hbbeta\|/n = O_\P(1) \sigma/\sqrt n$ is
of the same order as the width of confidence intervals
based on the least-squares estimator as $n\to 0$ while $p$ remains fixed.
\Cref{thm:De-Bias-within-constant} shows that under
this mild additional assumption on the prediction error
$\|\bX(\hbbeta-\bbeta)\|^2/n$,
the second term in \eqref{eq:De-Bias-hbbeta} indeed
corrects the bias, 
achieving $\langle \ba_0,\DeBias - \bbeta\rangle = O_\P(1) \sigma/\sqrt n$.

\subsection{Variance estimates}
By \Cref{prop:2nd-order-stein-Bellec-Zhang},
the conditional variance
$\Var_0[{\xi}_0]$ 
can be written as $\Var_0[{\xi}_0]=\E_0[V^*(\theta)]$ for  
\bel{Var(xi_0)} && 
V^*(\theta)
\defas
\|\by-\bX\hbbeta\|^2 + \trace[\{\nabla f(\bz_0)\}^2]
.
\eel
We allow the variance estimate to depend on the unknown
$\theta=\langle \ba_0,\bbeta\rangle$ as the resulting pivotal quantity,
$-V^*(\theta)^{-1/2}\xi_0=V^*(\theta)^{-1/2}(n-\df)(\htheta-\theta)$
via \eqref{-xi0}, would depend on $\theta$ anyway.
While $V^*(\theta)$ itself can be used to estimate $\Var_0[{\xi}_0]$,
its sign is unclear.
The following simplified version of it,
obtained by removing the smaller order terms in $V^*(\theta)$,
\begin{align}
    \label{def-hat-V-theta}
\widehat{V}(\theta) 
& \defas  \|\by-\bX\hbbeta\|^2 +
\trace[(\hbH - \bI_n)^2]
\bigl(\langle \ba_0,\hbbeta\rangle - \theta\bigr)^2 
\\& =
\|\by-\bX\hbbeta\|^2 + \|\hbH-\bI_n\|_F^2
\langle \ba_0,\bh\rangle^2, 
\nonumber
\end{align}
is non-negative. This follows from \Cref{prop:hat-bH} since 
$\bI_n-\hbH$ is almost surely 
positive semi-definite.
\Cref{lemma:delta-0} below shows that the relative bias
$\E_0[\widehat{V}(\theta)]/\Var_0[\xi_0] - 1$
converges to 0 in probability,
i.e., $\widehat{V}(\theta)$ is asymptotically unbiased
for $\Var_0[\xi_0]$.

\begin{restatable}{lemma}{lemmaDeltaZero}
    \label{lemma:delta-0}
    Under \Cref{assum:main} there exists
    $\Omega_n$ with $\P(\Omega_n^c)\le C_0(\gamma,\mu)n^{-1/2}$
    and
    \begin{equation}
        \label{eq:delta-0-to-0}
    \E\Bigl[
        I_{\Omega_n}
    \Big|\frac{\E_0[\widehat{V}(\theta)]}{\Var_0[\xi_0]} -  1 \Big|
    \Bigr]
    \le
    \E\Bigl[
        I_{\Omega_n} ~
    \frac{\E_0[|\widehat{V}(\theta) - V^*(\theta)|]}{ \Var_0[\xi_0]}
    \Bigr] 
    \le \frac{\C(\gamma,\mu)}{n}.
    \end{equation}
\end{restatable}

An alternative variance estimate, that does not depend on the
unknown parameter $\theta$, is given
by replacing $\theta$ in $\widehat{V}(\theta)$ by the point
estimate $\langle \ba_0,\DeBias\rangle$ with
$\DeBias$ in \eqref{eq:De-Bias-hbbeta}:
\begin{equation}\label{eq:check-V-variance-def}
\check V(\ba_0)
=
\widehat{V}({\langle \ba_0,\DeBias\rangle})
=
\|\by-\bX\hbbeta\|^2 +
\|\bI_n - \hbH\|_F^2
\frac{\langle \bz_0, \by-\bX\hbbeta\rangle^2}{ (n-\df)^2}.
\end{equation}
The next lemma provides
$\check V(\ba_0)/\widehat{V}(\theta)\to^\P 1$
and that $\check V(\ba_0)$ is also asymptotically unbiased
in the sense $\E_0[\check{V}(\ba_0)]/\Var_0[\xi_0] \to^\P 1$.
\Cref{lemma:delta-0,lemma:check-V-variance} are proved in
\Cref{sec:proof-variance}.

\begin{restatable}{lemma}{lemmaCheckVariance}
    \label{lemma:check-V-variance}
    Under \Cref{assum:main} there exists
    $\Omega_n$ with $\P(\Omega_n^c)\le C_0(\gamma,\mu)n^{-1/2}$
    and
    \bel{eq:check-variance-1}
    \max \Bigl\{
    \E\Bigl[
        I_{\Omega_n}
        \Big|\frac{\check V(\ba_0)^{1/2}}{\widehat{V}(\theta)^{1/2}} - 1\Big|^2
    \Bigr]
    ,~~
    \E\Bigl[
        I_{\Omega_n}
        \Big|\frac{\E_0[\check V(\ba_0)]^{1/2}}{\E_0[\widehat{V}(\theta)]^{1/2}} - 1\Big|^2
    \Bigr]
    \Bigr\}
    &\le& \frac{\C(\gamma,\mu)}{n}.
    \eel
\end{restatable}

\subsection{Asymptotic normality of de-biased estimates}
\label{sec:3.7-general-result}

Throughout this section,
$\Phi(t)=\P(N(0,1)\le t)$ denotes the standard normal cdf.
For a given penalty function $g:\R^p\to\R$,
we define the deterministic oracle $\bbeta^*$ and its associated noiseless
prediction risk $\risk$ by
\begin{equation}
\label{oracle-general-g-strongly-convex}
  \bbeta^* \defas \argmin_{\bb\in\R^p}
      \big\|\bSigma^{1/2}(\bbeta-\bb)\big\|^2/2+g(\bb)
    ,
    \quad
    \bh^* \defas
    \bbeta^*-\bbeta,
    \quad
    {\risk}
    \defas
    \sigma^2 + \|\bSigma^{1/2}\bh^*\|^2
    .
\end{equation}
Our first result provides asymptotic normality of the de-biased
estimate when the error $\langle \ba_0,\hbbeta-\bbeta\rangle$
of $\hbbeta$ in direction $\ba_0$ is negligible compared
to $\risk$.

\begin{restatable}{theorem}{theoremFromConsistency}
    \label{thm:from-consistency}
    Let \Cref{assum:main} be fulfilled.
    Let  $\DeBias$ be as in \eqref{eq:De-Bias-hbbeta}. 
    Then, for any $\ba_0$
    with $\|\bSigma^{-1/2}\ba_0\|=1$ such that
    $\langle \ba_0,\bh\rangle^2/\risk \to^{\P}0$,
   \begin{equation*}
    {\sup_{t\in \R}\bigg[\bigg|} 
    \P
    \bigg(
        \frac{\xi_0}{V_0^{1/2}}
        \le 
        t
    \bigg)
    { - \Phi(t)\bigg| + \bigg|}
    \P
\bigg(
    {\frac{ \langle \ba_0, \DeBias \rangle-\theta 
        }{V_0^{1/2}/(n - \df )}}
        \le
        t
    \bigg) - \Phi(t)\bigg|\bigg] \to 0,
    \end{equation*}
    where
    $V_0$ denotes any of the four quantities: 
    $\Var_0[\xi_0]$,
    $\|\by-\bX\hbbeta\|^2$,
    $\widehat{V}(\theta)$
    or $\check V(\ba_0)$.
\end{restatable}

\Cref{thm:from-consistency} is proved in
\Cref{sec:proof-asymptotic-normality}.
The theorem, as well as its variants below,
are obtained
by applying
\Cref{thm:L2-distance-from-normal} 
conditionally on $(\bep,\bX\bQ_0)$
to the function $f(\bz_0)$ in \eqref{def-f-xi}. 
This argument relies on the normality of $\bz_0$ conditionally on $(\bep,\bX\bQ_0)$ 
and thus the Gaussian design assumption. Here is an outline. Define 
\begin{equation}
    \label{definition-eps_0-a_0}
    \delta_1^2(\ba_0)
    \defas
\E_0\big[\|\nabla f(\bz_0)\|_F^2\big]
\big/
\big\{
\E_0\bigl[\|f(\bz_0)\|^2\bigr] +
\E_0\big[\|\nabla f(\bz_0)\|_F^2\big]
\big\}
\end{equation}
where $\E_0$ and $\Var_0$ are the conditional expectation
and conditional variance given $(\bX\bQ_0,\bep)$.
It is sufficient to show that $\delta_1^2(\ba_0)\to^\P 0$ in order to
prove asymptotic normality of $\xi_0/\Var_0[\xi_0]^{1/2}$
by  \Cref{thm:L2-distance-from-normal}
and of $\xi_0/\|f(\bz_0)\|$ by \Cref{thm:consistency-variance}
since $\delta_1^2=\delta_1^2(\ba_0)$
satisfies $2\delta_1^2\ge \max(\eps_1^2,\epsdoublebar_1^2)$ for 
the $\eps_1^2,\epsdoublebar_1^2$ in \Cref{thm:L2-distance-from-normal,thm:consistency-variance}. 
The proof makes rigorous the following informal bound:
$$
\delta_1^2(\ba_0)
=
\frac{
    \E_0[\|\nabla f(\bz_0)\|_F^2]
    }{
    \E_0[\|f(\bz_0)\|^2]
    +
    \E_0[\|\nabla f(\bz_0)\|_F^2]
}
\lesssim \frac{\E_0[\|\bI_n - \hbH\|_F^2 \langle \ba_0,\bh\rangle^2]}{
    C_*^2(\gamma,\mu) n \risk
}
+ O_\P(n^{-1/2})
$$
for some constant $C_*(\gamma,\mu)$,
by establishing a lower bound on
$\|f(\bz_0)\|^2 = \|\by-\bX\hbbeta\|^2$
for the denominator (\Cref{lemma:Deltas-a-b-c-d-strongly-convex,lemma:n-df-strongly-convex,lemma:Delta_n-to-0}),
and by showing that the rank one term $\bw_0(\by-\bX\hbbeta)^\top$
in \eqref{eq:identity-bw0} is negligible in the numerator.
Finally, $\|\bI_n - \hbH\|_F^2 \le n$ always holds by
\Cref{prop:hat-bH} and 
$\langle \ba_0,\bh\rangle^2/\risk$ is shown to be uniformly integrable,
so that the assumption $\langle \ba_0,\bh\rangle^2/\risk\to^\P 0$
grants $\E[\delta_1^2(\ba_0)]\to 0$.
The next two results identify directions $\ba_0$ such that
$\langle \ba_0,\bh\rangle^2/\risk\to^\P 0$ holds.

\begin{restatable}{theorem}{theoremMainResult}
    \label{thm:main-result}
    There exists an absolute constant $C^*>0$ such that
    the following holds.
    Let \Cref{assum:main} be fulfilled, 
    $\DeBias$ be as in \eqref{eq:De-Bias-hbbeta}. 
    Then for any increasing sequence $a_p\to+\infty$
    (e.g., $a_p = \log\log p$), the subset
    \begin{equation}
        \label{set-tilde-S-}
        \overline{S} = \left\{
            \bv\in S^{p-1}:
            \E[\langle \bSigma^{1/2}\bv,\bh\rangle^2/\|\bSigma^{1/2}\bh\|^2]
            \le C^*/a_p
        \right\}
    \end{equation}
    of the unit sphere $S^{p-1}$ in $\R^p$
    has relative volume $|\overline{S}|/|S^{p-1}| \ge 1- 2e^{-p/a_p}$
    and 
    \begin{equation}
        \label{eq:main-thm-conclusion-sup-a_0}
        \begin{split}
    \sup_{\ba_0\in\bSigma^{1/2}\overline{S}}
    &{\ \sup_{t\in\R}}
    {\bigg[}\bigg|
    \P
    \bigg(
        \frac{\xi_0}{ V_0^{1/2} } \le t
    \bigg)
    - \Phi(t)
    \bigg|
    +
    \bigg|
    \P
\bigg(
    \frac{\langle\ba_0, \DeBias-\bbeta\rangle
    }{V_0^{1/2}{/(n - \df )}}
        \le t
   \bigg)
    - \Phi(t)
    \bigg|{\bigg]}
    \to 0
    \end{split}
    \end{equation}
    where $V_0$ denotes any of the four quantities: 
    $\Var_0[\xi_0]$,
    $\|\by-\bX\hbbeta\|^2$,
    $\widehat{V}(\theta)$ or $\check V(\ba_0)$.
    Furthermore, 
    with $\be_j\in\R^p$ the $j$-th canonical basis vector
    and $\phi_{\rm cond}(\bSigma)=\|\bSigma\|_{op}\|\bSigma^{-1}\|_{op}$, 
    the asymptotic normality in \eqref{eq:main-thm-conclusion-sup-a_0} uniformly holds 
    over at least
    $(p-\phi_{\rm cond}(\bSigma)a_p/C^*)$ canonical directions
    in the sense that
    $J_p = \{ j\in[p]: \be_j/\|\bSigma^{-1/2}\be_j\| \in \bSigma^{1/2}\overline{S}\}$ 
    has cardinality
    $|J_p|\ge p-\phi_{\rm cond}(\bSigma)a_p/C^*$. 
\end{restatable}

\Cref{thm:main-result} is proved in
\Cref{sec:proof-asymptotic-normality}.
For a given sequence of directions $\ba_0\in\bSigma^{1/2}S^{p-1}$,
if 
$b_p = \E[\langle \ba_0,\bh\rangle^2/\|\bSigma^{1/2}\bh\|^2]\to 0$
then it follows by choosing
$a_p = C^*/b_p$ that $\ba_0\in \bSigma^{1/2}\overline{S}$
for the $\overline{S}$ in \eqref{set-tilde-S-}
so that \eqref{eq:main-thm-conclusion-sup-a_0}
implies that asymptotic normality holds for this sequence of $\ba_0$.
In other words, asymptotic normality holds for all $\ba_0$
such that
$\E[\langle \ba_0,\bh\rangle^2/\|\bSigma^{1/2}\bh\|^2]\to 0$.
Thus a sequence of directions $\ba_0$ for which asymptotic normality
does not follow from \Cref{thm:main-result}
is a sequence such that
$\E[\langle \ba_0,\bh\rangle^2/\|\bSigma^{1/2}\bh\|^2]$
does not vanish, i.e., the squared error $\langle \ba_0,\bh\rangle^2$
in direction $\ba_0$ carries
a constant fraction of the full prediction error 
$\|\bSigma^{1/2}\bh\|^2$.
Such direction $\ba_0$ must thus be very special,
which is embodied by the exponentially small bound on the
relative volume
$|\overline{S}\setminus S^{p-1}|/|S^{p-1}|$.

\begin{restatable}{theorem}{theoremGisANorm}
    \label{thm:g-is-a-norm}
    Under \Cref{assum:main} there exists
    $\Omega_n$ with $\P(\Omega_n^c)\le C_0(\gamma,\mu)n^{-1/2}$
    and
    \begin{equation}
        \label{eq:marginal-a0}
        \E\bigl[I_{\Omega_n}
        \bigl({\langle \ba_0,\hbbeta-\bbeta\rangle} 
        +
    (n-\df)^{-1}
    \bz_0^\top(\by-\bX\hbbeta)
    \bigr)^2\bigr]
    \le
    \risk{}  \C(\gamma,\mu) /n
    \end{equation}
    If additionally
    $g$ is a seminorm 
    then
    $
    |\bz_0^\top(\by-\bX\hbbeta)|/n
    =
    |\ba_0^\top\bSigma^{-1} \bX^\top(\by-\bX\hbbeta)|/n
    \le 
    g(\bSigma^{-1}\ba_0)
    $ always holds
    by properties of the subdifferential of a norm.
    Consequently,
    if $g(\bSigma^{-1}\ba_0)^2/\risk\to 0$
    then
    $\langle \ba_0,\bh\rangle^2/\risk \to^\P 0$
    and the conclusions of \Cref{thm:from-consistency}
    hold.
\end{restatable}

\Cref{thm:g-is-a-norm} is proved in
\Cref{sec:proof-asymptotic-normality}.
The first part of the theorem says that 
the estimation error $\langle \ba_0,\bh\rangle$
is essentially $-\langle \bz_0,\by-\bX\hbbeta\rangle/(n-\df)$ 
up to an error term of order $\risk O_\P(n^{-1/2})$,
so that $\langle \ba_0,\bh\rangle^2/\risk \to^\P 0$
if and only if $(n-\df)^{-2}\langle \bz_0,\by-\bX\hbbeta\rangle^2/\risk\to^\P 0$.
Combined with the fact that $(1-\df/n)$ is bounded away from 0
by \Cref{lemma:n-df-strongly-convex}, this implies that
$\langle \ba_0,\bh\rangle^2/\risk \to^\P 0$
if and only if $n^{-2}\langle \bz_0,\by-\bX\hbbeta\rangle^2/\risk\to^\P 0$.
The last part of the theorem relies on the property
$g(\bu) \le \sup_{\bs\in\partial g(\bu)} \bu^\top\bs$
for any norm $g$
where $\partial g(\bu)$ is the subdifferential of $g$ at $\bu$.
This property also holds if $g$ is a semi-norm, however
it is unclear how to extend the last part of the above
theorem if $g$ is not a 
semi-norm.

For the Lasso, the penalty function is $g(\bb)=\lambda\|\bb\|_1$
and condition $g(\bSigma^{-1}\ba_0)^2/\risk \to 0$ becomes
$\lambda^2 \|\bSigma^{-1}\ba_0\|_1^2 / \risk \to 0$.
Typically, the tuning parameter $\lambda$ is chosen as
$\lambda \propto \sigma(2\log(p/k)/n)^{1/2}$ with $k=1$ 
\cite[among others]{bickel2009simultaneous} or 
$k=s_0$ 
\cite{sun2013sparse,lecue2015regularization_small_ball_I,bellec2016slope,feng2019sorted,bellec2018nb_lsb}, 
where $s_0 = \|\bbeta\|_0$. 
For such choices, the condition
$\lambda^2 \|\bSigma^{-1}\ba_0\|_1^2 /\risk\to 0$ can be 
written as $\|\bSigma^{-1}\ba_0\|_1 = (\risk/\sigma^2)^{1/2} o(\sqrt{n/\log(p/k)})$ 
and since $\risk\ge \sigma^2$, a sufficient condition is
$\|\bSigma^{-1}\ba_0\|_1 = {o(\sqrt{n/\log(p/k)})}$.
If $\ba_0 = \be_j$ is a vector of the canonical basis,
the normalization \eqref{normalization-a_0} gives $(\bSigma^{-1})_{jj}=1$
and $\|\bSigma^{-1}\be_j\|_1$ is the $\ell_1$ norm of the $j$-th column of $\bSigma^{-1}$.
The condition $\|\bSigma^{-1}\ba_0\|_1 ={o(\sqrt{n/\log(p/k)})}$
allows, for instance, the $j$-th column of $\bSigma^{-1}$ to have
{$o(\sqrt{n/\log(p/k)})$ constant entries.}
This assumption is weaker than that of some previous studies; for instance
\cite{javanmard2018debiasing} requires $\|\bSigma^{-1}\ba_0\|_1 =O(1)$ 
for $\ba_0=\bfe_j$.
The following example illustrates the benefit of picking a proper penalty level $\lam$. 

\begin{example} Let $p/n\to\gamma<1$ and $g(\bb)=\lambda\|\bb\|_1$. 
\begin{enumerate}
    \item 
    For $\lam=0$, the Lasso and de-biased Lasso are both identical 
    to the least squares estimator 
    and the de-biasing correction proportional to $\bz_0^\top(\by-\bX\hbbeta)$ is 0 since $\bX^\top(\by-\bX\hbbeta)=0$,
    so that 
    $\htheta - \theta = \langle \ba_0,\bh\rangle = \ba_0^\top(\bX^\top\bX)^{-1}\bX^\top\bep$ in \eqref{def-hat-theta},
    $\df=p$, 
    $\widehat{V}(\theta)\approx \|\by-\bX\hbbeta\|^2 \sim \sigma^2\chi^2_{n-p}$ and 
    $\sqrt{n}(\htheta - \theta) \xrightarrow{d} N\big(0,\sigma^2/(1-\gamma)\big)$.
    \item
    Suppose $|\Shat|/n+\|\bX\bh\|^2/n+\|\bSigma^{1/2}\bh\|^2=o_\P(1)$ for suitable 
    $\lam\ge \sigma \sqrt{2\log(p/s_0)/n}$ as in  
    \cite{ZhangH08,bellec_zhang2019dof_lasso}. 
    Then, $\widehat{V}(\theta)=(1+o_\P(1))n\sigma^2$ and 
    $\sqrt{n}(\htheta - \theta) \xrightarrow{d} N\big(0,\sigma^2\big)$. 
\end{enumerate}
\end{example} 

\subsection{Confidence intervals}
\Cref{thm:from-consistency,thm:main-result,thm:g-is-a-norm}
are valid for any choice of the variance estimate
among $\|\by-\bX\hbbeta\|^2,
\widehat{V}(\theta)$ in \eqref{def-hat-V-theta}
and $\check V(\ba_0)$ in \eqref{eq:check-variance-1}
for directions $\ba_0$ such that
$\langle \ba_0,\bh\rangle^2/\risk \to^\P 0$ holds.
For such direction $\ba_0$,
the choice $\|\by-\bX\hbbeta\|^2$ leads to the narrowest
confidence interval for $\theta$, namely
\begin{equation}
    \P\bigl(~
\theta\in
\bigl[
\langle \ba_0,\DeBias\rangle
\pm z_{\alpha/2} (n-\df)^{-1}\|\by-\bX\hbbeta\|
\bigr]
~
\bigr)
\to 1-\alpha
\label{eq:narrow-confidence-interval}
\end{equation}
where $[u\pm v]$ denotes the interval $[u-v, u+v]$,
$\P(|N(0,1)|> z_{\alpha/2}) = \alpha$
and $\DeBias$ is the de-biased estimator in \eqref{eq:De-Bias-hbbeta}.
The choice $\check V(\ba_0)$ leads to intervals
with larger multiplicative coefficient for $z_{\alpha/2}$, namely
\begin{equation}
\theta\in
\bigg[
\langle \ba_0,\DeBias\rangle
\pm z_{\alpha/2}
\bigg(
    \frac{\|\by-\bX\hbbeta\|^2}{(n-\df)^2}
    +
    \frac{\|\bI_n-\hbH\|_F^2\langle \bz_0,\by-\bX\hbbeta\rangle^2}{
        (n-\df)^4
    }
\bigg)^{1/2} 
\bigg]
\label{eq:CI-with-variance-spike}
\end{equation}
has probability converging to $1-\alpha$
for directions $\ba_0$ satisfying any of the above theorems.
For such directions,
the choice $\widehat{V}(\theta)$ justifies
confidence intervals of the form \eqref{CI-introduction}
as 
\begin{equation}
    \bigl(
    (n-\df)
    (\langle \ba_0,\hbbeta\rangle - \theta)
    + \langle \bz_0,\by-\bX\hbbeta\rangle
    \bigr)^2
    -
    \widehat{V}(\theta) z_{\alpha/2}^2
    \le 0
\label{eq:event-CI-sec3}
\end{equation}
holds with probability converging to $1-\alpha$.
Given the expression for
$\widehat{V}(\theta)$ in \eqref{def-hat-V-theta},
the left-hand side of \eqref{eq:event-CI-sec3} is a quadratic polynomial
in $\theta$ with dominant coefficient
$(n-\df)^2 - z_{\alpha/2}\|\bI_n-\hbH\|_F^2$.
Since $\|\bI_n-\hbH\|_F^2\le n-\df$ almost surely by properties of $\hbH$
in \Cref{prop:hat-bH}
and $n-\df \ge C_*(\gamma,\mu) n$ for  some constant $C_*(\gamma,\mu)$
with probability approaching one
by \Cref{lemma:n-df-strongly-convex},
in the same event the dominant coefficient is positive.
The intersection of events \eqref{eq:event-CI-sec3}
and $\{n-\df \ge C_*(\gamma,\mu) n\}$ has probability converging to $1-\alpha$
and in this event,
$\theta\in [\Theta_1(z_{\alpha/2}),\Theta_2(z_{\alpha/2}))]$
where $\Theta_1(z_{\alpha/2}),\Theta_2(z_{\alpha/2})$ are the two
real roots of the left-hand side of \eqref{eq:event-CI-sec3} as a quadratic function of $\theta$.

\subsection{Variance spike}
One can pick any choice among the three variance estimates
in \Cref{thm:from-consistency}
because it assumes $\langle \ba_0,\bh\rangle^2/\risk\to^\P 0$
and this limit in probability implies both
$\widehat{V}(\theta)/\|\by-\bX\hbbeta\|^2 \to^\P 1$
and
$\check V(\ba_0)/\|\by-\bX\hbbeta\|^2 \to^\P 1$.
These limits in probability to 1
are made rigorous by \eqref{eq:check-variance-1}
and by the lower bound
$\|\by-\bX\hbbeta\|^2 \ge 
\risk n \bigl(C_*^2(\gamma,\mu) - O_\P(n^{-1/2})\bigr)$
obtained from 
\Cref{lemma:Delta_n-to-0,lemma:n-df-strongly-convex,lemma:Deltas-a-b-c-d-strongly-convex} as explained in \eqref{eq:imply-variance-consistency}
of the proof.

The reason that the estimates $\widehat{V}(\theta)$
and $\check V(\ba_0)$ were introduced is that
the simpler estimate $\|\by-\bX\hbbeta\|^2$
is not asymptotically unbiased for $\Var_0[\xi_0]$
for directions $\ba_0$ such that
$\langle \ba_0,\bh\rangle^2/\risk$ does not converge to 0
in probability:
While the relative bias of $\widehat{V}(\theta)$
and $\check V(\ba_0)$ provably converges to 0 in
\Cref{lemma:delta-0} and \eqref{eq:check-variance-1}
for all directions $\ba_0$,
the same cannot be said for the simpler estimate
$\|\by-\bX\hbbeta\|^2$.

\begin{restatable}{theorem}{theoremVarianceConsistentIFF}
    \label{theorem:variance-iff}
    Let \Cref{assum:main} be fulfilled. Then the following are equivalent:
    \begin{multicols}{2}
    \begin{enumerate}
\item
    $\|\by-\bX\hbbeta\|^2/\Var_0[\xi_0]\to^\P 1$,
        \item 
    $\E_0[\|\by-\bX\hbbeta\|^2]/\Var_0[\xi_0]\to^\P 1$,
        \item 
    $\langle \ba_0,\bh\rangle^2/\risk \to^\P 0$,
\item
    $\langle \ba_0,\bh\rangle^2 n/\|\by-\bX\hbbeta\|^2 \to^\P 0$,
\item
    $\langle \bz_0,\by-\bX\hbbeta\rangle^2 /(n\|\by-\bX\hbbeta\|^2) \to^\P 0$,
\item
    $\widehat{V}(\theta)/\|\by-\bX\hbbeta\|^2\to^\P 1$,
\item
    $\check{V}(\ba_0)/\|\by-\bX\hbbeta\|^2\to^\P 1$.

\item[\vspace{\fill}] 
    \end{enumerate}
    \end{multicols}
\end{restatable}

\Cref{theorem:variance-iff} is proved in
\Cref{sec:proof-variance}.
It shows that for the directions $\ba_0$
such that 
$\langle \ba_0,\bh\rangle^2 n /\|\by-\bX\hbbeta\|^2 \to^\P 0$ does not hold,
e.g., directions such that the error $\langle \ba_0,\bh\rangle^2$
is of the same order as the average squared residual $\|\by-\bX\hbbeta\|^2/n$ 
(see Lemma \ref{lemma:moment-equivalence} in Section \ref{sec:loss-equiv}),
the simpler estimate $\|\by-\bX\hbbeta\|^2$
fails to account for a non-negligible part of the variance
$\Var_0[\xi_0]$ by item (i) 
above.
The goal of the estimates $\widehat{V}(\theta)$ and $\check V(\ba_0)$
is to repair this
as $\E_0[\widehat{V}(\theta)]/\Var_0[\xi_0]\to^\P 1$
and
$\E_0[\check{V}(\ba_0)]/\Var_0[\xi_0]\to^\P 1$
hold for all directions by \Cref{lemma:delta-0,lemma:check-V-variance},
even for directions $\ba_0$ such that (i)-(vii) above fail.
Note that the quantity
$\langle \bz_0,\by-\bX\hbbeta\rangle^2 /(n\|\by-\bX\hbbeta\|^2)$
in item (v) is observable (i.e., does not depend on $\bbeta$), so that
(i)-(vii) 
are expected to hold when 
this quantity is sufficiently small.

For directions $\ba_0$ such that
$\langle \ba_0,\bh\rangle^2 n /\|\by-\bX\hbbeta\|^2 \to^\P 0$
does not hold, we expect
a variance spike, i.e., an extra additive term in the variance
estimate equal to
$\|\bI_n-\hbH\|_F^2\langle \ba_0,\bh\rangle^2$ in
$\widehat{V}(\theta)$ and to
$\|\bI_n-\hbH\|_F^2\langle \bz_0, \by-\bX\hbbeta\rangle^2/(n-\df)^2$ in
$\check V(\ba_0)$.
The confidence interval
\eqref{eq:narrow-confidence-interval}
that does not take into account this variance spike is expected to 
be too narrow and to suffer from undercoverage
for directions $\ba_0$ with large $\langle \ba_0,\bh\rangle^2
n /\|\by-\bX\hbbeta\|^2$.
The wider confidence interval \eqref{eq:CI-with-variance-spike}
is expected to repair this, although
for directions $\ba_0$
such that $\langle \ba_0,\bh\rangle^2 n\|\by-\bX\hbbeta\|^2 \to^\P 0$
does not hold
the current theory does not prove whether the asymptotic distribution
is normal.
The theoretical evidence that the variance spike occurs
is grounded in the relative asymptotic unbiasedness
of $\widehat{V}(\theta)$ in \eqref{eq:delta-0-to-0}
and of $\check V(\ba_0)$ in \eqref{eq:check-variance-1},
combined with the negative result for the simpler variance
estimate $\|\by-\bX\hbbeta\|^2$ 
via \Cref{theorem:variance-iff} as discussed above. 
\Cref{figure:lasso-variance} illustrates the variance
spike on simulations for the Lasso and direction $\ba_0$
proportional to the first canonical basis vector.

\begin{figure}[b]
    \centering
\begin{tabular}{|c|c|c|c|c|}
    \hline
    $\lambda$
    & $\frac{(n-\df)\langle \ba_0,\DeBias -\bbeta\rangle}{\|\by-\bX\hbbeta\|}$
    & $\frac{(n-\df)\langle \ba_0,\DeBias -\bbeta\rangle}{\hat V(\theta)^{1/2}}$
    & $\|\bSigma^{1/2}(\hbbeta-\bbeta)\|^2$
    & $\langle \ba_0,\hbbeta-\bbeta\rangle^2$
    \\
    \hline
    0.005
    & \includegraphics[width=1.2in]{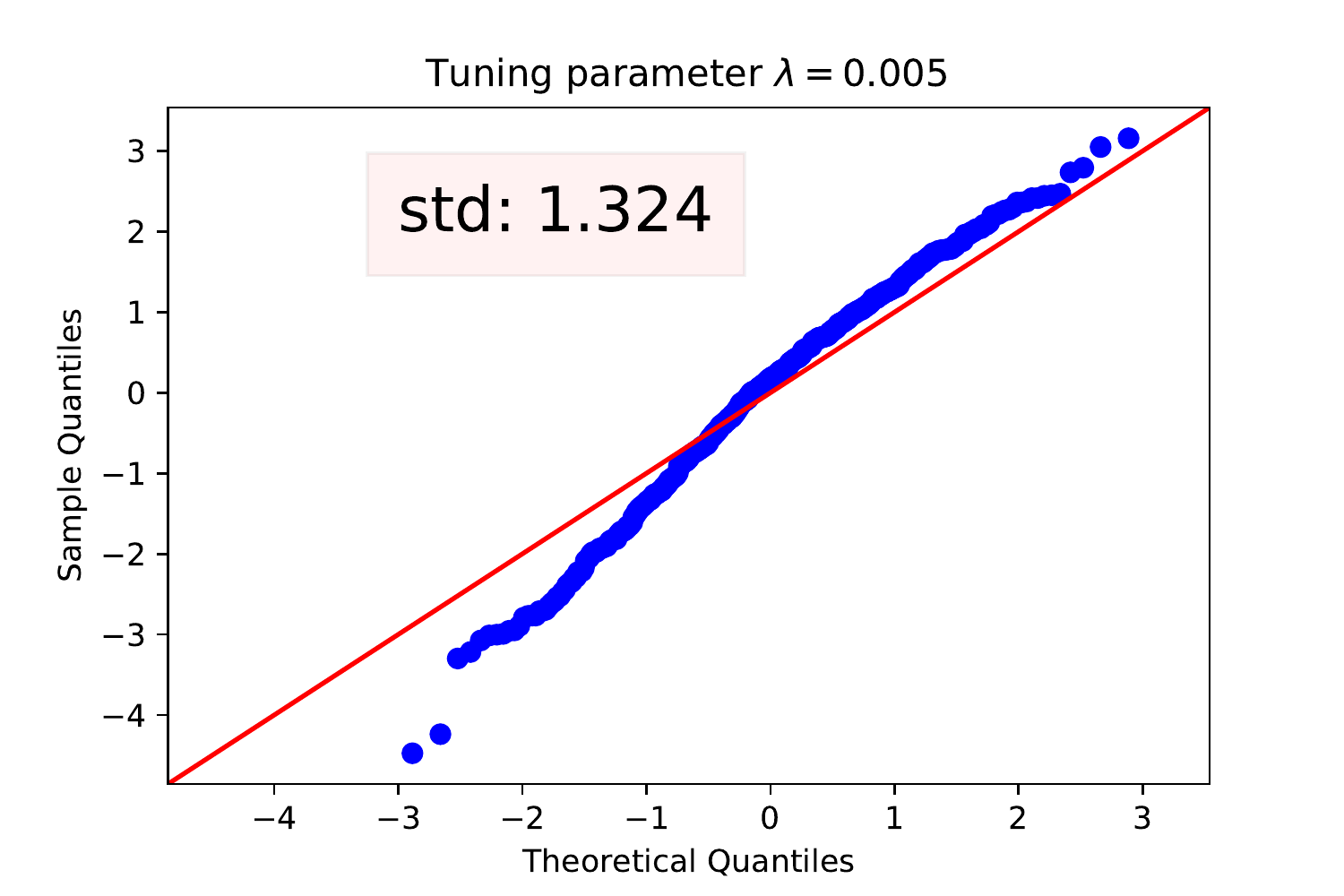}
    & \includegraphics[width=1.2in]{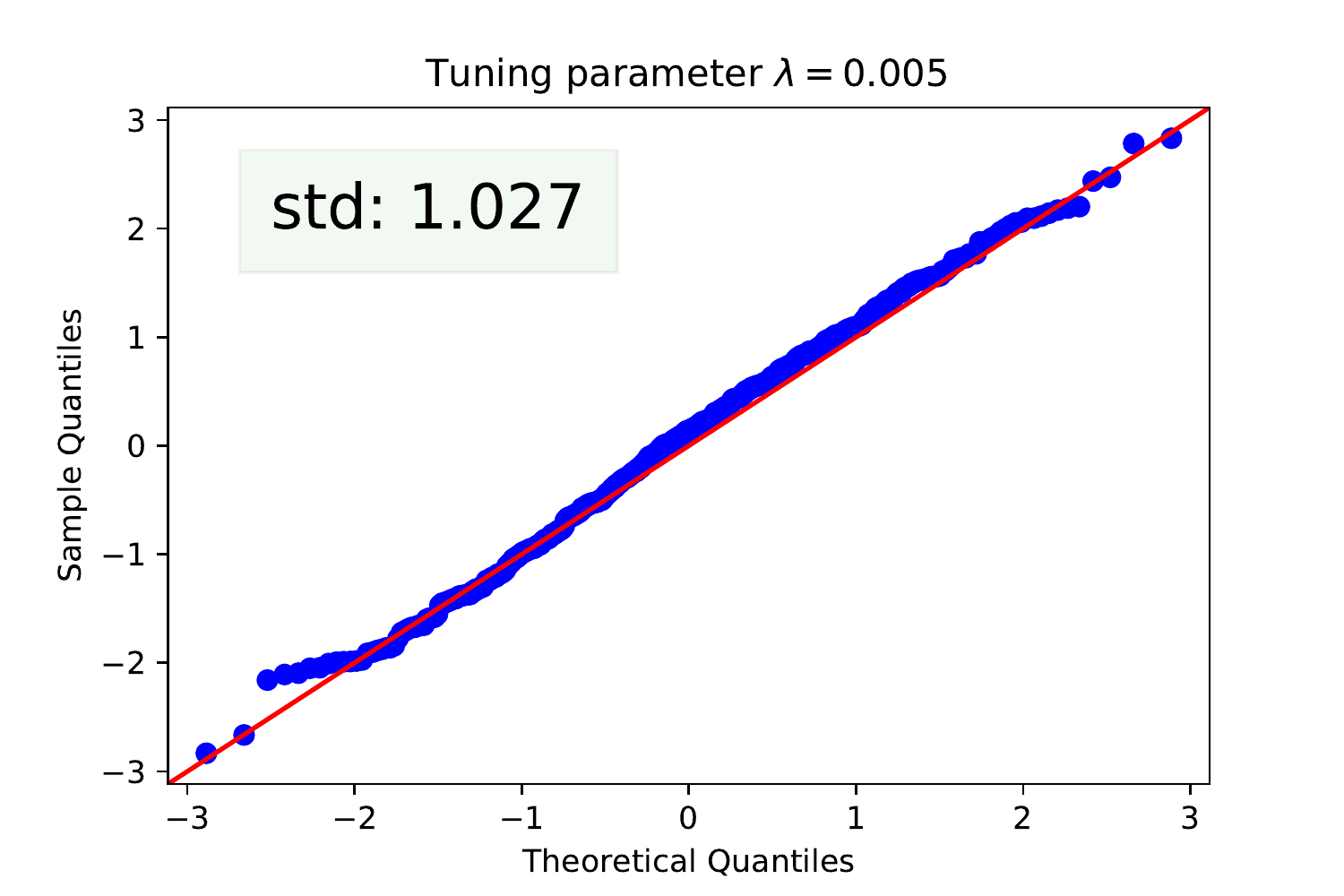}
    & 3.23 $\pm$0.45
    & 0.32 $\pm$0.11
    \\
    \hline
    0.01
    & \includegraphics[width=1.2in]{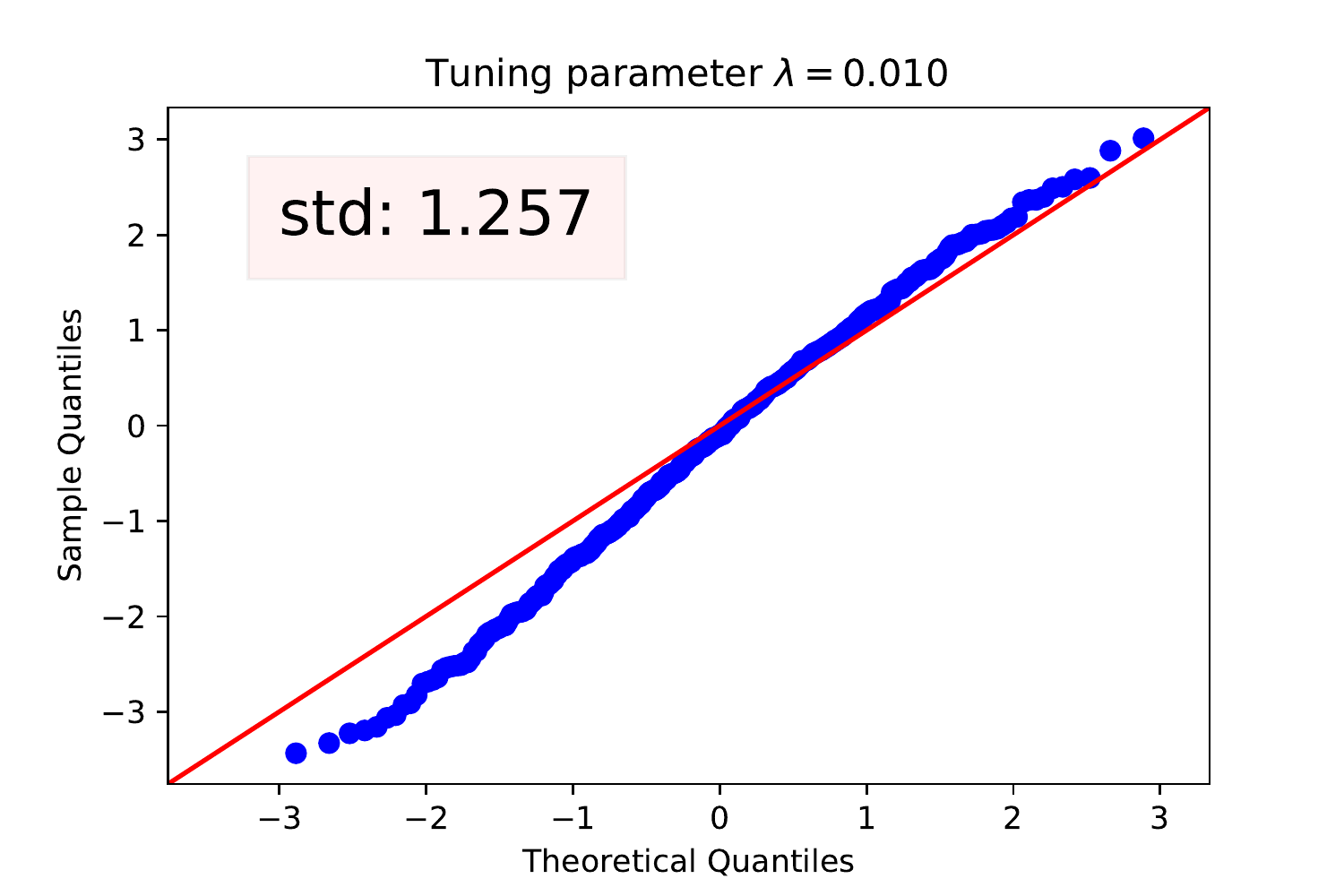}
    & \includegraphics[width=1.2in]{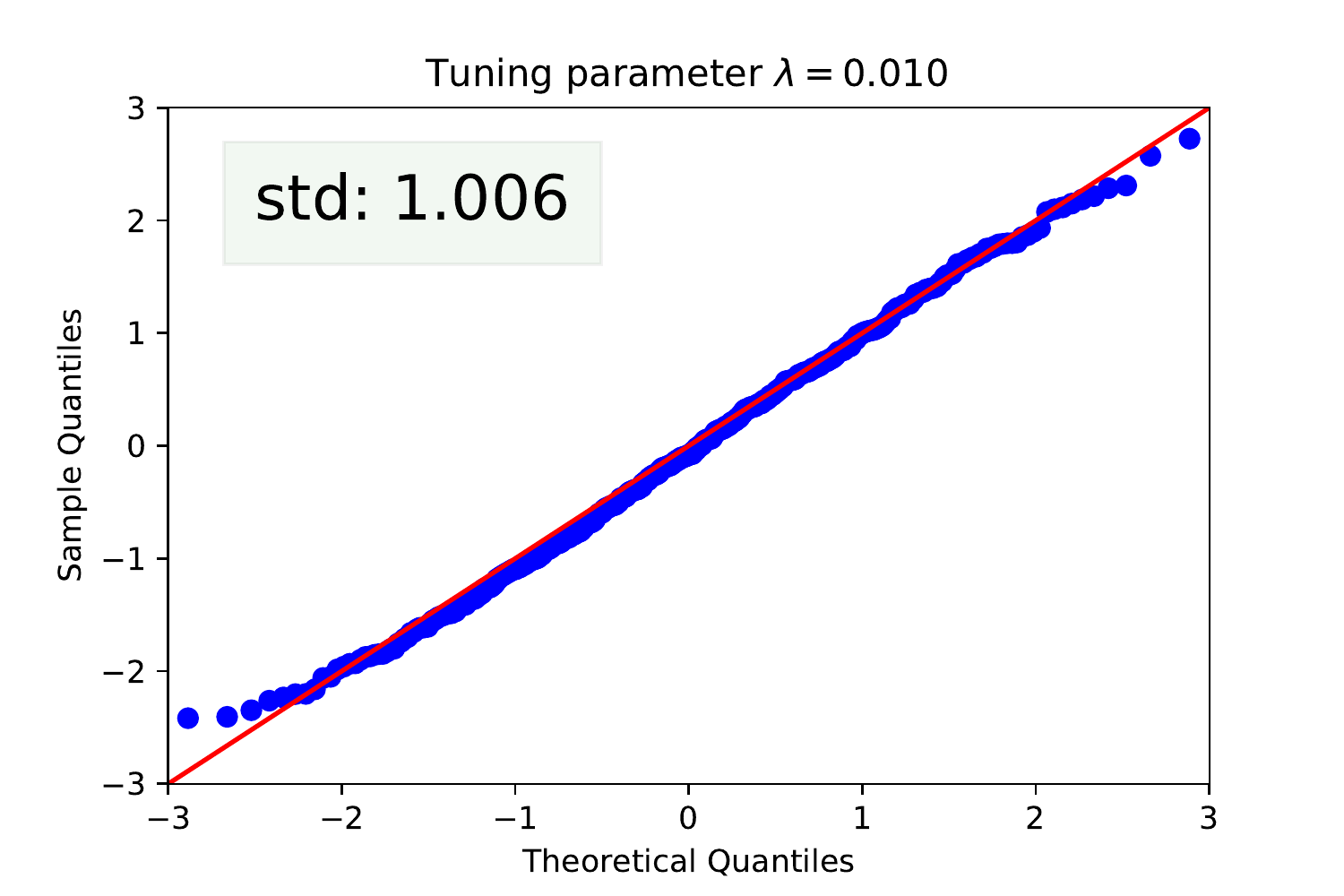}
    & 2.62 $\pm$0.41
    & 0.40 $\pm$0.12
    \\
    \hline
    0.05
    & \includegraphics[width=1.2in]{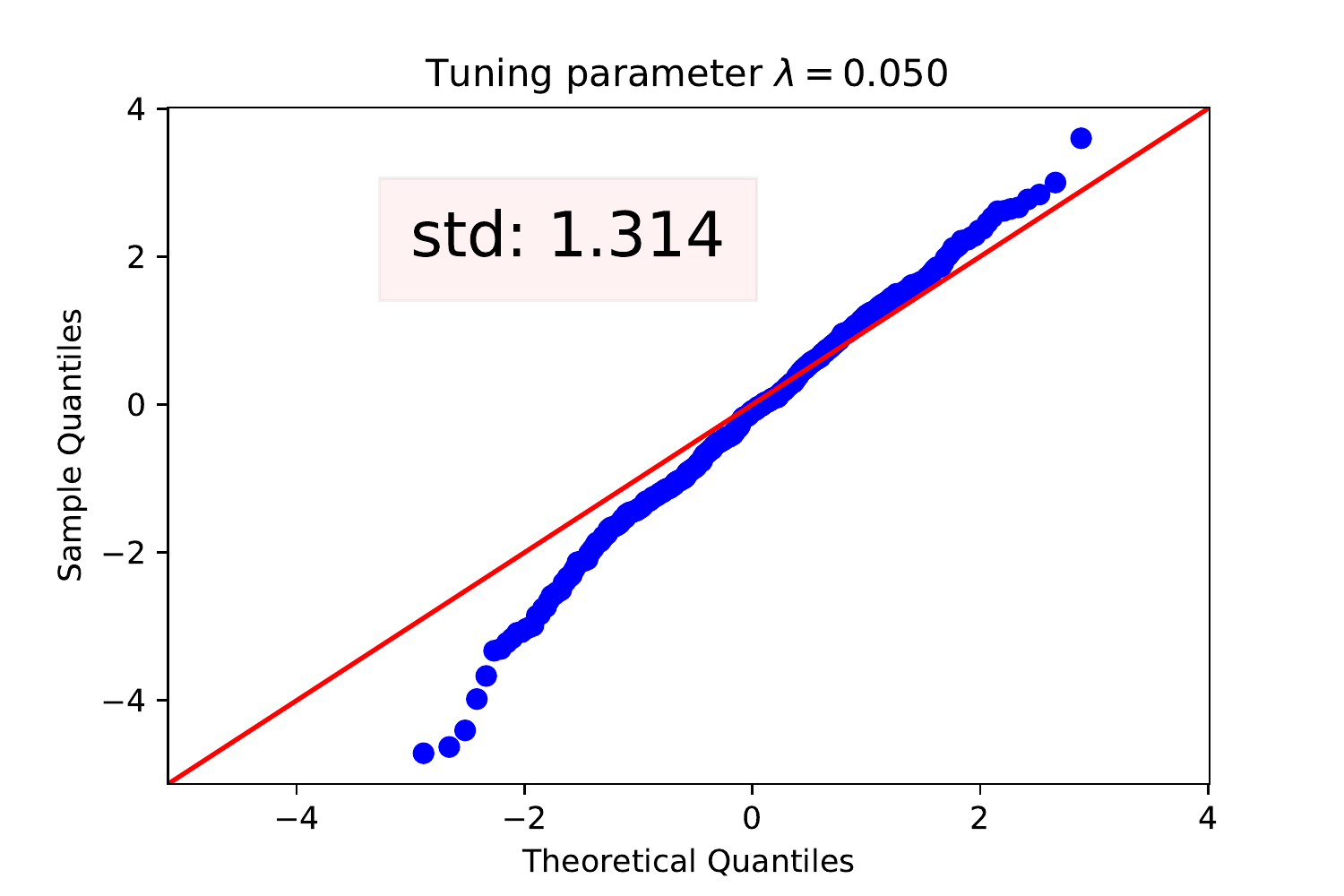}
    & \includegraphics[width=1.2in]{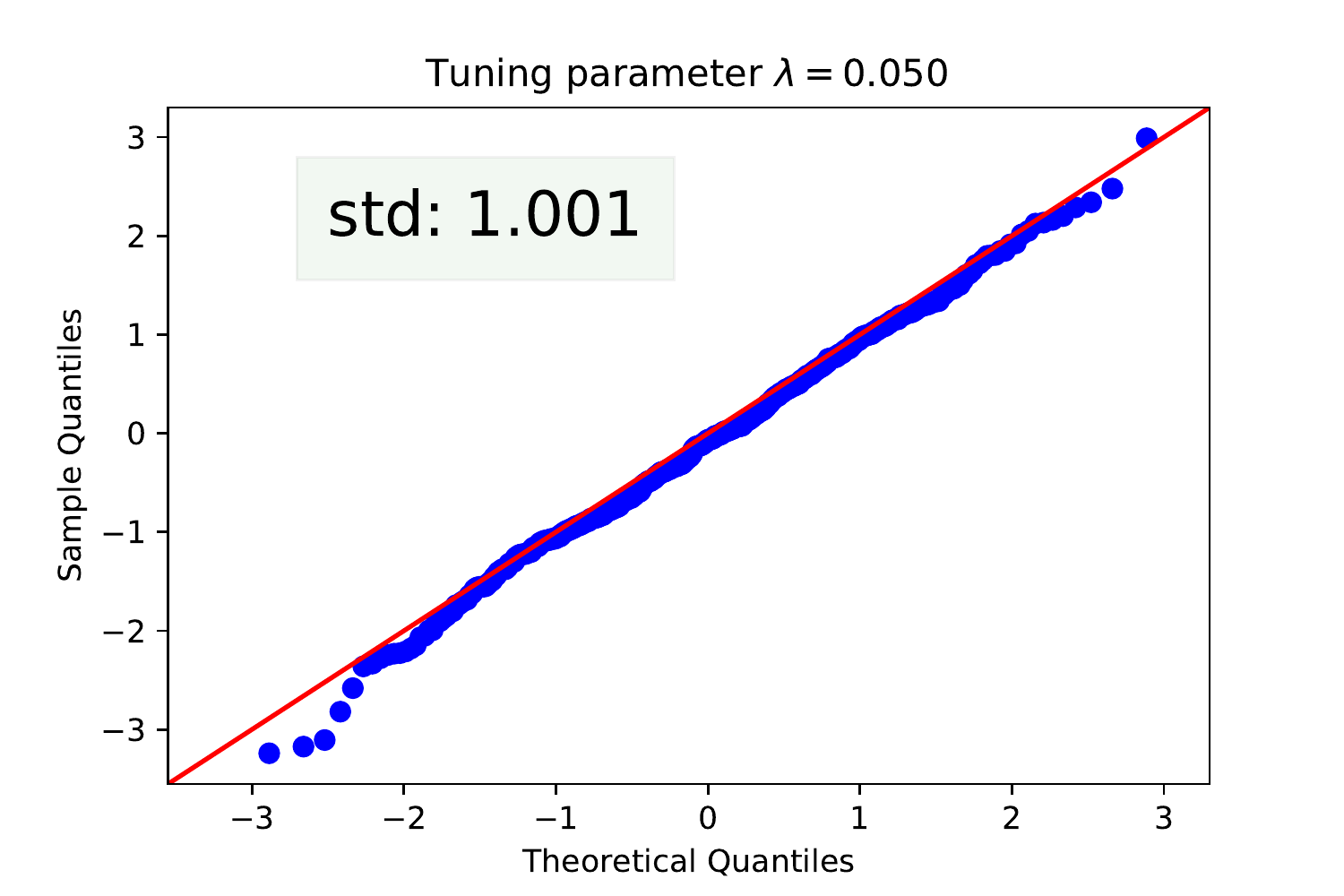}
    & 4.39 $\pm$0.86
    & 1.81 $\pm$0.39
    \\
    \hline
    0.1
    & \includegraphics[width=1.2in]{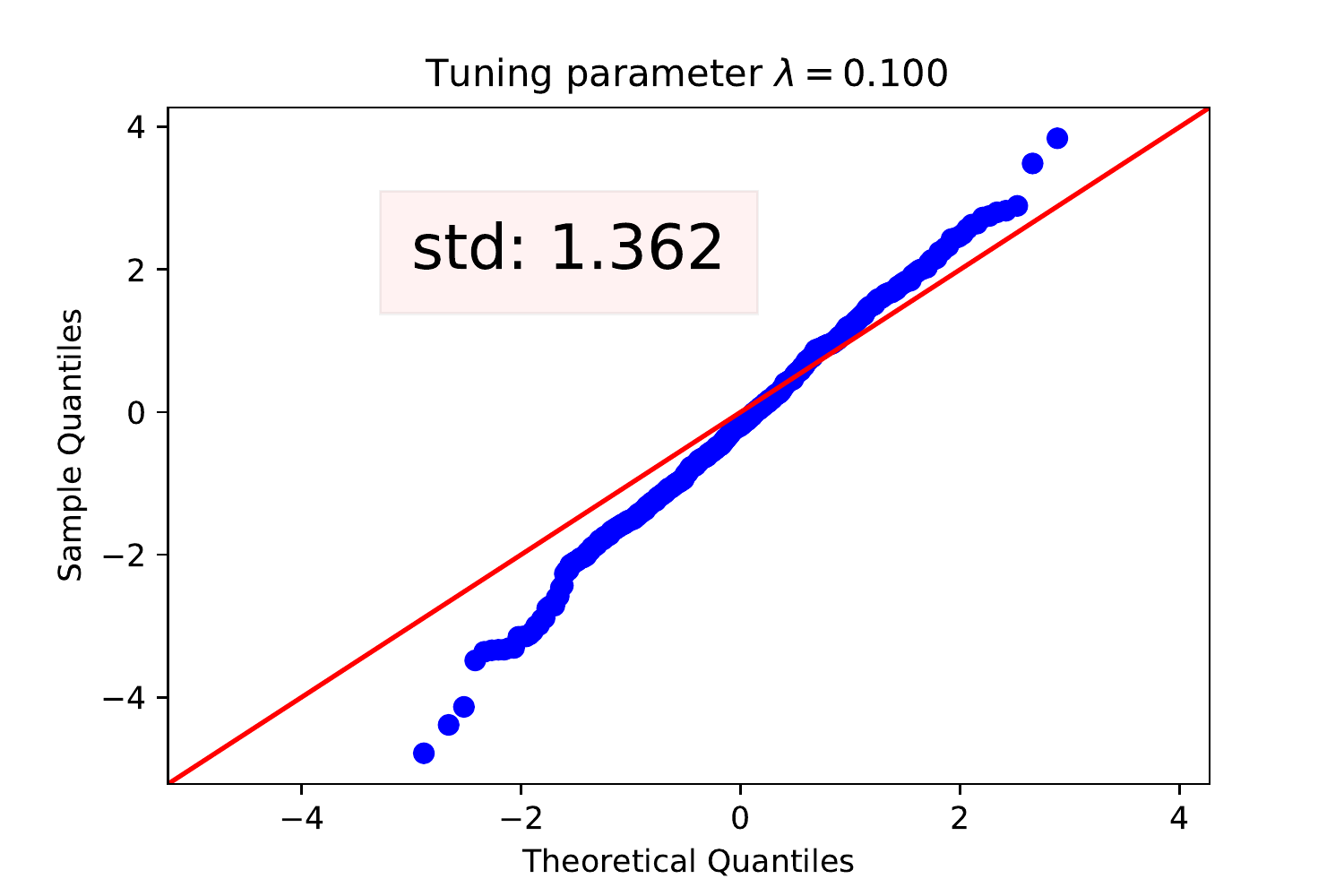}
    & \includegraphics[width=1.2in]{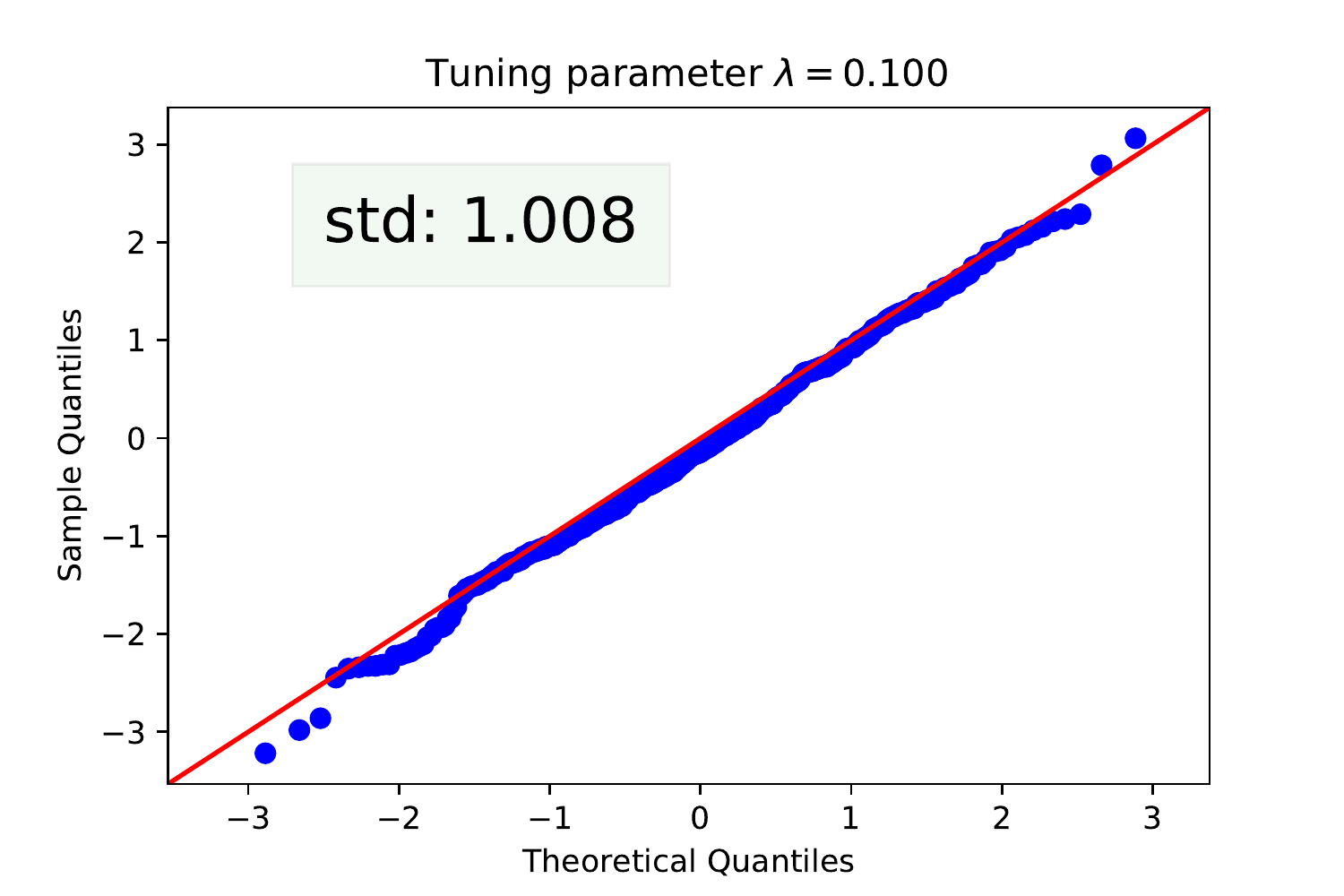}
    & 11.70 $\pm$2.40
    & 5.6 $\pm$1.14
    \\
    \hline
\end{tabular}
\caption{
    \small
    Standard normal QQ-plots of $\langle \ba_0,\DeBias-\bbeta\rangle /V_0^{1/2}$ with $V_0=\|\by-\bX\hbbeta\|^2$ (second column)
    and with $V_0 = \hat V(\theta)=\|\by-\bX\hbbeta\|^2 + (n-\df)\langle \ba_0,\hbbeta-\bbeta\rangle^2$ (third column),
    prediction error $\|\bSigma^{1/2}(\hbbeta-\bbeta)\|^2$ (fourth column)
    and squared estimation error in direction $\ba_0$ (fifth column)
    for the Lasso \eqref{lasso} for each $\lambda\in\{0.005,0.01, 0.05, 0.1\}$
    in the following setting:
    $(s,n,p,\sigma^2)=(200, 750, 1000,1.0)$,
    coefficient vector $\bbeta$ is $s$-sparse with $\beta_1 = 20$,
    $\beta_j = \pm 1$ for $j=2,...,s$ (independent random signs),
    $\beta_j=0$ for $j>s$;
    covariance matrix  
    $ \bSigma^{-1} = \bI_p + 0.9 s^{-1/2} (\be_1\sgn(\bbeta)^\top + \sgn(\bbeta)\be_1^\top)$,
    direction
    $\ba_0 = \be_1 / (\be_1^\top \bSigma^{-1}\be_1)^{1/2}$
    for $\be_1\in\R^p$ the first canonical basis vector.
    512 repetitions were used and $(\bSigma,\bbeta)$ are the same
    across these repetitions.
    We see that
    $V_0=\|\by-\bX\hbbeta\|^2$ yields an empirical standard deviation
    (std)
    substantially larger than 1 (second column), whereas using
    $V_0=\hat V(\theta)$ repairs this with an std
    close to 1 (third column).
    This choice of $(\bSigma,\bbeta)$ is a minor modification
    of \cite[Example 2.1 and Figure 2]{bellec_zhang2019dof_lasso}:
    it is constructed so that $\langle \ba_0,\hbbeta-\bbeta\rangle^2$
    captures a substantial fraction of the full prediction error
    $\|\bSigma^{1/2}(\hbbeta-\bbeta)\|^2$.
}
\label{figure:lasso-variance}
\end{figure}

The second term in the variance estimates
\eqref{def-hat-V-theta} and \eqref{eq:check-V-variance-def} 
is necessary for certain $\ba_0$
for the estimate $\hbbeta=\mathbf{0}$
which corresponds to \eqref{eq:intro-hbbeta-g}
with penalty $g(\mathbf{0})=0$ and $g(\bb)=+\infty$ for $\bb\ne \mathbf{0}$.
For $\bSigma=\bI_p$, $\ba_0=\bbeta/\|\bbeta\|$ and $V_*=2\|\bbeta\|^2+\sigma^2$, 
\begin{equation}
\frac{- \xi_0}{\widehat{V}(\theta)}
= \frac{ - n \langle \ba_0,\bbeta\rangle + \bz_0^\top\by}{
    (\|\by\|^2 + n \langle \ba_0, \bbeta\rangle^2)^{1/2}
}
= \frac{{(nV_*)^{-1/2}}\sum_{i=1}^n
    (
    {-\|\bbeta\|} + z_{0i}\eps_i + \|\bbeta\| z_{0i}^2
    )
    }{
   (\|\by\|^2/n + \|\bbeta\|^2)^{1/2}{/V_*^{1/2}}
}.
\label{zero-estimator-variance}
\end{equation}
By the CLT, the numerator of the rightmost
quantity converges to $N(0,1)$ 
and in the denominator, 
$(\|\by\|^2/n + \|\bbeta\|^2)/V_* \to^\P 1$ 
by the weak law of large numbers, so that \eqref{zero-estimator-variance}
converges in law to $N(0,1)$ by Slutsky's theorem.
On the other hand, 
if the variance estimate $\|\by-\bX\hbbeta\|^2$ is used instead
of $\widehat{V}(\theta)$ in the denominator, 
the CLT
$-\xi_0/\|\by\| \to^d N(0,V_*/(\|\bbeta\|^2 + \sigma^2))$ still holds but 
the asymptotic variance is
$1 + (1+\sigma^2/\|\bbeta\|^2)^{-1}>1$.

\subsection{Relaxing strong convexity when $p>n$}
\label{sec:relaxing-strong-convex-lasso-GL}

The previous theorems are valid under \Cref{assum:main}:
Either $\gamma<1$ and $g$ is an arbitrary convex function,
or $\gamma\ge 1$ and $g$
is required to be strongly convex with parameter $\mu$.
If $p/n>1$ and the penalty $g$ is not strongly convex,
the techniques of the present paper still provide
asymptotic normality results under additional assumptions
as shown in the following result.

Consider either the Lasso
\bel{lasso}
\hbbeta = \argmin_{\bb\in\R^p}\left\{\|\by - \bX\bb \|^2/(2n) + \lam \|\bb\|_1\right\} 
\eel
for some $\lambda>0$ or the group Lasso norm $\|\cdot\|_{GL}$
and group Lasso $\hbbeta$ defined as
\begin{equation}
\label{group-lasso}
    \hbbeta = \argmin_{\bb\in\R^p}\left\{\|\by - \bX\bb\|^2/(2n) + \|\bb\|_{GL}\right\},
    \quad
    \|\bb\|_{GL} = \sum_{k=1}^K \lambda_k\|\bb_{G_k}\|_{2} 
\end{equation}
where $(G_1,...,G_K)$ is a partition of $\{1,...,p\}$ into $K$ non-overlapping groups
and $\lambda_1,...,\lambda_K>0$ are tuning parameters.
If each $G_k$ is a singleton and $\lambda_k=\lambda>0$ for all $k=1,...,p$, 
then \eqref{group-lasso} reduces to the Lasso \eqref{lasso}.

\begin{restatable}{theorem}{theoremGroupLassoPLargerN}
    \label{thm:lasso-p-larger-than-n}
    Let $\gamma\ge 1, \kappa\in (0,1)$ be constants independent of $\{n,p\}$.
    Consider a sequence of regression problems 
    with $p/n\le \gamma$ and
    invertible $\bSigma$.
    Assume that 
    the group Lasso estimator $\hbbeta$ in \eqref{group-lasso} satisfies  
    \begin{equation}
    \P\bigl(\|\hbbeta\|_0 \le \kappa n/2\bigr)
    \to 1.
    \label{eq:assum-lasso}
    \end{equation}
    If $\ba_0$ is such that 
        $\|\bSigma^{-1/2}\ba_0\|=1$ and
    $\langle \ba_0,\hbbeta-\bbeta\rangle^2/\risk\overset{\P}{\to} 0$
    for {the} $\risk$ in \eqref{oracle-general-g-strongly-convex} then 
    \begin{equation}
        \label{eq:thm_lasso-p-bigger-n}
        \sup_{t\in\R}
    \big|\P\bigl(
        \|\by-\bX\hbbeta\|^{-1}
    \big(
        (n-\df)\langle \ba_0,\hbbeta-\bbeta\rangle  + \bz_0^\top(\by-\bX\hbbeta)
    \big)
    \le t
    \bigr)
    - \Phi(t)
    \big|\to 0.
    \end{equation}
    Furthermore, for any $a_p$ with $a_p\to\infty$
    and $\overline{S}$ in \eqref{set-tilde-S-},
    the relative volume bound given after \eqref{set-tilde-S-} holds,
    and the asymptotic normality 
    \eqref{eq:thm_lasso-p-bigger-n} holds uniformly
    over all $\ba_0\in \bSigma^{1/2}\overline{S}$
    and uniformly over at least 
    $(p-\phi_{\rm cond}(\bSigma)a_p/C^*)$ canonical directions
    in the sense that
    $J_p = \{ j\in[p]: \be_j/\|\bSigma^{-1/2}\be_j\| \in \bSigma^{1/2}\overline{S}\}$ 
    has cardinality
    $|J_p|\ge p-\phi_{\rm cond}(\bSigma)a_p/C^*$. 
    %
\end{restatable}

\Cref{thm:lasso-p-larger-than-n} is proved in
\Cref{sec:p-larger-n-without-strong-convexity}.
The strong convexity requirement in \Cref{assum:main}
is relaxed and replaced by assuming the high-probability bound
$\|\hbbeta\|_0\le\kappa n/2$ on the number of non-zero coordinates.
Surprisingly no conditions are required on the true regression
vector $\bbeta$ or on the tuning parameters, although these
quantities affect whether $\P(\|\hbbeta\|_0\le \kappa n/2)\to 1$ is satisfied. 
\Cref{fig:group-lasso} illustrates \Cref{thm:lasso-p-larger-than-n}
on simulated data.

\begin{figure}[b]
    \centering
\begin{tabular}{ccc}
    \includegraphics[width=1.7in]{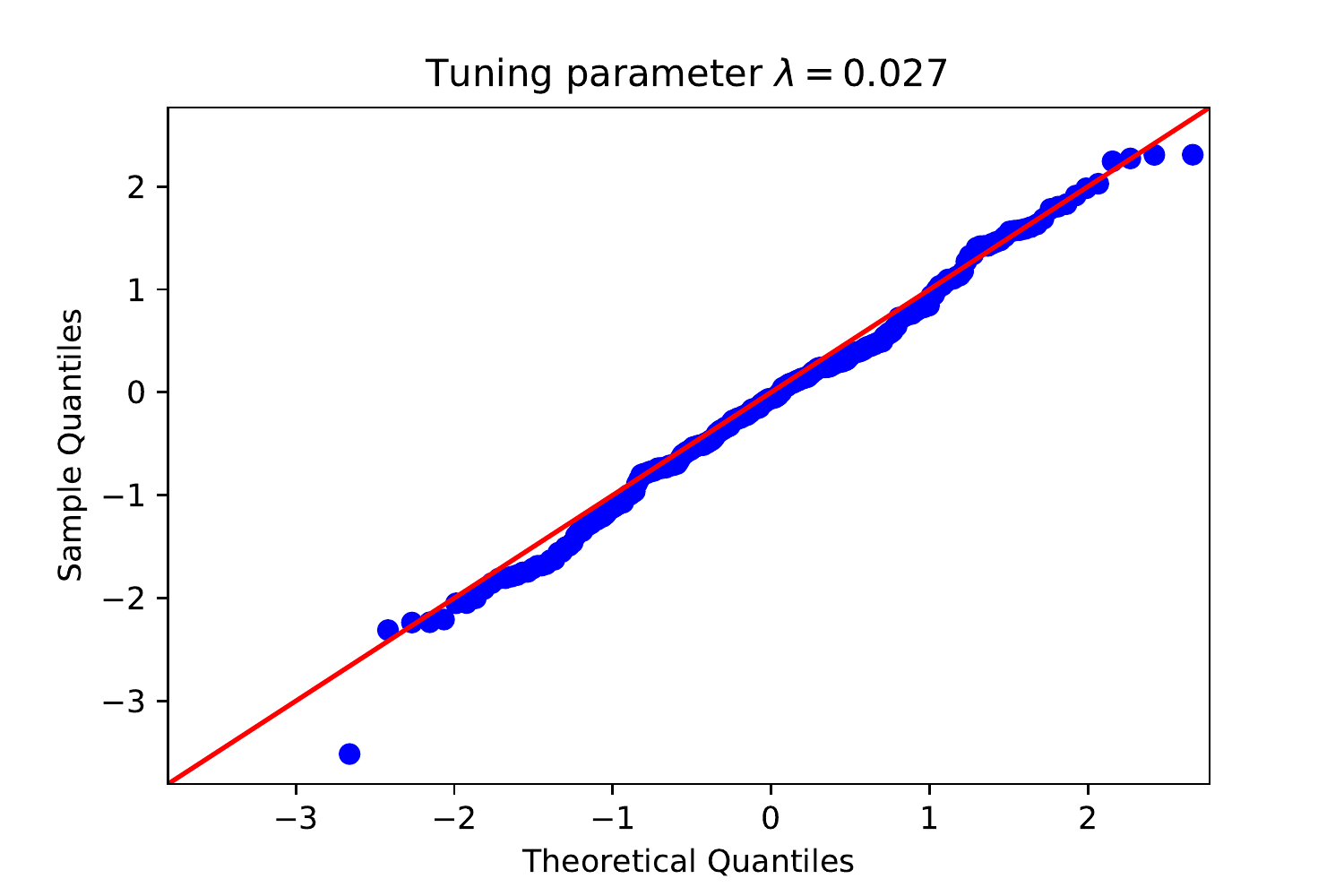}
    &
    \includegraphics[width=1.7in]{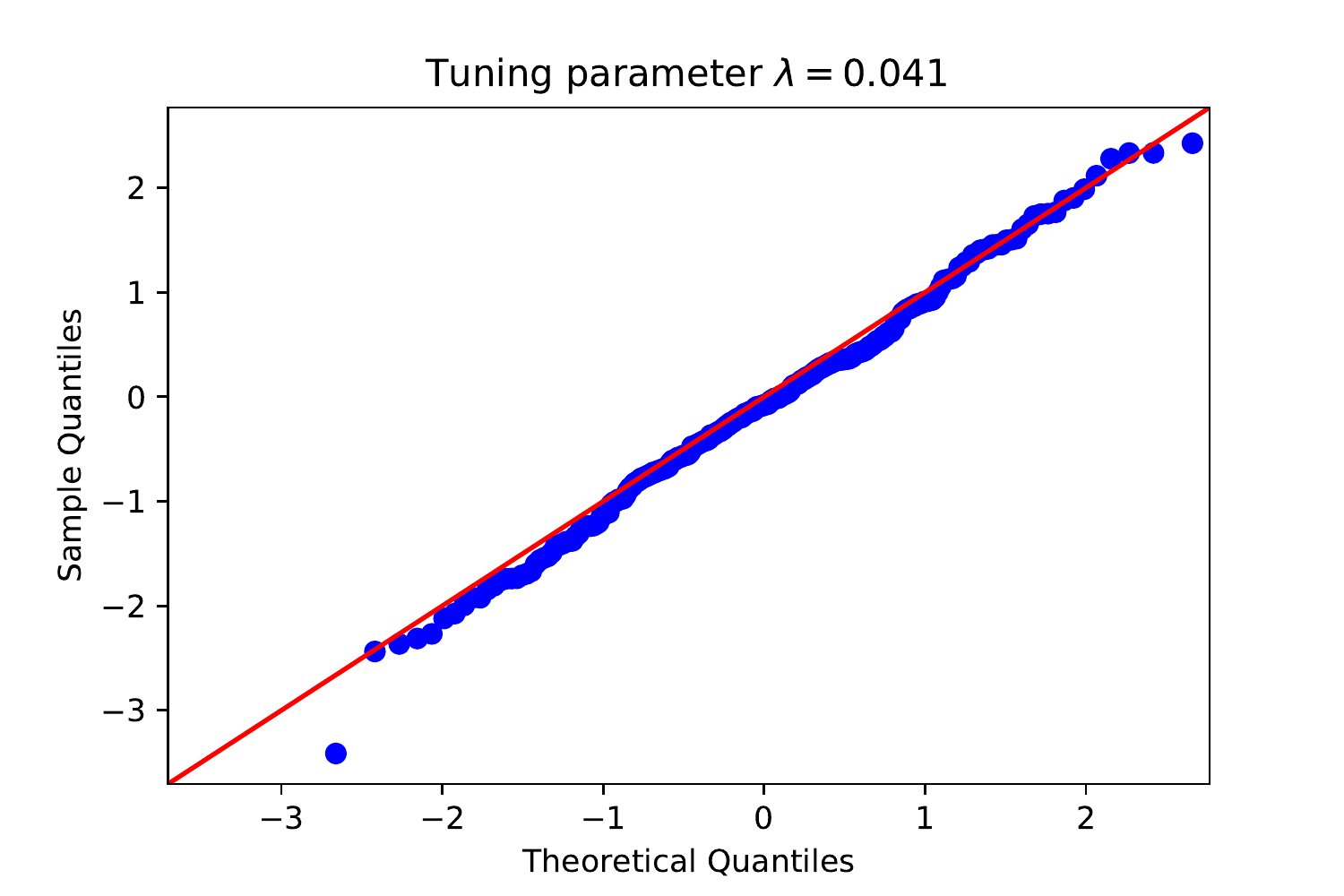}
    &
    \includegraphics[width=1.7in]{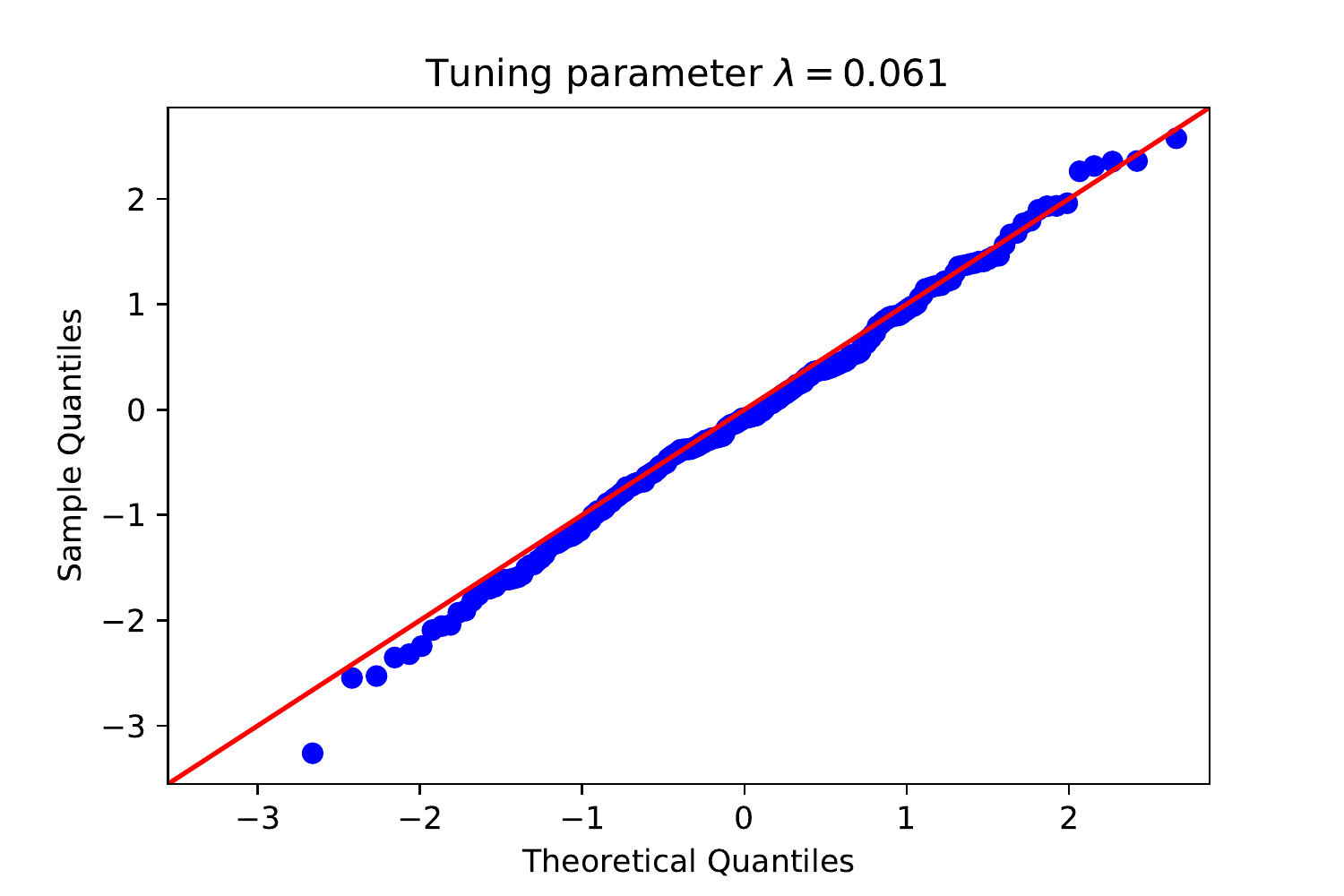}
    \\
    \includegraphics[width=1.7in]{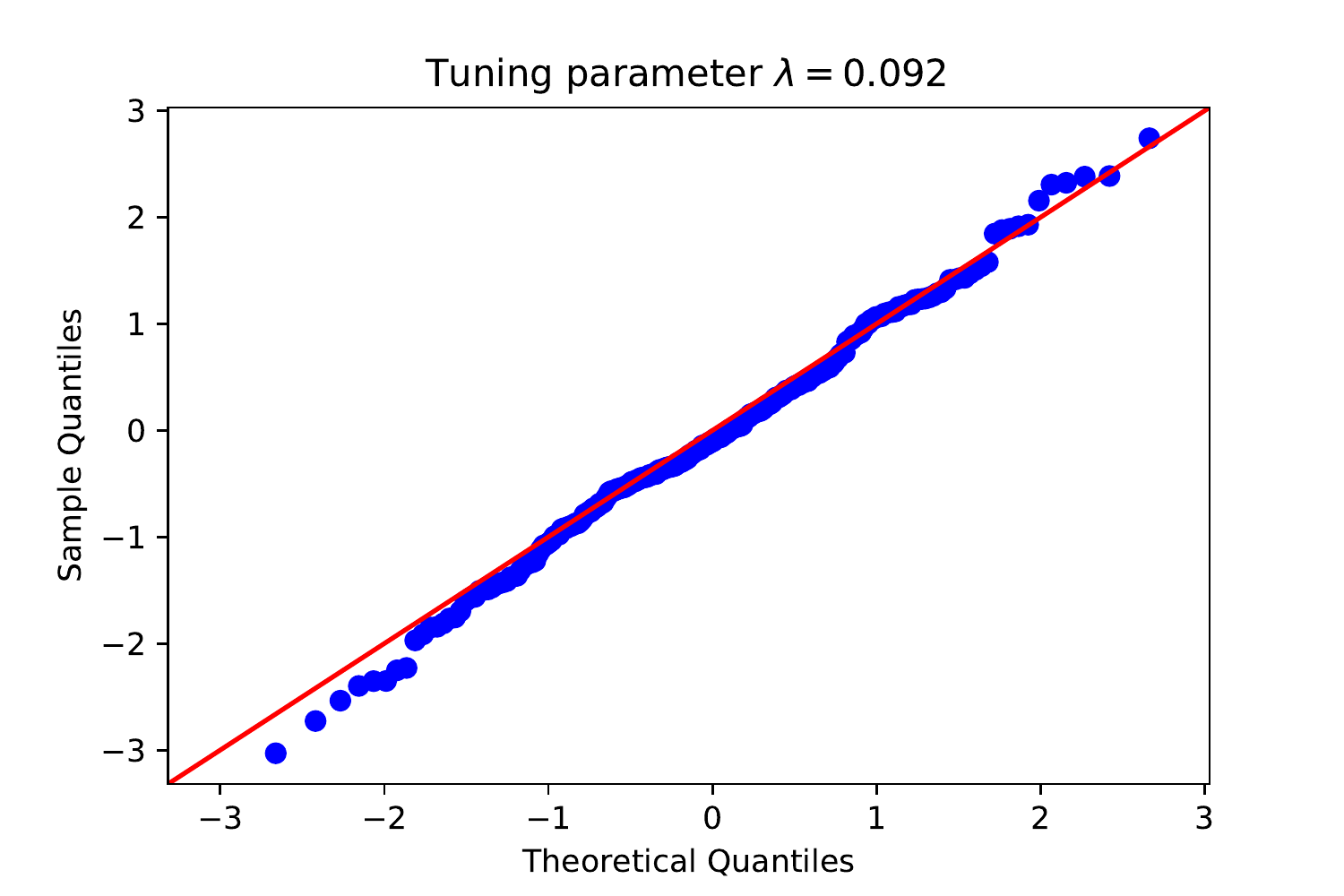}
    &
    \includegraphics[width=1.7in]{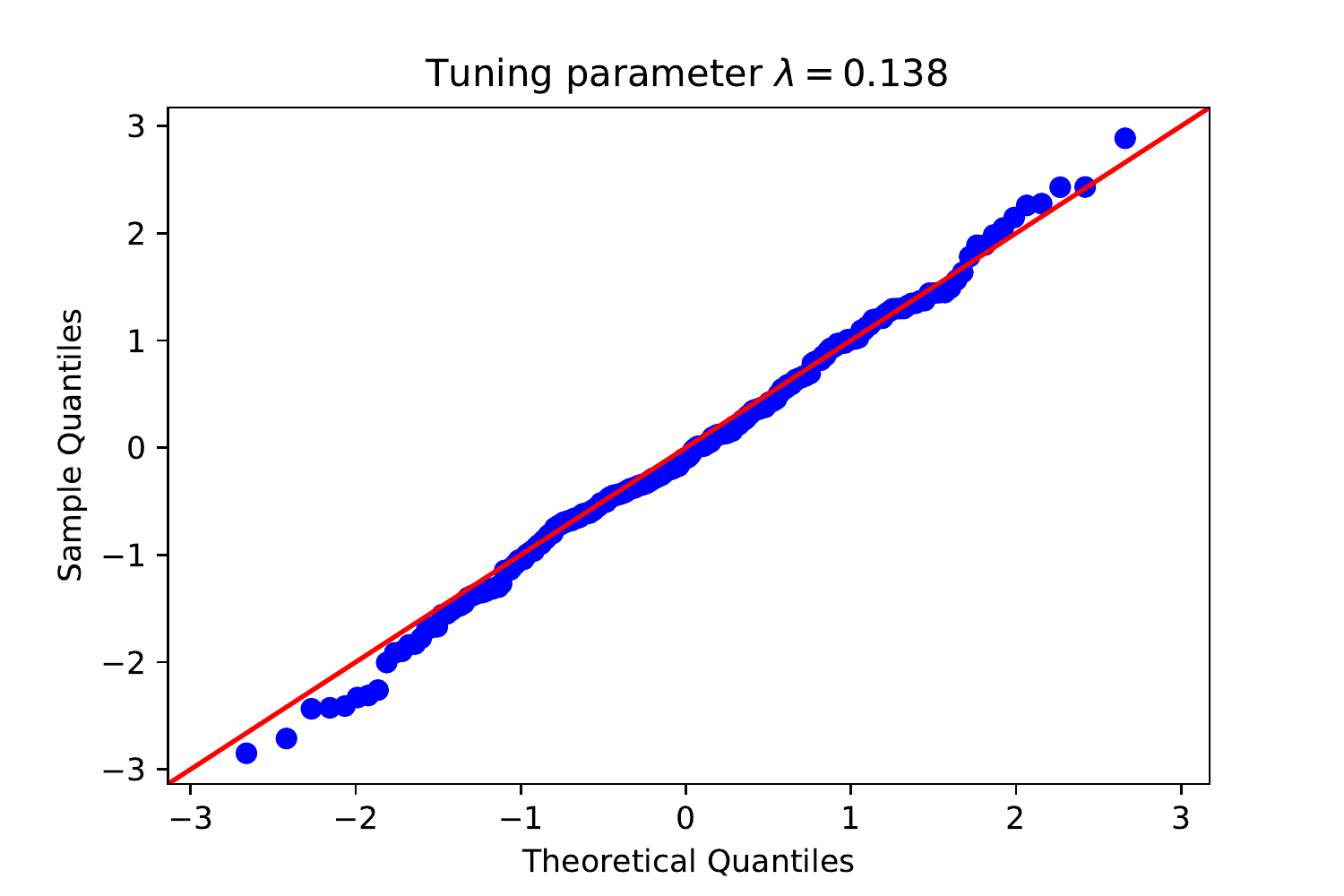}
    &
    \includegraphics[width=1.7in]{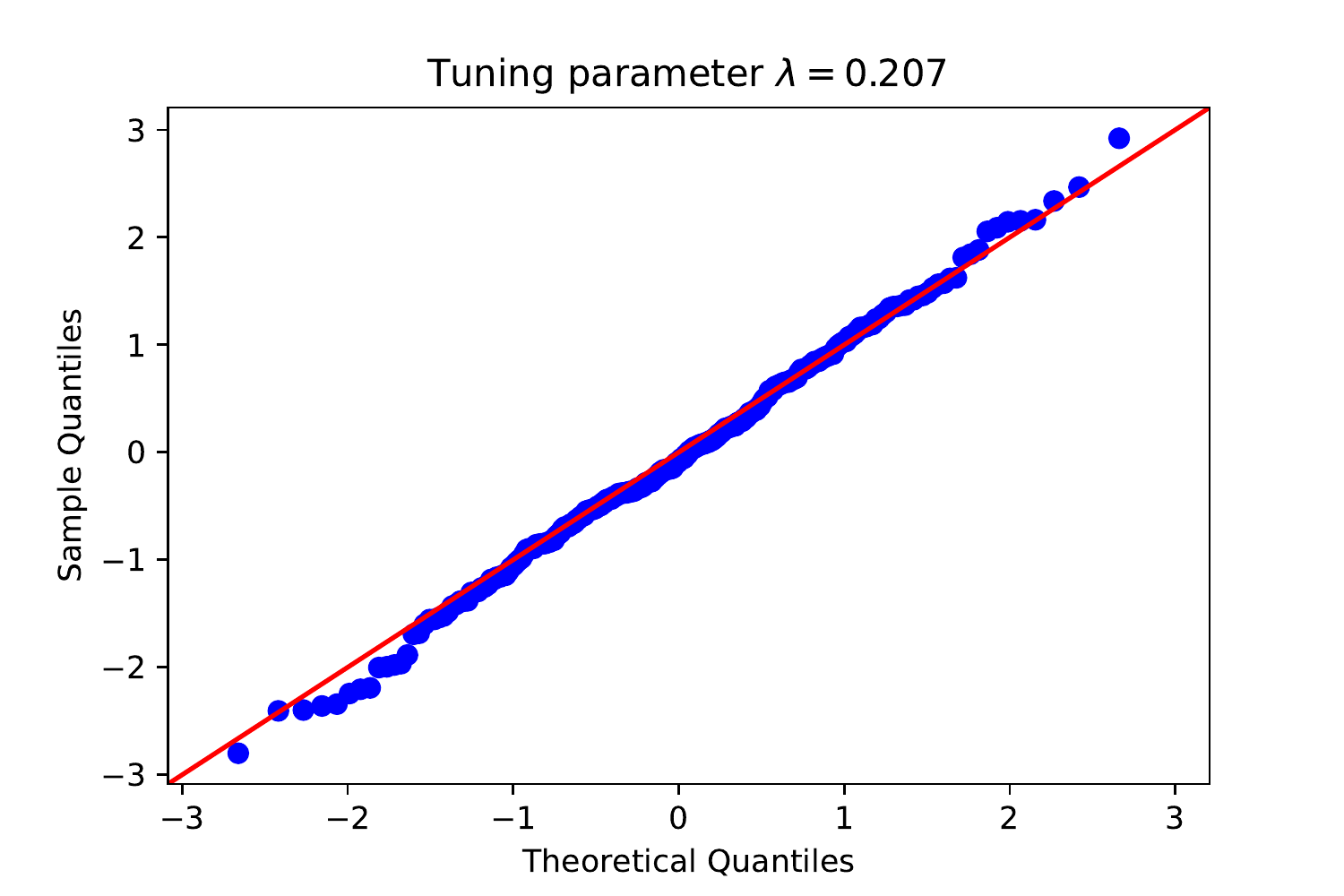}
    \\ \includegraphics[width=1.7in]{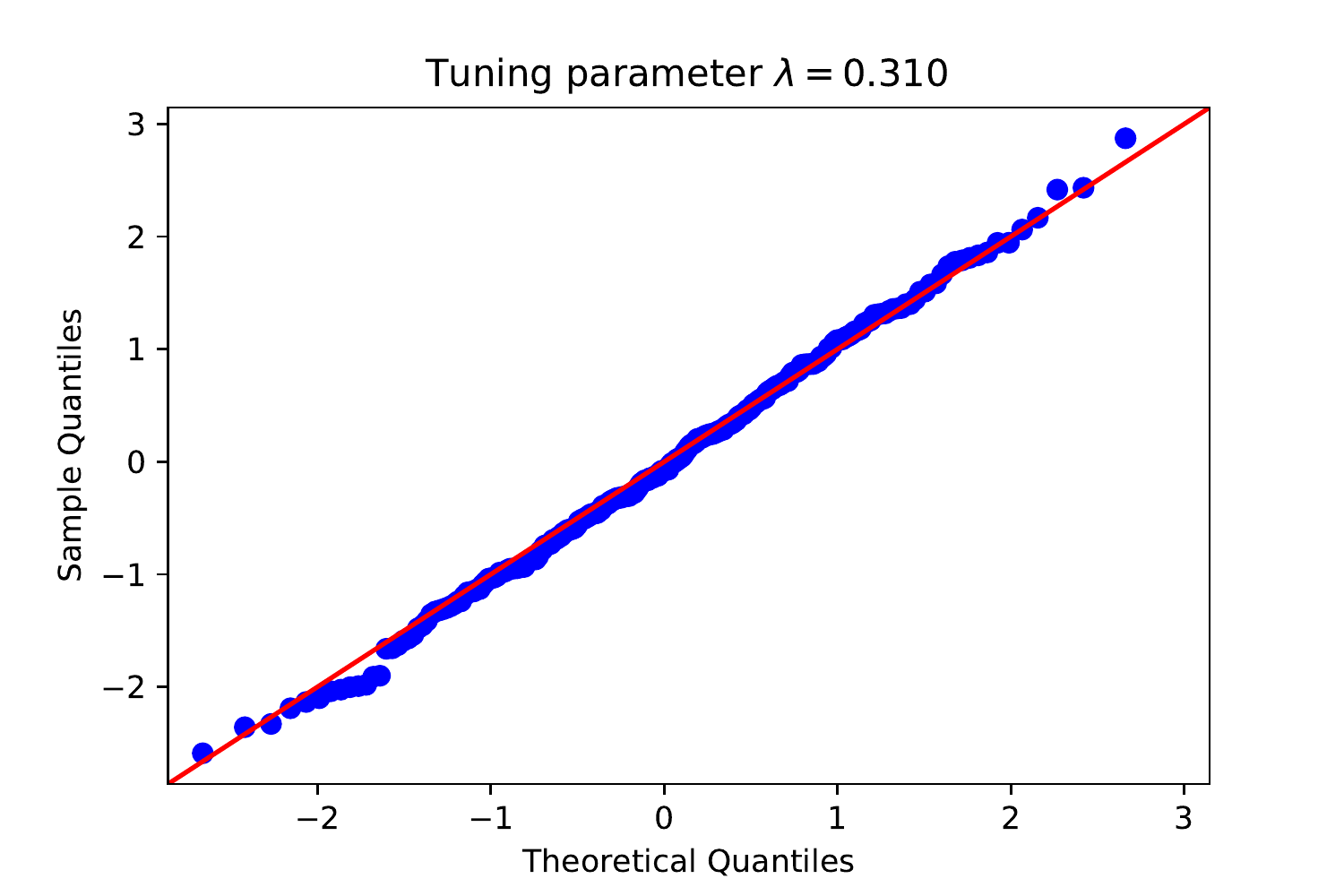}
    &
    \includegraphics[width=1.7in]{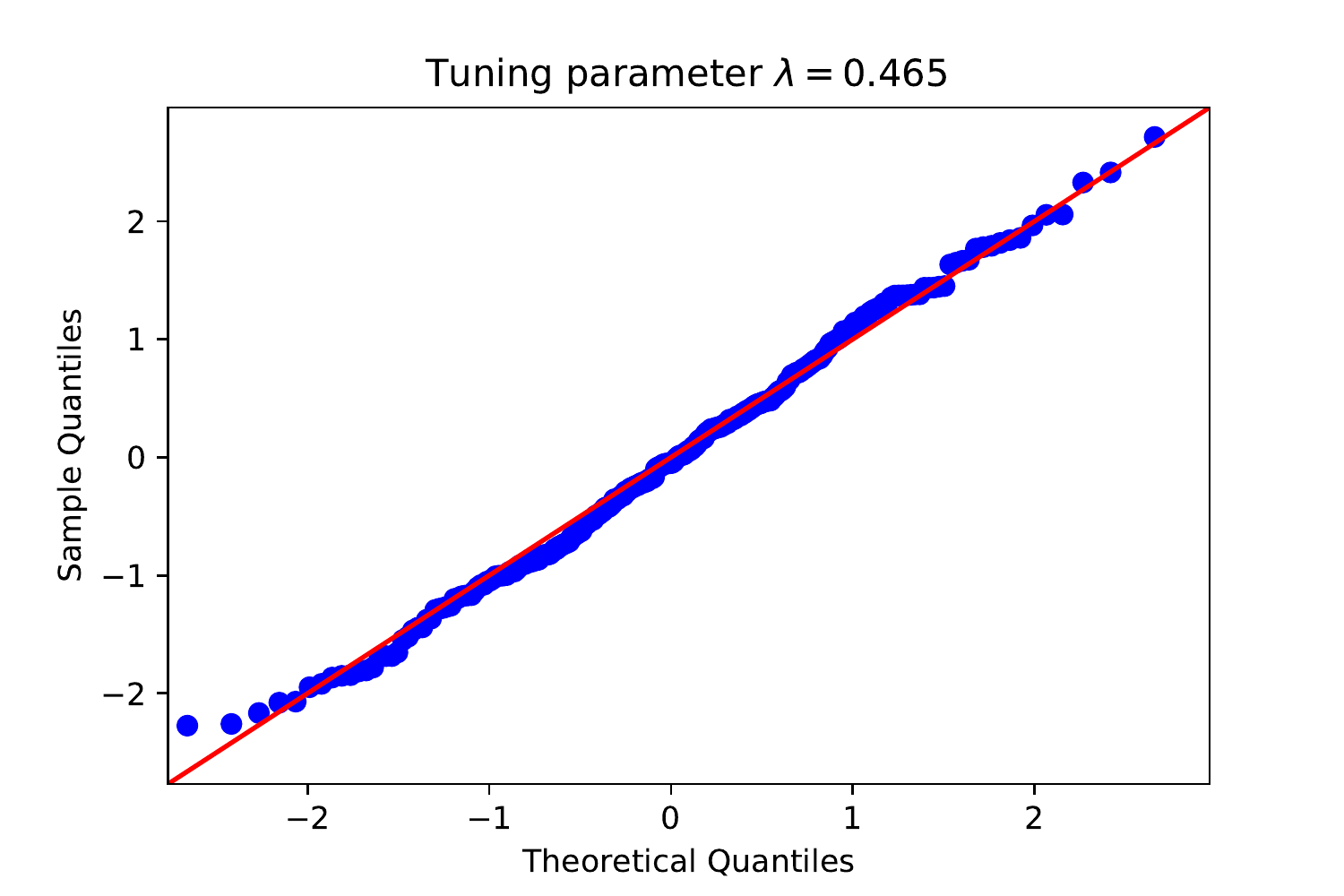}
    &
    \includegraphics[width=1.7in]{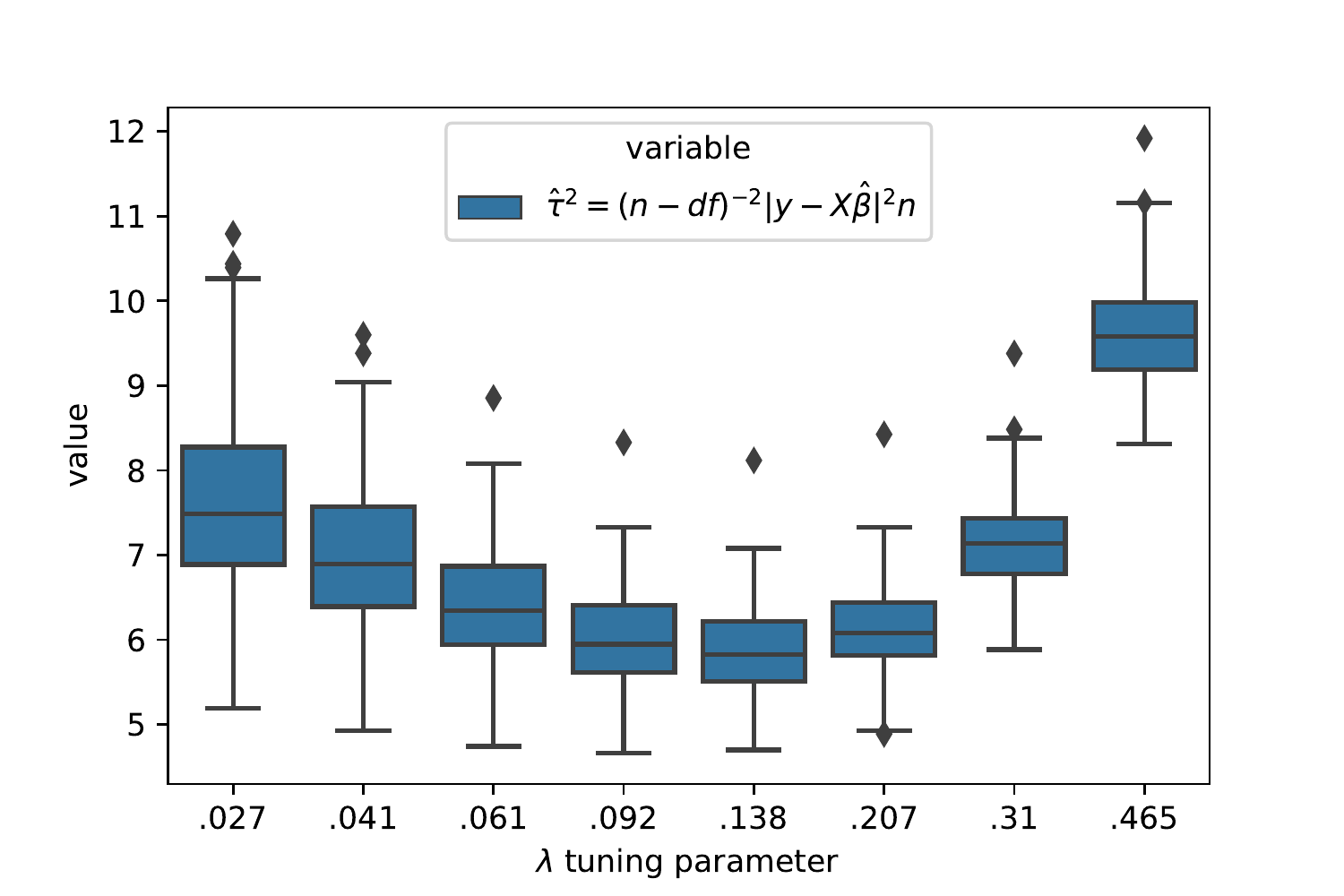}
\end{tabular}
    \caption{
        \small
        Standard normal QQ-plots of 
        $\langle \ba_0, \DeBias-\bbeta\rangle(n-\df)/\|\by-\bX\hbbeta\|$ for
        $(n,p,\sigma^2)=(600,900,2)$ and
        the group-Lasso \eqref{group-lasso}
        with 30 non-overlapping groups each of size 30,
        where all $\lambda_k$ in \eqref{group-lasso} are
        equal to a single parameter $\lambda$.
        The unknown coefficient vector $\bbeta$ is the same across
        all 256 repetitions and has 240 nonzero coefficients,
        all equal to 1 and
        belonging to 8 groups (so that the group sparsiy of $\bbeta$ is 8,
        and within these 8 groups all coefficients are equal to 1).
        The design covariance $\bSigma$ is generated once 
        as $\bSigma=\bW/(5p)$ where $\bW$ has Wishart distribution with covariance
        $\bI_p$ and $5p$ degrees of freedom. This choice of $(\bbeta,\bSigma)$ is 
        the same across all 256 repetitions.
        The direction of interest is
        $\ba_0=\be_1 / \|\bSigma^{-1/2}\be_1\|$ where
        $\be_1\in\R^p$ is the first canonical basis vector.
        The first 8 plots above are standard normal QQ-plots
        across 256 repetitions 
        for 8 different choices of $\lambda$. The ninth plot
        shows, for each $\lambda$,
        boxplots of $\hat\tau^2 = (1-\df/n)^{-2}\|\by-\bX\hbbeta\|^2/n$
        across the 256 repetitions. This
        $\hat\tau$ is proportional to the length of the corresponding
        confidence interval \eqref{eq:narrow-confidence-interval}
        so that the smallest confidence interval \eqref{eq:narrow-confidence-interval} is achieved for $\lambda=0.138$.
    }
\label{fig:group-lasso}
\end{figure}

\paragraph*{Comparison with existing works on the Lasso}
The Lasso is largely the most studied initial estimator
in previous literature on de-biasing and asymptotic normality,
so it provides a level playing field to compare our method
with existing results.
In {the approximate message passing (AMP) literature which includes} 
most existing works in the $n/p \to \gamma$ regime, e.g.
\cite{el_karoui2013robust,donoho2016high,JavanmardM14b} or more recently 
\cite{thrampoulidis2015lasso,miolane2018distribution,thrampoulidis2018precise,sur2018modern},
it is assumed that {$\bSigma=\bI_p$} 
and that the empirical distribution 
{$G_{n,p}(t) = p^{-1}\sum_{j=1}^p I\{\sqrt{n}\beta_j\le t\}$  
converges in distribution and in the second moment to some ``prior'' G
as $n,p\to+\infty$. Assume these conditions 
and consider the $j$-th component
$\debias_j$ of $\DeBias$ in \eqref{eq:De-Bias-hbbeta}
for $\bSigma=\bI_p$, that is,
\bes
\debias_j &=& \hbeta_j + \langle \bX\be_j, \by-\bX\hbbeta\rangle\big/(n-\df)
.
\ees
Then, the Lasso has the interpretation as its soft thresholded de-biased version, 
\bes
\hbeta_j = 
\eta\bigl(
\debias_j
~;~
\lam/(1-\df/n)
\bigr)
\quad\text{ where } \quad
\eta(u ; t) = \sgn(u)(|u|-t)_+
\ees
and the main thrust of the AMP theory is that the joint 
empirical distribution of the de-biased errors and the true coefficients, 
\bes
H_{n,p}(u,t) = p^{-1}\sum_{j=1}^p I\left\{\sqrt{n}\debias_j - \sqrt{n}\beta_j \le u, \sqrt{n}\beta_j\le t\right\}, 
\ees
converges in distribution and the second moment to the limit $H$ with independent 
$N(0,\tau_0)$ and $G$ components, where $\tau_0$} 
is characterized by a system of non-linear equations with
2 or 3 unknowns. 
These non-linear equations depend on the loss (here, the $\ell_2$ loss),
the penalty (here, the $\ell_1$-norm), the distribution of the noise,
as well as the prior distribution
that governs the empirical distribution of the coefficients of $\bbeta$. 
We note that these works typically assume that $\bX$ has $N(0,1/n)$ entries, so 
that their coefficient vector is equivalent to our $\sqrt{n}\bbeta$. 
For instance,
\cite{miolane2018distribution} characterizes the limit of  
the empirical distribution of $(\sqrt n \DeBias,\sqrt n\bbeta)$ 
in terms of two parameters,
$\{\tau_*(\lambda),\kappa_*(\lambda)\}$, that are defined as solutions
of the non-linear equations in \cite[Proposition 3.1]{miolane2018distribution};
see also \cite[Proposition 4.3]{celentano2019fundamental} for similar results
applicable to permutation-invariant penalty.
This approach presents some {drawbacks:} For instance it requires 
the convergence of 
the empirical distribution {$G_{n,p}$} 
to a limit (which can be viewed as a prior), 
it yields {the limiting distribution 
for the joint empirical distribution $H_{n,p}$ of the estimation errors and the unknown coefficients but} 
not for a fixed coordinate. 

The above \Cref{thm:lasso-p-larger-than-n} for the Lasso 
differs from this previous literature in major ways.
First, it provides a limiting distribution for the de-biased version 
of $\langle \ba_0, \hbbeta\rangle$ for a single, fixed direction $\ba_0$:
\Cref{thm:lasso-p-larger-than-n}
does not involve the empirical distributions of $\sqrt{n}\bbeta$, 
$\sqrt{n}\hbbeta$ or its de-biased version.
This contrasts with previous literature on the $n/p\to\gamma$ regime
where the confidence interval guarantee holds on average over the coefficients
$\{1,...,p\}$
\cite{el_karoui2013robust,donoho2016high,JavanmardM14b,sur2018modern}.
This improvement is important in practice: if the practitioner is interested
in the effect of a specific {effect} $j_0\in\{1,...,p\}$, it is important to construct
confidence intervals with strict type I error control for $\beta_{j_0}$,
as opposed to a controlled type I error that only
holds on average over all {coefficients}.
Another feature of the results in this paper is that there is no need to assume a prior
on the coefficients of $\bbeta$ in the limit.

Surprisingly, \Cref{thm:from-consistency,thm:main-result,thm:lasso-p-larger-than-n} and their proofs completely bypass 
solving the non-linear equations that appear in the aforementioned works 
as the nonlinearity is directly treated here with the 
the normal approximation in \Cref{thm:L2-distance-from-normal}. 
Moreover, \Cref{thm:from-consistency,thm:main-result,thm:lasso-p-larger-than-n} handle correlations in $\bSigma$ with a direct approach, 
while it 
is still unclear whether the non-linear equations approach from 
previous works can be extended to $\bSigma\ne \bI_p$.

\section{Examples}    \label{sec:examples}
We now present
three penalty functions
for which closed-form expressions for $\smash{\hbH}$ and $\bw_0$
are available. 
In this section, when computing gradients with respect to $\bz_0$
in order to find closed-form expressions for $\bw_0$
in \eqref{eq:identity-bw0},
we consider $(\bX\bQ_0,\bep)$ fixed as in
\eqref{gradient-f-bz0-explicit}. Explicitly,
$\nabla \hbbeta(\bz_0)^\top$ is uniquely defined as
\begin{equation}
    \label{gradient-bz0-while-XQ0-bep-fixed}
    \tbbeta - \hbbeta =[\nabla \hbbeta(\bz_0)]^\top \bfeta + o(\|\bfeta\|)
\end{equation}
where 
$\hbbeta=\argmin_{\bb\in\R^p}\|\bX(\bb-\bbeta) - \bep\|^2/(2n)+g(\bb)$ and
$\tbbeta=\argmin_{\bb\in\R^p}\|(\bX+\bfeta\ba_0^\top)(\bb-\bbeta) - \bep\|^2/(2n)+g(\bb)$.
When computing gradients with respect to $\by$ in order
to find closed-form expressions for $\hbH$ in \eqref{def-matrix-H},
we view $\hbbeta(\by,\bX)=\argmin_{\bb\in\R^p}\|\bX\bb - \by\|^2/(2n)+g(\bb)$
as a function of $(\by,\bX)$ and
if $\by\mapsto \hbbeta(\by,\bX)$
is differentiable at $\by$ for a fixed $\bX$ then
\begin{equation}
\hbbeta(\tby,\bX) - \hbbeta(\by,\bX) =[(\partial/\partial\by)\hbbeta(\by,\bX)] (\tby-\by) + o(\|\tby-\by\|)
\label{eq:jacobian-y}
\end{equation}
where $(\partial/\partial\by)\hbbeta(\by,\bX)\in\R^{p\times n}$ is the Jacobian.
Once the Jacobian $(\partial/\partial\by)\hbbeta(\by,\bX)$ is computed,
$\hbH$ in \eqref{def-matrix-H} is given by $\hbH{}^\top =
\bX (\partial/\partial \by)\hbbeta(\by,\bX)$.
We use the Jacobian notation $(\partial/\partial\by)\hbbeta(\by,\bX)$
when computing the derivatives with respect to $\by$ to avoid
confusion with the gradient $\nabla \bbeta(\bz_0)$ in \eqref{gradient-bz0-while-XQ0-bep-fixed}. 

\subsection{Twice continuously differentiable penalty}
\label{sec:computation-twice-differentiation-g}
The simplest example for which closed-form expressions for
$\hbH,\df,\bw_0$ can be obtained is that
of twice continuously differentiable and strongly convex penalty $g$.
If $g$ is strongly convex,
\Cref{lemma:lipschitzness} proves that the Fr\'echet derivative 
of $\bh=\hbbeta-\bbeta$ with respect to $(\bep,\bX)$
exist for almost every $(\bep,\bX)$ by Rademacher's theorem.
At a point $(\bep,\bX)$ where the derivative exist,
we obtain a closed form expression for the gradient
\eqref{eq:identity-bw0} as follows.
The KKT conditions of the optimization problem \eqref{penalized-hbbeta-general}
read
$
\bX^\top\bigl(\by-\bX\hbbeta\bigr) =
\bX^\top\bigl(\bep-\bX\bh\bigr) = n\nabla g(\hbbeta).$
Differentiation with respect to $\bz_0$
for a fixed $(\bep,\bX\bQ_0)$ as in \eqref{gradient-bz0-while-XQ0-bep-fixed}
gives
$$\big\{\bX^\top\bX + n[\nabla^2 g(\hbbeta)] \big\} (\nabla \hbbeta(\bz_0))^\top
= \ba_0 (\by-\bX\hbbeta)^\top - \bX^\top \langle \ba_0,\bh\rangle
.
$$
By the product rule, this provides the derivative of $f(\bz_0)=
\bX\bh- \bep$, namely
\begin{equation}
\nabla f(\bz_0)^\top = 
\bX\bigl(
    \bX^\top\bX +  n\nabla g(\hbbeta) 
\bigr)^{-1} 
\bigl[
    \ba_0 (\by-\bX\hbbeta)^\top 
    -
    \langle \ba_0,\bh\rangle \bX^\top
\bigr]
+ \bI_n \langle \ba_0,\bh\rangle.
\label{nabla-f-twice-continuous-differentiable}
\end{equation}
Regarding $\hbH$ involving differentiation with respect to $\by$, 
the Lipschitz condition of the map $\by\mapsto \hbbeta$ for strongly convex $g$ follows from 
\eqref{eq:lipschitz-continuity-y} in the proof of \Cref{prop:hat-bH}. 
Hence the Jacobian in \eqref{eq:jacobian-y} exists almost everywhere,
and differentiation of the KKT conditions
for fixed $\bX$ gives
$({\bX^\top\bX} + n \nabla^2 g(\hbbeta)) (\partial/\partial\by)\hbbeta(\by,\bX)
= \bX^\top$
so that
$$
\hbH = \big(\bX(\partial/\partial\by)\hbbeta(\by,\bX)\big)^\top
= \bX\bigl(\bX^\top\bX + n \nabla^2 g(\hbbeta)\bigr)^{-1}\bX^\top.
$$
Identity \eqref{nabla-f-twice-continuous-differentiable}
combined with this expression for $\hbH$
provides \eqref{eq:identity-bw0} with 
$$\bw_0=\bX\{\bX^\top\bX +  n\nabla g(\hbbeta) \}^{-1} \ba_0.$$

\subsection{Lasso} \label{sec:de-biasing-lasso}

Consider the Lasso $\hbbeta$ in \eqref{lasso}. 
For $(\bep,\bX)$ with continuous distribution such as Gaussian 
under consideration here, almost surely $\hbbeta$ is unique and 
\bel{lasso-unique}
&& \bX_{\Shat}^\top(\by - \bX\hbbeta)/n = \lam \sgn(\hbbeta_{\Shat}),\ \ 
\|\bX_{\Shat^c}^\top(\by - \bX\hbbeta)/n\|_\infty < \lam,\ \ 
\rank(\bX_{\Shat})=|\Shat|, 
\eel
for the Lasso as in \cite{zhang2010nearly, tibshirani2013lasso} and 
\cite[Proposition 3.9]{bellec_zhang2018second_order_stein}, 
so that the Jacobian of the mapping 
$(\bz_0,\bep,\bX\bQ_0)\to \bX\hbbeta$ with respect to $\bz_0$ and $\bep$ 
can be computed directly by differentiating the KKT condition as in 
\cite{tibshirani2013lasso, bellec_zhang2018second_order_stein,bellec_zhang2019dof_lasso}. 
The following proposition provides 
closed-form expressions for the gradients 
for the Lasso estimator which are valid almost surely
and require no assumption on the sparsity of $\bbeta$
or the penalty level. 

\begin{proposition}
    \label{proposition:lasso-general}
    Let $\lambda>0$ and consider the Lasso $\hbbeta$ in \eqref{lasso}.
    Let $\Shat = \{j\in[p]: \hbeta_j\ne 0\}$.
    For almost every $(\bepbar,\bXbar)\in\R^{n\times (p+1)}$,
    there exists a neighborhood of $(\bepbar,\bXbar)$ in which
    the map
    $(\bep,\bX)\mapsto \Shat$ is constant,
    $|\Shat|\le n$, $\bX_{\Shat}^\top\bX_{\Shat}$ is invertible
    and the map
    $(\bep,\bX)\mapsto \hbbeta$ is Lipschtiz.
    In this neighborhood, almost surely
    $\bigl[\nabla \hbbeta(\bz_0)^\top\bigr]_{\Shat^c} = \mathbf{0} \in\R^{n\times |\Shat^c|}$, 
    $$
    \bigl[\nabla \hbbeta(\bz_0)^\top\bigr]_{\Shat} = 
    (\bX_{\Shat}^\top\bX_{\Shat})^{-1}
    \bigl((\ba_0)_{\Shat}(\bX\hbbeta - \by)^\top - \bX_{\Shat}^\top \langle \ba_0,\bh\rangle\bigr) 
    \in \R^{n\times |\Shat|},
    $$
    $\hbH =  \bX_{\Shat}(\bX_{\Shat}^\top\bX_{\Shat})^{-1}\bX_{\Shat}^\top$,
    $\df=|\Shat|$
    and \eqref{eq:identity-bw0} holds with $\bw_0 =
    \bX_{\Shat}(\bX_{\Shat}^\top\bX_{\Shat})^{-1}(\ba_0)_{\Shat}$.
\end{proposition}


\begin{proof}[Proof of \Cref{proposition:lasso-general}] 
    Proposition 3.9 in \cite{bellec_zhang2018second_order_stein}
    proves, for almost every $(\bep,\bX)$, uniqueness of $\hbbeta$ and \eqref{lasso-unique}. 
    Let $(\bepbar,\bXbar)\in\R^{n\times (p+1)}$ be a point at which \eqref{lasso-unique} holds. 
    It follows from \eqref{lasso-unique} and the Holder continuity of $(\bep,\bX)\mapsto \bep-\bX\hbbeta$ established
    after \eqref{lipschitz-eq-1}  that for 
    almost every $(\bepbar,\bXbar)$ there is an open neighborhood in $\R^{n\times (p+1)}$
    in which $\Shat=\Sbar$, $\sgn(\hbbeta_{\Shat})=\bs_{\Sbar}$ and $\rank(\bX_{\Shat})=|\Sbar|$ 
    are constants, so that 
    $\hbbeta$ is locally equal to 
    $\argmin_{\bb\in\R^{\Sbar}}\|\by- \bX_{\Sbar}\bb\|^2/(2n)
    + \lam \bs_{\Sbar}^\top\bb_{\Sbar}$ with the linear penalty $\lam \bs_{\Sbar}^\top\bb_{\Sbar}$.
    In this neighborhood $(\bep,\bX)\to \hbbeta$ has the analytic expression 
    \bes
    \hbbeta_{\Sbar} = (\bX_{\Sbar}^\top\bX_{\Sbar})^{-1} (\bX_{\Sbar}^\top\by - n\lam \bs_{\Sbar}),
    \qquad
    \hbbeta_{\Sbar^c} = \mathbf{0}_{\Sbar^c}.
    \ees
    Differentiating the above immediately yields the formulas for 
    $\hbH$, $\df$ and $\bigl[\nabla \hbbeta(\bz_0)^\top\bigr]_{\Shat^c}$. 
    For $\bigl[\nabla \hbbeta(\bz_0)^\top\bigr]_{\Shat}$, 
    differentiating both sides of 
    $\bX_{\Sbar}^\top(\bX_{\Sbar}(\hbbeta - \bbeta) - \bep) = - n\lam \bs_{\Sbar}$ yields 
    \bes
    (\ba_0)_{\Sbar} (\bX_{\Sbar}\hbbeta - \by)^\top
    +\bX_{\Sbar}^\top \ba_0^\top\bh
+\bX_{\Sbar}^\top\bX_{\Sbar}[\nabla\hbbeta(\bz_0)^\top]_{\Sbar} = {\bf 0}
    \ees
    due to $\bX = \bX\bQ_0+\bz_0\ba_0^\top$. 
    Finally, the formula for $\bw_0$ follows from 
    $(\pa/\pa\bz_0)(\bX\hbbeta - \by) = \bX(\pa/\pa\bz_0)\hbbeta + \bI_n\ba_0^\top\bh$ and 
    simple algebra.
\end{proof}

\subsection{Group Lasso}
\label{sec:de-biasing-group-lasso}

Consider a partition $(G_1,...,G_K)$ of $\{1,...,p\}$
and the group Lasso estimator in \eqref{group-lasso}.
Let ${\widehat{B}} = \{k {\in [K]}: \|\hbbeta_{G_k}\|\ne 0 \}$ be the set of active groups
and $\Shat=\cup_{k\in {\widehat{B}}} G_k$ the union of all active groups. 
Define
the block diagonal matrix
$\bM=\diag\bigl((\bM_{G_k,G_k})_{k\in {\widehat{B}}} \bigr)
\in\R^{|\Shat|\times |\Shat|}$ by
\begin{equation}
    \label{matrix-M-group-lasso}
    \bM_{G_k,G_k} = 
    {n \lambda_k }{\|\hbbeta_{G_k}\|^{-1}}
    \left(\bI_{G_k} - \|\hbbeta_{G_k}\|^{-2}\hbbeta_{G_k}\hbbeta_{G_k}^\top\right),
    \qquad
    \bM\in\R^{|\Shat|\times |\Shat|}.
\end{equation}
The following proposition provides closed-form expressions 
for the gradients for the Group Lasso estimator and related quantities 
$\hbH$ and $\bw_0$ in terms of $\Shat$ and $\bM$.
Its proof is given in \Cref{sec:proof-gradient-GL}.
Note that the formula for $\hbH$ was known
\cite{vaiter2012degrees}.

\begin{restatable}{proposition}{LemmaGradientGroupLasso}
    \label{lemma:gradient-group-lasso}
    The following holds for for almost every 
    $(\bybar,\bXbar)\in\R^{n\times (1+p)}$.
    The set $\overline{B}=\{k\in[K]: \|\bbetabar_{G_k}\|>0\}$
    of active groups
    is the same for all minimizers $\bbetabar$
    of \eqref{group-lasso} at $(\bybar,\bXbar)$
    and
    $\widehat{B}=\overline{B}$
    for all $(\by,\bX)$ in a sufficiently small
    neighborhood of $(\bybar,\bXbar)$.
    If additionally $\bXbar_{\overline{S}}^\top\bXbar_{\overline{S}}$
    is invertible
    where $\overline{S}=\cup_{k\in\overline{B}}G_k$
    then
    the map $(\by,\bX)\mapsto \hbbeta$ is Lipschitz 
    in a sufficiently small neighborhood of $(\bybar,\bXbar)$.
    In this neighborhood
    we have
    \begin{equation*}
[\nabla \hbbeta(\bz_0)]_{\Shat^c} = 0,
        \quad
          [\nabla \hbbeta(\bz_0)]_{\Shat}^{\top} =
 (\bX_{\Shat}^\top\bX_{\Shat}+\bM)^{-1}[
    (\ba_0)_{\Shat}(\by - \bX\hbbeta)^\top 
    - \langle \ba_0, \bh\rangle \bX_{\Shat}^\top ],
    \end{equation*}
$\hbH = \bX_{\Shat}(\bX_{\Shat}^\top\bX_{\Shat}+\bM)^{-1}\bX_{\Shat}^\top$
and \eqref{eq:identity-bw0} holds with $\bw_0 = 
\bX_{\Shat}
(\bX_{\Shat}^\top\bX_{\Shat}+\bM)^{-1}
(\ba_0)_{\Shat}$.
\end{restatable}

\section{Proof of the main results in Section~\ref{sec:application-de-biasing} }
\label{sec:proofs}

In order to prove \Cref{thm:from-consistency,thm:main-result},
we apply the bound on the normal approximation in
\Cref{thm:L2-distance-from-normal}.
We recall here some notation used throughout the proof.
Let $\hbbeta$ be the estimator \eqref{penalized-hbbeta-general},
$\hbH$ the gradient of $\by\to\bX\hbbeta$ as in \eqref{eq:gradient-phi-intro},
$\ba_0\in\R^p$ with $\|\bSigma^{-1/2}\ba_0\|=1$,
$\bz_0$ and $\bQ_0$ as defined in \eqref{u_0}, 
\begin{equation}
    \label{eq:recap}
\theta=\langle\ba_0,\bbeta\rangle,
\qquad
f(\bz_0) = \bX\hbbeta-\by, \qquad \xi_0 = \bz_0^\top f(\bz_0) - \dv f(\bz_0).
\end{equation}
Vector $\bw_0\in\R^n$ is given by \Cref{lemma:existence-bw0}.
The oracle $\bbeta^*$
and its associated noiseless prediction risk $\risk$
are given by \eqref{oracle-general-g-strongly-convex}.
Throughout,
$\E_0$ denotes the conditional expectation given $(\bep,\bX\bQ_0)$
and $\Var_0$ the conditional variance given $(\bep,\bX\bQ_0)$.

\subsection{Lipschitzness of regularized least-squares
and existence of $\bw_0$}
\label{sec:proof-lipschitzness}

By Rademacher's theorem, a Lipschitz function $U\to\R$
for some open set $U\subset \R^q$
is Fr\'echet differentiable almost everywhere in $U$.
The following lemma is the device that 
verifies the Lipschitz condition for 
the mappings $(\bep,\bX)\mapsto\hbbeta$ and $(\bep,\bX)\mapsto\bX\bh-\bep$
in certain open set $U$,
and consequently differentiability almost everywhere in $U$.

\begin{lemma}
    \label{lemma:lipschitzness}
    Let {$\bbeta\in\R^p$,} $\bX$ and $\tbX$ 
    be two design matrices of size $n\times p$, 
    and $\bep$ and  $\tbep$ two noise vectors in $\R^n$. 
    Let $g:\R^p\to\R$ be convex such that minimizers
    \bes
        \hbbeta
        \in \argmin_{\bb\in\R^p}\left\{
            \frac{\|\bep+\bX(\bbeta - \bb)\|^2}{2n} + g(\bb)
        \right\},
        \quad
        \tbbeta
        \in \argmin_{\bb\in\R^p}\left\{
            \frac{\|\tilde\bep+\tbX(\bbeta - \bb)\|^2}{2n} + g(\bb)
        \right\}
    \ees
    exist.
        Let $\bh=\hbbeta-\bbeta$,
        $\bff = \bX\bh - \bep$,
        $\tbh = \tbbeta-\bbeta$,
        $\tbf = \tbX\tbh - \tbep$.
    Let also
    $D_g(\tbbeta,\hbbeta) = (\tbbeta-\hbbeta)^\top\big\{(\pa g)(\tbbeta) - (\pa g)(\hbbeta)\big\}$
    where $(\pa g)(\tbbeta) = n^{-1} \tbX{}^\top(\tbep -\tbX\tbh)$
    is the subdifferential at $\tbbeta$ given by the optimality condition
    of the above optimization problem
    and similarly for $(\pa g)(\hbbeta)$,
    with $D_g(\hbbeta,\tbbeta)\ge 0$
    by the monotonicity of the subdifferential.
    Then
    \bel{lipschitz-eq-1}
    n D_g(\tbbeta,\hbbeta)		
    + \|\bff - \tbf\|^2
    &=&
        (\tbh - \bh)^\top(\bX-\tbX)^\top\bff
        + (\tbep - \bep + (\bX-\tbX)\bh)^\top(\bff - \tbf) 
    \cr&=&
    \trace[(\bX - \tbX)^\top(\bff \tbh{}^\top - \tbf \bh^\top)]
    + (\tbep - \bep)^\top(\bff - \tbf).
    \eel
    If $g$ is coercive (i.e., $g(\bx)\to +\infty$ as $\|\bx\|\to+\infty$)
    then the map
    $(\bep,\bX)\mapsto \bep - \bX\bh$ is Holder continuous
    with coefficient $1/2$ on every compact. We also have
    \begin{align}
        \nonumber
   & n D_g(\tbbeta,\hbbeta)	
    +
    \|\bX(\hbbeta-\tbbeta)\|^2/2
    +
    \|\tbX(\hbbeta-\tbbeta)\|^2/2
    \\ &=
    (\hbbeta-\tbbeta)^\top(\bX^\top\bep - \tbX^\top\tbep)
        +
        (\hbbeta-\tbbeta)^\top(\bX^\top \bX - \tbX{}^\top\tbX)(\bh + \tbh)/2
    \label{eq:lipschitz-3}
    \\ &\le
    \|\hbbeta - \tbbeta\|
    \|\bep - \tbep\| \|\bX+\tbX\|_{op}/2
    \label{lipschitz-eq-2}
      \\&\qquad +
    \|\hbbeta - \tbbeta\|
    \|\bX-\tbX\|_{op}
    \bigl[
    (\|\bep+\tbep\|/2
    + 
    (\|\bX\|_{op}+\|\tbX\|_{op}))
    (\|\tbh\| + \|\bh\|)/2
    \bigr]
    .
    \nonumber
    \end{align}
    If either $g$ is strongly convex
    or if there exists a constant $\kappabar>0$
    and a bounded neighborhood $\mathcal N$ of $(\bepbar,\bXbar)$
    such that $\|\hbbeta-\tbbeta\|\kappabar \le \|\bX(\hbbeta-\tbbeta)\|$
    for all $\{(\bep,\bX),(\tbep,\tbX)\}\subset \mathcal N$
    then the map $(\bep,\bX)\mapsto \hbbeta$ is Lipschitz
    in $\mathcal N$.
\end{lemma}
\begin{proof}[Proof of \Cref{lemma:lipschitzness}] 
    The KKT conditions for $\hbbeta$ and $\tbbeta$ provide
    \bes
    n(\hbbeta-\tbbeta)^\top(\partial g)(\hbbeta) 
    =
    (\tbh - \bh)^\top\bX^\top\bff
    ,\qquad
    n(\tbbeta-\hbbeta)^\top(\partial g)(\tbbeta) 
    =
    (\bh - \tbh)^\top\tbX{}^\top\tbf.
    \ees
    Summing
    and adding $\|\bff - \tbf\|^2=(\tbep - \bep)^\top(\bff-\tbf)
    + (\bX\bh - \tbX\tbh)^\top(\bff - \tbf)$
    on both sides,
    \bes
    n D_g(\hbbeta,\tbbeta) + \|\bff -  \tbf\|^2
    &=&
    \tbh{}^\top(\bX -\tbX)^\top\bff
    + \bh^\top(\tbX-\bX)^\top\tbf
    +
    (\tbep - \bep)^\top(\bff - \tbf)
    \ees
    so that \eqref{lipschitz-eq-1} holds.

    By optimality of $\hbbeta$,
    $\|\bff\|^2/(2n) + g(\hbbeta)
    \le \|\bX\bbeta + \bep\|^2/(2n)$.
    If $g$ is coercive, this implies
    that for every compact $K\subset\R^{n\times (1+p)}$,
    $\|\bff\|+\|\bh\|$ and $\|\tbf\|+\|\tbh\|$ are bounded by a constant
    depending only on $g,\bbeta,n,K$ 
    if $\{(\bep,\bX),(\tbep,\tbX)\}\subset K$.
    In this case, \eqref{lipschitz-eq-1} implies that 
    $\|\tbf - \bff\|^2 \le(\|\tbX - \bX\|_F + \|\bep-\tbep\|) C(g,\bbeta,n,K)$ for some other constant  depending on $g,\bbeta,n,K$ only.
    This implies Holder continuity 
    of $(\bep,\bX)\mapsto \bep - \bX\bh$ with Holder coefficient $1/2$
    on every compact.

    For \eqref{eq:lipschitz-3} and \eqref{lipschitz-eq-2},
    the KKT conditions for $\hbbeta$ yield
    \bes
    n(\hbbeta-\tbbeta)^\top (\pa g)(\hbbeta) + \|\bX(\hbbeta-\tbbeta)\|^2/2
     = (\hbbeta-\tbbeta)^\top \bX^\top(\bep-\bX(\bh+\tbh)/2). 
    \ees
    Summing the above and its $\tbbeta$ counterpart yields 
    the equality \eqref{eq:lipschitz-3}.
    Writing $\bX^\top\bep - \tbX^\top\tbep =
    (\tbX+\bX)^\top(\bep-\tbep)/2 + (\bX-\tbX)^\top(\tbep+\bep)/2$
    and similarly
    $\bX^\top\bX - \tbX{}^\top\tbX
    =(\bX+\tbX)^\top(\bX - \tbX)/2
    + (\bX-\tbX)^\top(\bX + \tbX)/2$,
    inequality \eqref{lipschitz-eq-2} follows.
    To prove the Lipschitz condition in $\mathcal N$,
    we note that for a fixed value of $(\bep,\bX,\bh)$,
    the right hand side of \eqref{lipschitz-eq-2}
    is linear in $\|\tbh\|$ while the left-hand side
    is quadratic in $\|\tbh\|$ thanks to either strong convexity
    of $g$ or the assumption on $\kappabar$.
    This implies that $\|\tbh\|$ is bounded
    uniformly for all $(\tbep,\tbX)$ in $\mathcal N$.
    Since $\bep,\bX,\tbep,\tbX,\|\tbh\|,\|\bh\|$ are all bounded in
    $\mathcal N$, \eqref{lipschitz-eq-2}
    divided by $\|\hbbeta-\tbbeta\|$
    provides the desired Lipschitz property.
\end{proof}

\restateIfEnabled{
    \lemmaExistenceWZero*
}

\begin{proof}[Proof of \Cref{lemma:existence-bw0}]
    Under \Cref{assum:main},
    \Cref{lemma:lipschitzness} implies that the map
    $(\bep,\bX)\mapsto \bff = \bX\bh - \bep$ is
    Lipschitz in an open neighborhood of almost every point, and thus 
    $\hbH$ and $\nabla f(\bz_0)$ are defined as Fr\'echet derivatives almost surely 
    in \eqref{def-matrix-H} and \eqref{eq:gradient-f-bXQ0-bep} respectively. 
    To prove \eqref{eq:identity-bw0}, 
    i.e. that the range of $\nabla f(\bz_0) - \langle \ba_0,\bh\rangle (\bI_n-\hbH)$ 
    is the linear span of $\bff$, 
    we study the directional derivative in a direction 
    $\bfeta\in\R^n$.
    For two pairs $(\bep,\bX)$ and $(\tbep,\tbX)$ 
    with $\tbX=\bX + t \bfeta \ba_0^\top = \bX\bQ_0+(t \bfeta + \bz_0)\ba_0^\top$ 
    and $\tbep = \bep + t \bfeta \langle \ba_0,\bh\rangle$, 
    consider the solutions $\hbbeta$ and $\tbbeta$
    defined in \Cref{lemma:lipschitzness} and 
    $\bphi_t = \tbX(\tbbeta-\bbeta) - \tbep$
    with $\bphi_0 = \bX(\hbbeta-\bbeta)-\bep = \bff$.
    When the map $(\bep,\bX)\mapsto \bff$ is
    Fr\'echet differentiable at $(\bep,\bX)$, 
    \bel{pf-lm-3-1}
    \lim_{t\to 0+} (\bphi_t - \bphi_0)/t = \bigl( \nabla f(\bz_0) - \langle \ba_0,\bh\rangle
    (\bI_n-\hbH) \bigr)^\top \bfeta
    \eel
    by the chain rule and the linearity of the Fr\'echet derivative, noting that
    $(\partial/\partial\bep)(\bep-\bX\bh) = \bI_n - \hbH$.
    For this specific choice of $(\tbep,\tbX)$ we have
    \begin{equation}
        \label{eq:specific-choice-tbX-tbep}
        (\tbX-\bX)\bh + \bep - \tbep = 0. 
    \end{equation}
    It follows that the second term in the first line
    of \eqref{lipschitz-eq-1} is zero, so that \eqref{lipschitz-eq-1} gives 
    $$
    \mu n \|\bSigma^{1/2}(\hbbeta-\tbbeta)\|^2
    +
    \|\bphi_0- \bphi_t\|^2
    \le 
    |\langle \ba_0, \bh-\tbh\rangle
    t \bfeta^\top\bff| 
    $$
    due to $\tbX - \bX = t \bfeta \ba_0^\top$. 
    Consequently $\phi_t-\phi_0=0$ when $\bfeta^\top\bff = 0$. 
    This and \eqref{pf-lm-3-1} give \eqref{eq:identity-bw0}. 
    Moreover, for $\bff\neq {\bf 0}$, 
    $\bw_0= \lim_{t\to 0} (\bphi_t - \bphi_0)/t$ for $\bfeta = -\bff/\|\bff\|^2$, 
    so that \eqref{eq:bound-norm-bw0} is an upper bound 
    for $\lim_{t\to 0+} \|\bphi_t - \bphi_0\|/t$ in the case of 
    $\|\bSigma^{-1/2}\ba_0\|=1= - \bfeta^\top\bff$ where 
    $$
    \mu n \|\bSigma^{1/2}(\hbbeta-\tbbeta)\|^2
    +
    \|\bphi_0- \bphi_t\|^2
    \le |t|\, \|\bSigma^{1/2}(\bh-\tbh)\| 
    $$
    by the previous display. 
    For $\mu>0$, the above inequality gives $\|\bphi_0 - \bphi_t\|^2
    \le t^2 (4\mu n)^{-1}$ using $uv\le u^2/4 + v^2$.
    For $\mu=0$, 
    \bes
    \phi_{\min}(\tbW) \|\bSigma^{1/2}(\bh-\tbh)\|^2
    \le  \|\tbX{}^\top(\bh-\tbh)\|^2
    = \|\bphi_t - \bphi_0\|^2
    \ees
    with $\phi_{\min}(\tbW)$ being 
    the smallest eigenvalue of $\tbW = \bSigma^{-1/2}\tbX{}^\top\tbX\bSigma^{-1/2}$. 
    Hence, \eqref{eq:bound-norm-bw0} holds in either cases.   
\end{proof}

\subsection{Loss equivalence to oracle estimators}\label{sec:loss-equiv}
To apply \Cref{thm:L2-distance-from-normal} with respect to $\bz_0$
to $f$ in \eqref{eq:recap},
we will need to control
expectations involving $\|\bw_0\|$,
$\langle \ba_0,\bh\rangle$, $\|\bX\bh\|$
and $\|\by-\bX\hbbeta\|$.
To this end, define the random variables $F_+$ and $F$ by
\bel{def-F_+}
F_+ \defas
\bigl(\|\bg\|^2/n\bigr)
\vee 
\bigl(\|\bep\|^2/(\sigma^2n)\bigr)
\vee\bigl(\|\bep-\bX\bh^*\|^2/(n\risk)\bigr)
\vee 1
\eel 
with $\bg = \bX\bh^*/\|\bSigma^{1/2}\bh^*\|$ 
and the $\bh^*$ and $R_*$ in \eqref{oracle-general-g-strongly-convex},  
and 
\bel{def-F}
F\defas 
2/[1\wedge\max\{\mu, \phi_{\min}(\bSigma^{-1/2}(\bX^\top\bX/n)\bSigma^{-1/2})\}]. 
\eel
We note that the three random vectors
$\bep/\sigma,\bg$ and $(\bep-\bX\bh^*)/\risk{}^{1/2}$
have $N(\mathbf{0},\bI_n)$ distribution,
so that $F_+$ is of the form
$F_+=\max_{i=1,2,3} W_i/n$ where each $W_i$ has the $\chi^2_n$ distribution.
Thus by \Cref{prop:bounded-negative-moments} and properties of the $\chi^2_n$ distribution, 
\bel{inequalities-moments}
     \E[F_+^4] \vee \E[F^{10}] 
 \le C(\gamma,\mu),\quad \E[(F_+-1)^2] \le
 3 \Var[\chi^2_n]/n^{2} = 
 6/n.
\eel
It follows from \eqref{normalization-a_0}, \Cref{lemma:moment-equivalence}
below and \eqref{eq:bound-norm-bw0} that almost surely
\bel{new-bd-1}&& 
\langle \ba_0,\bh\rangle^2
\le \|\bSigma^{1/2}\bh\|^2
\vee(\|\bX\bh\|^2/n)  \le F_+F^2{\risk},\  
\quad
\|\bw_0\|^2
\le F/(2n) ,
\\&&
\|\by-\bX\hbbeta\|^2/n \le 2 F_+ + 2F_+F^2\risk \le 4F_+F^2\risk, 
\label{new-bd-2}
\eel
for the $\bw_0$ in \Cref{lemma:existence-bw0}.
The moment inequalities in \eqref{inequalities-moments}
and the almost sure bounds
\eqref{new-bd-1}-\eqref{new-bd-2} allow us
to control expectations
involving $\|\bw_0\|$,
$\langle \ba_0,\bh\rangle$, $\|\bX\bh\|$
and $\|\by-\bX\hbbeta\|$ throughout the proofs.
The following lemma provides the first
inequality in \eqref{new-bd-1}.

\begin{lemma}[Deterministic lemma]
    \label{lemma:moment-equivalence} 
    Consider the linear model \eqref{LM} and a convex penalty $g(\cdot)$. Let 
    $\hbbeta$ in \eqref{penalized-hbbeta-general}
    and $\bbeta^*,\bh^*,\risk$ be defined in
    \eqref{oracle-general-g-strongly-convex}.
    {Suppose the penalty satisfies $\bu^\top\{(\pa g)(\bu+\bbeta^*) - (\pa g)(\bbeta^*)\}\ge\mu\|\bSigma^{1/2}\bu\|^2\ 
    \forall \bu\in \R^p$ with $\mu\in [0,1/2]$.}  
    Let $F_+$ be defined in \eqref{def-F_+} and $F$ any random variable
    satisfying
    \begin{equation}
    1\le F/2\quad \text{ and either }\quad
            \|\bSigma^{1/2}\bh\|^2/(n\|\bX\bh\|^2)
            \le F/2
            \quad\text{ or }\quad
            \mu^{-1} = F/2,
    \label{eq:condition-F}
    \end{equation}
    for instance \eqref{def-F}.
    Then, 
    \bel{lm-F-1}\qquad
    && \|\bSigma^{1/2} \bh\|^2 \le F^2\max\big({{\overline \sigma}^2},\|\bSigma^{1/2} \bh^*\|^2\big)
   \le F_+F^2{\risk},   
 \\ \label{lm-F-2}
    && \|\bX\bh\|^2/n \le 
    \max\big\{2F{{\overline \sigma}^2},\;\;{{\overline \sigma}^2} + F^2\|\bSigma^{1/2}\bh^*\|^2\big\}
    \le F_+F^2{\risk}, 
    \eel
    where ${{\overline \sigma}^2} = {F_+\sigma^2 + (F_+-1)\|\bSigma^{1/2}\bh^*\|^2 
    = (F_+-1){\risk}  + \sigma^2}$.
%
\end{lemma}

\begin{proof}[Proof of \Cref{lemma:moment-equivalence}]
    The {KKT conditions for}
    $\hbbeta$, i.e., 
    $n \partial g(\hbbeta) = \bX^\top(\by-\bX\hbbeta)$,
    yield
    \bes
        {2(\hbbeta-\bbeta^*)^\top (\pa g)(\hbbeta)}
        &=& {2(\hbbeta-\bbeta^*)^\top \bX^\top(\by-\bX\hbbeta)/n} 
        \cr &{=}&
        (
        \|\by-\bX\bbeta^*\|^2
        -\|\by-\bX\hbbeta\|^2
        -\|\bX(\hbbeta-\bbeta^*)\|^2
        )/n 
        \\ &{=}&
        (\|\bX\bh^*\|^2 - \|\bX\bh\|^2
        -\|\bX(\hbbeta-\bbeta^*)\|^2
        +2 \bep^\top\bX(\bh-\bh^* 
        ) )/n 
        \nonumber
        \\
        &\le&
        (\|\bX\bh^*\|^2 - \|\bX\bh\|^2 + \|\bep\|^2)/n.
        \nonumber
    \ees
    Similarly, the KKT conditions
    $-\bSigma\bh^* = \partial g(\bbeta^*)$
    for $\bbeta^*$ yield
    \begin{equation}
        \label{eq:KKT-conditions-bbeta-star}
        {2(\bbeta^* - \hbbeta)^\top (\pa g)(\bbeta^*) } 
        +\|\bSigma^{1/2}(\hbbeta-\bbeta^*)\|^2 
        \le
        \|\bSigma^{1/2}\bh\|^2 - \|\bSigma^{1/2}\bh^*\|^2
        .
    \end{equation}
    Summing the two above displays yields 
    \bel{lower-bound-sigma-bar}
    {(1+2\mu)}\|\bSigma^{1/2}(\hbbeta-\bbeta^*)\|^2
    &\le& 
    {(F_+-1)\|\bSigma^{1/2}\bh^*\|^2
    + \|\bSigma^{1/2}\bh\|^2 - \|\bX\bh\|^2/n
    + F_+\sigma^2}
    \cr&=& 
    {{\overline \sigma}^2} 
    + {\|\bSigma^{1/2}\bh\|^2 - \|\bX\bh\|^2/n.}
    \eel
    For $\|\bSigma^{1/2} \bh\| \ge F \|\bSigma^{1/2}\bh^*\|$, 
    by the triangle inequality
    \begin{equation}
        \label{eq:lower-bound-triangle-sigma-bar}
     \|\bSigma^{1/2} \bh\|^2(1-1/F)^2
     \le 
     \|\bSigma^{1/2}(\hbbeta-\bbeta^*)\|^2
    \end{equation}
    provides a lower bound on the left-hand side of \eqref{lower-bound-sigma-bar} so that 
    \bel{eq:cases-sigmabar}
    {{\overline \sigma}^2}
    &\ge&  
    {\begin{cases}
    \big\{(1-1/F)^2 + 2/F-1)\big\}\|\bSigma^{1/2} \bh\|^2,
    & \text{ if } \|\bX\bh\|^2/n\ge (2/F)\|\bSigma^{1/2} \bh\|^2, 
    \cr \big\{(1-1/F)^2(1+2\mu) - 1\big\}\|\bSigma^{1/2} \bh\|^2, 
    & \text{ if } \|\bX\bh\|^2/n < (2/F)\|\bSigma^{1/2} \bh\|^2,    
    \end{cases}} 
    \cr &\ge & F^{-2}\|\bSigma^{1/2} \bh\|^2  
    \eel
    {due to $F = 2/\mu\ge 4$ in the second case.} 
    This gives \eqref{lm-F-1}. 
    For $\|\bSigma^{1/2} \bh\| \ge F \|\bSigma^{1/2}\bh^*\|$
    by \eqref{lower-bound-sigma-bar},    
    \eqref{eq:lower-bound-triangle-sigma-bar} and
    \eqref{eq:cases-sigmabar} we have  
    \bes
    \|\bX\bh\|^2/n \le {{\overline \sigma}^2}
    + \|\bSigma^{1/2} \bh\|^2\{1-(1-1/F)^2\}
    \le {{\overline \sigma}^2} +
    F^2{{\overline \sigma}^2}\{1-(1-1/F)^2\}
    = 2F{{\overline \sigma}^2}, 
    \ees
    and for $\|\bSigma^{1/2} \bh\| < F \|\bSigma^{1/2}\bh^*\|$ we have 
    $\|\bX\bh\|^2/n \le {{\overline \sigma}^2}+F^2\|\bSigma^{1/2}\bh^*\|^2$
    and thus \eqref{lm-F-2} holds.
\end{proof}

\subsection{Existence and properties of $\hbH$ and $\df$} 

\begin{proposition}
    \label{prop:hat-bH}
    Let $\bX\in\R^{n\times p}$ be any fixed design matrix,
    and $\hbbeta(\by)=\argmin_{\bb\in\R^p}\big\{\|\by-\bX\bb\|^2/(2n) +g(\bb)\big\}$.
    Then, the following statements hold. 

    \begin{enumerate}
        \item
            $\|\bX(\hbbeta(\by)-\hbbeta(\tby))\|
            \le \|\by-\tby\|$ for all $\by,\tby\in\R^n$, i.e., 
            $\by\mapsto \bX\hbbeta(\by)$ is 1-Lipschitz.
          Its gradient $\hbH$ 
          exists almost everywhere by Rademacher's
          theorem, that is, for almost every $\by$ there exists
          $\hbH\in\R^{n\times n}$
          with $\|\hbH\|_{op}\le 1$
          such that $\bX\hbbeta(\tby)=\bX\hbbeta(\by)+\hbH{}^\top \bfeta +
          o(\|\bfeta\|)$.
        \item
         For almost every $\by$, matrix $\hbH$ is symmetric with
         eigenvalues in $[0,1]$. 
         Consequently, with $\df = \trace(\hbH)$ as degrees of freedom, 
         $(n-\df)(1-\df/n)\le \|\bI_n-\hbH\|_F^2 \le n-\df$. 
    \end{enumerate}
\end{proposition}
\begin{proof}
    A proof of (i) is given in \cite{bellec2016bounds}.
    For completeness, the argument is the following:
    by \eqref{eq:lipschitz-3} with $\tbX=\bX$,
    $\tbep = \tby-\bX\bbeta$ and $\bep=\by -\bX\bbeta$ we have
\begin{equation}
    n D_g(\hbbeta(\tby),\hbbeta(\by))
    + \|\bX\hbbeta(\tby) - \bX\hbbeta(\by))\|^2
    \le (\by-\tby)^\top\bX(\hbbeta(\by)-\hbbeta(\tby)).
    \label{eq:lipschitz-continuity-y}
\end{equation}
Using
$D_g(\hbbeta(\tby),\hbbeta(\by))\ge0$ by monotonicity of the subdifferential
and the Cauchy-Schwarz inequality yields the desired Lipschitz property.
For (ii), define
\bes
u(\by)&=&(\|\by\|^2 - \|\by-\bX\hbbeta(\by)\|^2)/2 - n g(\hbbeta(\by))
    \\&=&\sup_{\bb\in\R^p} \{ \by^\top \bX\bb - \|\bX\bb\|^2/2 - n g(\bb)\}.
\ees
The function $u:\R^n\to\R$ is convex in $\by$ as a supremum of affine functions,
and $\bX\hbbeta(\by)$ is a subgradient of $u$ at $\by$.
Alexandrov's theorem as stated in \cite[Theorem D.2.1]{niculescu2006convex}
states that any convex $u:\R^n\to\R$
is twice differentiable at $\by$ for almost every $\by$
in the following sense:
$u$ is Fr\'echet differentiable at $\by$ with gradient $\nabla u(\by)$
and there exists a symmetric positive semi-definite matrix $\bS$ such that
for every $\varepsilon>0$ there exists $\delta>0$ such that for all $\tby\in\R^n$,
$$\|\tby-\by\|\le \delta\quad\text{implies}\quad \sup_{\bv\in\partial u(\tby)}\|\bv-\nabla u(\by)-\bS(\tby-\by)\|\le \varepsilon \|\tby-\by\|.
$$
By (i) and the definition of $\hbH$, for almost every $\by$ it holds
that $\bX\hbbeta(\tby)=\bX\hbbeta(\by) + \hbH{}^\top(\tby-\by) + o(\|\tby-\by\|)$.
Combining these two results and taking $\bv=\bX\hbbeta(\tby)$,
we get that $\bS=\hbH$ for almost every $\by$.
\end{proof}

\begin{restatable}{lemma}{lemmaDfStronglyConvexBoundedAwayFromN}
    \label{lemma:n-df-strongly-convex}
    Let \Cref{assum:main} be fulfilled with $n\ge 2$.
    Then there exists an event $\Omega_0$ independent of $(\bz_0,\bep)$
    such that
    $$
    \Omega_0 
    \subset
    \left\{
        n-\df
        \ge \|\bI_n - \hbH\|_F^2
    \ge C_*(\gamma,\mu)\, n \right\}
    \text{ with }
    \begin{cases}
        \P(\Omega_0^c)=0 & \text{ if } \gamma<1, 
        \\ \P(\Omega_0^c) \le e^{-n/2} & \text{ if } \gamma\ge1,
    \end{cases}
    $$
    where $C_*(\gamma,\mu)\in(0,1)$ depends on $\{\gamma,\mu\}$ only.
\end{restatable}
\begin{proof}[Proof of \Cref{lemma:n-df-strongly-convex}]
    If $\gamma<1$, the choice $C_*(\gamma,\mu)=(1-\gamma)$ works 
    with probability one because $\rank(\hbH)\le \rank(\bX) \le p\le \gamma n$
    and $\|\hbH\|_{op}\le 1$.

    If $\gamma\ge 1$ then we have $\mu>0$ in \Cref{assum:main}.
    Let $\Omega_0 = \{\|\bX\bQ_0\bSigma^{-1/2}\|_{op} \le \sqrt p + 2 \sqrt n \}$.
    By \cite[Theorem II.13]{DavidsonS01},
    $\P(\Omega_0) \ge 1-e^{-n/2}$ 
    due to $\bX\bQ_0\bSigma^{-1/2} 
    = \bX\bSigma^{-1/2}(\bI_p - (\bSigma^{-1/2}\ba_0)(\bSigma^{-1/2}\ba_0)^\top$ 
    with $\|\bSigma^{-1/2}\ba_0\|=1$. 
    Next, we hold $\bX$ fixed and study the derivatives of $\bX\hbbeta$
    with respect to $\by$. Let
    $\hbbeta(\by)$ be as in \Cref{prop:hat-bH}.
    Let $\bP=\bI_n-\bz_0\bz_0^\top /\|\bz_0\|^2$ be the projection
    onto $\{\bz_0\}^\perp$ so that $\bP\bX=\bP\bX\bQ_0$.
    Let $\by,\tby$ be such that $\bz_0^\top(\by-\tby) = 0$,
    or equivalently $\bP(\by-\tby) = \by-\tby$.
    By \eqref{eq:lipschitz-continuity-y} and \eqref{strong-convex},
    \bes
    n \mu \|\bSigma^{1/2}(\hbbeta(\by) - \hbbeta(\tby))\|^2
    + \|\bX(\hbbeta(\by) - \hbbeta(\tby) )\|^2
    &\le&
    (\by-\tby)^\top\bX(\hbbeta(\by) - \hbbeta(\tby))
  \\&=&
    (\by-\tby)^\top\bP\bX(\hbbeta(\by) - \hbbeta(\tby)).
    \ees
    On $\Omega_0$,
    $\mu(\sqrt\gamma + 2)^{-2}\|\bX\bQ_0(\hbbeta(\by)-\hbbeta(\tby))\|^2\le
    n\mu\|\bSigma^{1/2}(\hbbeta(\by)-\hbbeta(\tby))\|^2
    $. Combined with the above display, this implies
    $(1+\mu(\sqrt\gamma + 2)^{-2})\|\bP\bX(\hbbeta(\by)-\hbbeta(\tby))\| \le \|\bP(\by-\tby)\|$.
    With $\tby = \by + \bP\bfeta$
    and by definition of $\hbH$ we have
    $L^{-1}\|\bP\hbH\bP\bfeta\| \le \|\bP\bfeta\| + o(\|\bfeta\|)$
    for $L=(1+\mu(\sqrt\gamma +2)^{-2})^{-1}$,
    hence $\|\bP\hbH\bP\|_{op}\le L$.
    Since $\rank(\bP)=n-1$,
    by Cauchy's interlacing theorem
    $\bI_n - \hbH$ has at least $n-1$ eigenvalues
    no smaller than $1- L > 0$.
    Finally, since $\bI_n-\hbH$ is symmetric
    with eigenvalues in $[0,1]$
    by \Cref{prop:hat-bH},
    $$
    n-\df
    =
    \trace[\bI_n-\hbH]
    \ge\|\bI_n - \hbH\|_F^2
    \ge (n-1)(1-L)^2 
    \ge n C_*$$
    with $C_* = (1-L)^2/2$
    thanks to $n\ge 2$.
\end{proof}

\subsection{Lower bound on $\|\by-\bX\hbbeta\|^2/n$}
The following lemmas are useful to bound from below
the denominator in \eqref{definition-eps_0-a_0}.

\begin{lemma}
    \label{lemma:uniformly-bounded}
    Let \Cref{assum:main} be fulfilled. Then
    $\E[\xi_0^2]\le \C(\gamma,\mu) n \risk$ and
    \bel{eq:uniformly-bounded}
    \E[( (1-\df/n) \langle \ba_0,\bh\rangle + n^{-1}\langle \bz_0,\by-\bX\hbbeta\rangle)^2]/\risk
    &\le& \C(\gamma,\mu) n^{-1},
    \\\E[ I_{\Omega_0}( \langle \ba_0,\bh\rangle 
    +
    (n-\df)^{-1}
    \langle \bz_0,\by-\bX\hbbeta\rangle)^2]/\risk
    &\le& \C(\gamma,\mu) n^{-1}, 
    \label{eq:uniformly-bounded-2}
    \eel
    where $\Omega_0$ is the event from \Cref{lemma:n-df-strongly-convex}.
\end{lemma}
\begin{proof}[Proof of \Cref{lemma:uniformly-bounded}]
    By \eqref{eq:identity-bw0} combined with 
    \eqref{eq:bound-norm-bw0}, \eqref{new-bd-2}
    and \eqref{inequalities-moments},
    \begin{equation}
        \E[
        \langle \bw_0,\by-\bX\hbbeta\rangle^2
        ]
        \le
        \E[
        \|\bw_0\|^2
        \|\by-\bX\hbbeta\|^2
        ]
        \le
        \E[ (F/(2n))  4n F_+F^2]
        \le \C(\gamma,\mu) \risk
        \label{eq:bound-w0-inner-f}
    \end{equation}
    Similarly, by definition of $V^*(\theta)$ in \eqref{Var(xi_0)},
    $
    \E[\xi_0^2]
    =\E[V^*(\theta)]
    =
    \E[\|\by-\bX\hbbeta\|^2 + 
    \|\nabla f(\bz_0)\|_F^2]
    $.
    Using \eqref{new-bd-1}-\eqref{new-bd-2} we have
    $\E[\|\by-\bX\hbbeta\|^2 \le 4 n \risk \E[F_+F^2]$ and
    \bel{bound-nabla-f}
    \E[\|\nabla f(\bz_0)\|_F^2]
    &\le&
    \E[2 \|\bI_n-\hbH\|_F^2\langle \ba_0,\bh\rangle^2
    + 2 \langle \bw_0,\by-\bX\hbbeta\rangle^2
    ]
    \cr
    &\le&
    n
    \E[ 2 \langle \ba_0,\bh\rangle^2]
    + \C(\gamma,\mu) \risk
    \\
    &\le&
    n \risk
    \C(\gamma,\mu)
    \nonumber
    \eel
    thanks to $\nabla f(\bz_0)$ in \eqref{eq:identity-bw0},
     and
    $(a+b)^2\le2(a^2+b^2)$ for the first inequality,
    $\|\bI_n-\hbH\|_F^2\le n$ by \Cref{prop:hat-bH}
    and
    \eqref{eq:bound-w0-inner-f} for the second inequality,
    and \eqref{new-bd-1}-\eqref{inequalities-moments}
    for the third inequality.
    This provides $\E[\xi_0^2] \le \C(\gamma,\mu) n \risk$.
    Next, \eqref{eq:uniformly-bounded} holds due
    the bound \eqref{eq:bound-w0-inner-f} and
    the relationship in \eqref{-xi0} between
    $\xi_0$, $\langle \bw_0,\by-\bX\hbbeta\rangle$ and the integrand
    in left-hand side of \eqref{eq:uniformly-bounded}.
    Then \eqref{eq:uniformly-bounded-2} follows from
    \eqref{eq:uniformly-bounded} and
    $I_{\Omega_0}(1-\df/n)^{-2} \le C_*(\gamma,\mu)^{-2}$
    by \Cref{lemma:n-df-strongly-convex}.
\end{proof}

\begin{restatable}{lemma}{lemmaDeltasABCD}
    \label{lemma:Deltas-a-b-c-d-strongly-convex}
    Let $\hbbeta$ be as in
    \eqref{penalized-hbbeta-general}
    for convex $g$
    and let $\bbeta^*,\bh^*,\risk$ 
    be as in \eqref{oracle-general-g-strongly-convex}.
    Then
    \bel{lower-boud-f-Delta-a-d}
    (1-\df/n)^2/8
    &\le&
    \|\by-\bX\hbbeta\|^2/(n\risk)
    +
    \Delta_n^a+ \Delta_n^b + \Delta_n^c 
        \label{lower-bound-y-Xhbbeta-Delta-abcd}
    \\&\le&
    V^*(\theta)/(n\risk)
        + \Delta_n^d
        + \Delta_n^a+ \Delta_n^b + \Delta_n^c 
        \label{lower-bound-V^*-Delta-e}
    \eel
    where $V^*(\theta)$ is defined in \eqref{Var(xi_0)} and
    $\Delta_n^a,\dots,\Delta_n^d$ are nonnegative terms defined as
    \begin{align}
    \Delta_n^a &\defas
    \sigma^2\big|(1-\df/n)-\bep^\top(\by-\bX\hbbeta)/(n\sigma^2) \big|^2
    \big/\risk,
    \label{eq:def-delta-n-a}
    \\
    \Delta_n^b &\defas
    (F_+ -1)_+\|\by-\bX\hbbeta\|^2/(n\risk),
    \label{eq:delta-n-b}
    \\
    \Delta_n^c &\defas
    |(1-\df/n)\langle \ba_*,\bh\rangle
    - 
    \bg^\top(\bX\hbbeta-\by)/n|^2
    /\risk
    \label{eq:delta-n-c}
    \\
    \Delta_n^d &\defas
      n^{-1} |2 \bw_0^\top (\bI_n-{\hbH})(\by-\bX\hbbeta)\langle \ba_0,\bh\rangle|/\risk.
    \nonumber
    \end{align}
    where 
    $\bg=\bX\bh^*/\|\bSigma^{1/2}\bh^*\|$
    and $\ba_* = \bSigma\bh^* / \|\bSigma^{1/2}\bh^*\|$.
\end{restatable}
\begin{proof}[Proof of \Cref{lemma:Deltas-a-b-c-d-strongly-convex}]
    By the triangle inequality and definitions of $\Delta_n^a,\Delta_n^c$,
\begin{align}
    (1-\df/n)\sigma
 &\le
 (\bep/\sigma)^\top(\by-\bX\hbbeta)/n + 
(\Delta_n^a\risk)^{1/2},
\label{520}
\\
(1- \df/n)\langle \ba_*,\bh\rangle
  &\le
(\bg^\top(\bX\hbbeta-\by))/n
+(\Delta^c \risk)^{1/2},
\label{eq:lower-boud-g-a-Delta-c}
\\
(1- \df/n)(\sigma^2 + \bh^\top\bSigma\bh^*)
  &\le
  (\bep - \bX\bh^*)^\top(\by-\bX\hbbeta)/n
  + (\Delta_n^a)^{1/2}\risk
  + (\Delta_n^c     )^{1/2}\risk
  \nonumber
\end{align}
where the last line follows from the weighted sum
$\sigma\eqref{520}
+\|\bSigma^{1/2}\bh^*\|\eqref{eq:lower-boud-g-a-Delta-c}$ 
and using $\sigma \vee \|\bSigma^{1/2}\bh^*\| \le \risk{}^{1/2}$
for the last two terms.
By the KKT conditions of $\hbbeta$ and $\bbeta^*$,
$$(\bbeta^*-\hbbeta)^\top \partial g(\bbeta^*)
= (\bh-\bh^*)^\top\bSigma \bh^*,
\qquad
(\hbbeta-\bbeta^*)^\top \partial g(\hbbeta)
= (\hbbeta- \bbeta^*)^\top\bX^\top(\by-\bX\hbbeta)/n.$$
Summing these equalities
and using the monotonicity of the subdifferential yields
\bel{KKT-conditions-oracle-combined}
\|\bSigma^{1/2}\bh^*\|^2 
+ \|\by-\bX\hbbeta\|^2/n
  &\le &
\bh^\top\bSigma \bh^*
+ (\hbbeta-\bbeta^*)\bX^\top(\by-\bX\hbbeta)/n
+ \|\by-\bX\hbbeta\|^2/n
\cr
  & = &
\bh^\top\bSigma \bh^*
+ (\by-\bX\bbeta^*)^\top(\by-\bX\hbbeta)/n.
\eel
Combining \eqref{KKT-conditions-oracle-combined} multiplied by $1-\df/n$
with the line after \eqref{eq:lower-boud-g-a-Delta-c} gives
\bes
  &&(1-\df/n)
(\risk + \|\by-\bX\hbbeta\|^2/n)
  \\&\le&
(2-\df/n)(\by-\bX\bbeta^*)^\top(\by-\bX\hbbeta)/n
+
(\Delta_n^a)^{1/2}\risk+ (\Delta_n^c)^{1/2}\risk
  \\&\le&
2 \|\by-\bX\bbeta^* \| \| \by-\bX\hbbeta \|/n
+
2(\max\{\Delta_n^a,  \Delta_n^c \})^{1/2}\risk
\ees
using the Cauchy-Schwarz inequality
and $(2-\df/n)\le 2$ for the last inequality. Using
$(2a+2b)^2\le 8 (a^2+b^2)$
for the right-hand side
with $\|\by-\bX\bbeta^*\|^2\|\by-\bX\hbbeta\|^2/(n^2\risk{}^2)
\le \|\by-\bX\hbbeta\|^2/(n\risk) + \Delta_n^b$
completes the proof of \eqref{lower-boud-f-Delta-a-d}.
The second inequality, \eqref{lower-bound-V^*-Delta-e},
then follows from \eqref{eq:identity-bw0}, \eqref{Var(xi_0)} and
\bes
  &&
\trace\bigl[
    \bigl( (\bI_n-\hbH)\langle \ba_0,\bh\rangle
+ \bw_0 (\by-\bX\hbbeta) \bigr)^2
\bigr]
\\
&=&
\|\bI_n-\hbH\|_F^2\langle \ba_0,\bh\rangle^2
+
(\bw_0^\top(\by-\bX\hbbeta))^2
+
2 \bw_0^\top(\bI_n-\hbH)(\by-\bX\hbbeta)\langle \ba_0,\bh\rangle
\ees
which implies
$V^*(\theta) - \|\by-\bX\hbbeta\|^2
\ge - n\Delta_n^d\risk$.
\end{proof}

\begin{lemma}
    \label{lemma:Delta_n-to-0}
    Define
    $\Delta_n\defas\Delta_n^a+\Delta_n^b+\Delta_n^c+\Delta_n^d$
    where $\Delta_n^a,...,\Delta_n^d$ are defined
    in \Cref{lemma:Deltas-a-b-c-d-strongly-convex}.
    Under \Cref{assum:main} we have
    $\E[\Delta_n]\le C(\gamma,\mu) n^{-1/2}$.
\end{lemma}
\begin{proof}[Proof of \Cref{lemma:Delta_n-to-0}]
We bound each of
$\Delta_n^a,\Delta_n^b,\Delta_n^c,\Delta_n^d$ separately. 
We have
$\Delta_n^b\le(F_+-1)4F_+F^2$ by \eqref{new-bd-2}
so that $\E[\Delta_n^b]\le \C(\gamma,\mu) n^{-1/2}$
by virtue of  
\eqref{inequalities-moments}.
For $\Delta_n^a$ 
we have $\Delta_n^a = n^{-2}\sigma^{-2} |\sigma^2(n-\df)- \bep^\top(\by-\bX\hbbeta)
|^2/\risk$.
By  the Second Order Stein formula
(\Cref{prop:2nd-order-stein-Bellec-Zhang})
with respect to $\bep$ conditionally on $\bX$,
\bes
\E[\Delta_n^a]
           &=&
n^{-2} \E\big[\|\by-\bX\hbbeta\|^2/\risk + \sigma^2\trace(\{\bI_n-\hbH\}^2)/\risk\big]
           \le
n^{-1} \E\big[4F_+F^2  + 1\big]
\ees
where we used 
$\trace(\{\bI_n-\hbH\}^2)\le n$ from \Cref{prop:hat-bH}
and \eqref{new-bd-2} for the inequality.
Thanks to \eqref{inequalities-moments}, this shows
that $\E[\Delta_n^a]\le n^{-1} \C(\gamma,\mu)$. 
Similarly for $\Delta_n^d$ in \eqref{lower-bound-V^*-Delta-e},
$\Delta_n^d
\le 2 n^{-1} \|\bw_0\| \|\by-\bX\hbbeta\| |\langle \ba_0,\bh\rangle|/\risk
\le n^{-1} 2 (F/2)^{1/2} 2 F_+ F^2$
hence $\E[\Delta_n^d] \le n^{-1}\C(\gamma,\mu)$ 
by \eqref{new-bd-2} and \eqref{inequalities-moments}. 
For $\Delta_n^c$  we have
$\bg=\bz_0$ for $\ba_0=\ba_*$ so that 
$\E[\Delta_n^c]\le \C(\gamma,\mu) n^{-1}$
by \eqref{eq:uniformly-bounded}.
\end{proof}

\subsection{Event $\Omega_n$}
\label{sec:proof-high-probability-events}
With $\Omega_0,C_*(\gamma,\mu)$ in \Cref{lemma:n-df-strongly-convex} and
$\Delta_n$ in \Cref{lemma:Delta_n-to-0}, let
\begin{equation}
    \label{Omega_n}
    \Omega_n = \Omega_0 
    ~\cap~
    \bigl\{
        \E_0[\Delta_n] \vee \Delta_n \le C_*^2(\gamma,\mu)/16
    \bigr\}.
\end{equation}
By the union bound, Markov's inequality and the bound on $\E[\Delta_n]$
in \Cref{lemma:Delta_n-to-0},
\begin{equation}
\P(\Omega_n^c)\le \P(\Omega_0^c) + \C(\gamma,\mu) n^{-1/2}
\le C_0 n^{-1/2}
\label{eq:bound-proba-Omega_n}
\end{equation}
thanks to $\P(\Omega_0^c)\le e^{-n/2}$ in \Cref{lemma:n-df-strongly-convex}.
By \eqref{lower-bound-y-Xhbbeta-Delta-abcd},
\begin{equation}
    \Omega_n \subset \{
        \|\by-\bX\hbbeta\|^2 \ge\risk n C_*^2(\gamma,\mu)/16
    \}.
    \label{eq:lower-bound-f-in-Omega_n}
\end{equation}
Since $\Omega_0^c$ is independent of $\bz_0$,
taking the condition expectation $\E_0$ of 
\eqref{lower-bound-V^*-Delta-e} in $\Omega_0$ gives
\begin{equation}
    \Omega_n \subset \{ 
        \Var_0[\xi_0]
        \wedge
        \E_0[\|\by-\bX\hbbeta\|^2]
    \ge\risk n C_*^2(\gamma,\mu)/16 \}.
    \label{eq:lower-bound-Var-0-in-Omega_n}
\end{equation}

\subsection{Proofs of \Cref{lemma:negligible-term,lemma:delta-0,lemma:check-V-variance} and \Cref{theorem:variance-iff}
}
\label{sec:proof-variance}

\restateIfEnabled{
    \lemmaNegligibleTerm*
}

\begin{proof}[Proof of \Cref{lemma:negligible-term}]
    With $\Omega_n$ in \eqref{Omega_n},
    \eqref{eq:extra-term-to-0}
    follows from \eqref{eq:lower-bound-Var-0-in-Omega_n}
    and \eqref{eq:bound-w0-inner-f}.
\end{proof}

\restateIfEnabled{
    \lemmaDeltaZero*
}
\begin{proof}[Proof of \Cref{lemma:delta-0}]
    Let $\Omega_n$ be as in \eqref{Omega_n}.
    The first inequality in \eqref{eq:delta-0-to-0}
    follows from the triangle inequality.
    By \eqref{eq:lower-bound-Var-0-in-Omega_n} we
    have
    $\E\bigl[
        I_{\Omega_n} ~
        \E_0[|\widehat{V}(\theta) - V^*(\theta)|] / \Var_0[\xi_0]
    \bigr]
    \le
    \E[|\widehat{V}(\theta) - V^*(\theta)|] 16 /(n\risk
    C_*^2(\gamma,\mu)
    )
    $.
With  $V^*(\theta),\widehat{V}(\theta)$
in \eqref{Var(xi_0)}-\eqref{def-hat-V-theta}
and $\nabla f(\bz_0)^\top$ in \eqref{eq:identity-bw0},
$$
V^*(\theta) - \widehat{V}(\theta)
= \langle \bw_0, \by-\bX\hbbeta \rangle^2 
+2 \bw_0^\top (\bI_n-{\hbH})(\by-\bX\hbbeta)\langle \ba_0,\bh\rangle
.
$$
Using $\|\bI_n-\hbH\|_{op}\le 1$ from \Cref{prop:hat-bH}
and
\eqref{new-bd-1}-\eqref{new-bd-2} we find
by the Cauchy-Schwarz inequality
$
|V^*(\theta) - \widehat{V}(\theta)|
\le (2+ 2\sqrt 2)\risk F_+F^3$.
    The proof of \eqref{eq:delta-0-to-0}
    is complete by virtue 
    Holder's inequality and the moment bounds
    \eqref{inequalities-moments}.
\end{proof}

\restateIfEnabled{
    \lemmaCheckVariance*
}

\begin{proof}[Proof of \Cref{lemma:check-V-variance}]
    By the triangle inequality for the Euclidean norm in $\R^2$,
    \begin{equation}
    |\check V(\ba_0)^{1/2} - \widehat{V}(\theta)^{1/2}|
    \le \|\bI_n-\hbH\|_F |\langle \ba_0,\bh\rangle + (n-\df)^{-1}\langle \bz_0,\by-\bX\hbbeta\rangle|.
    \label{eq:difference-two-variance}
    \end{equation}
    Let $\Omega_n$ be as in \eqref{Omega_n}.
    Using $\widehat{V}(\theta)\ge\|\by-\bX\hbbeta\|^2$,
    the lower bound \eqref{eq:lower-bound-f-in-Omega_n}
    and $\|\bI_n-\hbH\|_F^2\le n$ by \Cref{prop:hat-bH},
    \bes
    &&\E[
        I_{\Omega_n}
        |{\check V(\ba_0)^{1/2}}/{\widehat{V}(\theta)^{1/2}} - 1|^2
    ]
    \\&\le&
    \E[
        I_{\Omega_n}
        (
    \langle \ba_0,\bh\rangle + (n-\df)^{-1}\langle \bz_0,\by-\bX\hbbeta\rangle
    )^2
    ] 16/
    (\risk C_*^2(\gamma,\mu))
    \ees
    so that $\Omega_n\subset \Omega_0$ and
    \eqref{eq:uniformly-bounded-2} completes the proof
    for the first term in the maximum in \eqref{eq:check-variance-1}.
    For the second term in the maximum,
    by the triangle inequality
    for the norm $\E_0[(\cdot)^2]^{1/2}$, we have
    $|\E_0[\check V(\ba_0)]^{1/2} - \E_0[\widehat{V}(\theta)]^{1/2}|
    \le
    \E_0[|\check V(\ba_0)^{1/2} - \widehat{V}(\theta)^{1/2}|^2]^{1/2}$.
    The proof is completed by using
    again \eqref{eq:difference-two-variance}, the lower
    bound \eqref{eq:lower-bound-Var-0-in-Omega_n}
    on $\E_0[\|\by-\bX\hbbeta\|^2]$ in $\Omega_n$
    and the same argument as for the first term in the maximum.
\end{proof}

\restateIfEnabled{
    \theoremVarianceConsistentIFF*
}
\begin{proof}[Proof of \Cref{theorem:variance-iff}]

    (v) $\Leftrightarrow$ (iv) 
    is due to
    $C_*(\gamma,\mu) n \le \|\bI_n-\hbH\|_F^2 \le n$ in $\Omega_n$
    by \Cref{lemma:n-df-strongly-convex}
    and \Cref{prop:hat-bH} combined with
    \eqref{eq:DeBias-within-constant-factor}.

    (iv) $\Leftrightarrow$ (vi) follows from
    $C_*(\gamma,\mu) n \le \|\bI_n-\hbH\|_F^2 \le n$ in $\Omega_n$.

    (vi) $\Leftrightarrow$ (vii) is proved in \Cref{lemma:check-V-variance}.

    (iii) $\Rightarrow$ (i),
    (iii) $\Rightarrow$ (vi) and
    (iii) $\Rightarrow$ (vii),
    are shown in the proof of 
    \Cref{thm:from-consistency}.

    (iv) $\Rightarrow$ (iii) follows from
    $\|\by-\bX\hbbeta\|^2/(n\risk) = O_\P(1)$
    by \eqref{new-bd-2} and \eqref{inequalities-moments}.

    (iii) $\Rightarrow$ $\delta_1^2\to^\P 0$ was shown in
    the proof of \Cref{thm:from-consistency}, and $\delta_1^2\to^\P 0$
    implies
    $\E[\|\nabla f(\bz_0)\|_F^2]/\E_0[\|\by-\bX\hbbeta\|^2]\to^\P 0$
    and 
    $\E[\trace[(\nabla f(\bz_0))^2]]/\E_0[\|\by-\bX\hbbeta\|^2]\to^\P 0$
    so that (iii) $\Rightarrow$ (ii) holds.

    By \Cref{lemma:delta-0},
    (ii) implies that $\E_0[\widehat{V}(\theta)]/\E_0[\|\by-\bX\hbbeta\|^2]-1
    = \E_0[\|\bI_n-\hbH\|_F^2\langle \ba_0,\bh\rangle^2]
    /\E_0[\|\by-\bX\hbbeta\|^2]$ converges to 0 in probability.
    Since $\E_0[\|\by-\bX\hbbeta\|^2]/(n\risk) = O_\P(1)$
    by \eqref{new-bd-2} and \eqref{inequalities-moments},
    combined with $\|\bI_n-\hbH\|_F^2\ge C_*^2(\gamma,\mu) n$  in $\Omega_0$
    by \Cref{lemma:n-df-strongly-convex},
    this implies $I_{\Omega_0}\E_0[\langle \ba_0,\bh\rangle^2]/\risk =
    \E_0[
    I_{\Omega_0}
    \langle \ba_0,\bh\rangle^2]/\risk \to^\P 0$
    as $\Omega_0$ is independent of $\bz_0$.
    Thus (ii) implies (iii) by Markov's inequality with respect to $\E_0$.

    Finally, to show (i) $\Leftrightarrow$ (ii), 
    we have by the Gaussian Poincar\'e inequality
    $$
    \Var_0[
    \|\by-\bX\hbbeta\|^2
    ]
    \le \E_0[\|[\nabla f(\bz_0)](\by-\bX\hbbeta)\|^2]
    \le \E_0[\|\nabla f(\bz_0)\|_{op}^2 \|\by-\bX\hbbeta\|^2].$$
    With $\nabla f(\bz_0)$ in \eqref{eq:identity-bw0}
    and the bounds \eqref{new-bd-1}-\eqref{new-bd-2} we have
    $\|\by-\bX\hbbeta\|^2\le 4nF_+F^2 \risk$ and
    $\|\nabla f(\bz_0)\|_{op}^2\le2(2 F_+F^3 + F_+F^2)\risk$
    thanks to $\|\bI_n-\hbH\|_{op}\le 1$ by \Cref{prop:hat-bH}.
    Combined with the lower bound \eqref{eq:lower-bound-Var-0-in-Omega_n}
    on $\Var_0[\xi_0]$
    and the moment upper bounds \eqref{inequalities-moments} we obtain
    $\E\bigl[
    I_{\Omega_n}
    \Var_0[\|\by-\bX\hbbeta\|^2]/\Var_0[\xi_0]^2
    \bigr]\le \C(\gamma,\mu) n^{-1/2}$
    which gives (vii) $\Leftrightarrow$ (i).
\end{proof}

\subsection{Proofs of 
\Cref{thm:De-Bias-within-constant} and
asymptotic normality results}
\label{sec:proof-asymptotic-normality}

\restateIfEnabled{
    \theoremConstantTimesResiduals*
}

\begin{proof}[Proof of \Cref{thm:De-Bias-within-constant}]
    Let $\Omega_n$ be as in \eqref{Omega_n}.
    Since $I_{\Omega_n}\|\by-\bX\hbbeta\|^{-2}
    \le 16/(C_*^2(\gamma,\mu) n \risk)$ by
    \eqref{eq:lower-bound-f-in-Omega_n}, using
    \eqref{eq:uniformly-bounded-2}
    completes the proof of \eqref{eq:DeBias-within-constant-factor}.
    For the second part, random variables
    bounded in $L_2$ are stochastically bounded
    so that \eqref{thm:De-Bias-within-constant}
    provides $|\langle \ba_0,\DeBias-\bbeta\rangle|
    = O_\P(1) \|\by-\bX\hbbeta\|/(n-\df)$,
    and $I_{\Omega_0}(1-\df/n)^{-1} \le C_*(\gamma,\mu)^{-1}$
    for $\Omega_0$ in \Cref{lemma:n-df-strongly-convex}
    provides $(1-\df/n) = O_\P(1)$.
\end{proof}

\restateIfEnabled{
    \theoremFromConsistency*
}

\begin{proof}[Proof of \Cref{thm:from-consistency}]
    Let $\Omega_n$ be as in \eqref{Omega_n}.
    Let $\delta_1^2$ be the quantity in \eqref{definition-eps_0-a_0},
    omitting the dependence in $\ba_0$ as it is clear from context.
    Since $\delta_1^2\le 1$ by definition,
    $\E[\delta_1^2] \le \E[\Omega_n\delta_1^2]+ \P(\Omega_n^c)$.
    In $\Omega_n$,
    \eqref{eq:lower-bound-Var-0-in-Omega_n} provides a lower
    bound on the denominator of $\delta_1^2$ so that
    $\E[I_{\Omega_n}\delta_1^2]
    \le
    \E[\|\nabla f(\bz_0)\|_F^2] 16/(n\risk C_*^2(\gamma,\mu))
    $.
    By \eqref{bound-nabla-f} and the bound \eqref{eq:bound-proba-Omega_n}
    on $\P(\Omega_n^c)$,
    we obtain
    \begin{equation}
    \E[\delta_1^2]
    \le
    \E[I_{\Omega_n}\delta_1^2]
    +
    \P(\Omega_n^c)
    \le
    \C(\gamma,\mu)
    \bigl(
    \E[\langle \ba_0,\bh\rangle^2/\risk]
    + n^{-1}
    \bigr)
    +
    C_0(\gamma,\mu) n^{-1/2}
    \label{ineq-single-a0}
    \end{equation}
    Furthermore, 
    $\langle \ba_0,\bh\rangle^2/\risk$ is bounded in $L_2$
    thanks to
    $\E[\langle \ba_0,\bh\rangle^4/\risk^2]
    \le \E[F_+^2F^4]
    \le \C(\gamma,\mu)$
    by \eqref{new-bd-1} and \eqref{inequalities-moments}.
    Since a sequence of random variables uniformly
    bounded in $L_2$ is uniformly integrable,
    the assumption
    $\langle \ba_0,\bh\rangle^2/\risk \to^\P 0$
    implies $\E[\langle \ba_0,\bh\rangle^2/\risk]\to0$
    and thus $\E[\delta_1^2]\to 0$. 
    This completes the proof that $\E[\delta_1^2] \to 0$
    and that $\xi_0/\Var_0[\xi_0]^{1/2}\to^d N(0,1)$ by 
    \Cref{thm:L2-distance-from-normal}.
    Next, by \eqref{eq:limiting-F},
    $(n-\df)\langle \ba_0,\DeBias-\bbeta\rangle
    /\Var_0[\xi_0]^{1/2}
    \to^d N(0,1)$ also holds.
    It remains to prove $V_0/\Var_0[\xi_0]\to^\P 1$
    for all four possible choices for $V_0$.
    By \eqref{eq:consistency-variance},
    $\E[\delta_1^2]\to0$ implies
    $\|\by-\bX\hbbeta\|^2/\Var_0[\xi_0]\to^\P 1$, while
    \begin{equation}
    \label{eq:imply-variance-consistency}
    0\le
    \frac{\widehat{V}(\theta)}{\|\by-\bX\hbbeta\|^2}
    -1
    = 
    \frac{\|\hbH-\bI_n\|_F^2\langle \ba_0,\bh\rangle^2/(n\risk)}{\|\by-\bX\hbbeta\|^2/(n\risk)}.
    \end{equation}
    \Cref{prop:hat-bH} provides $\|\hbH-\bI_n\|_F^2\le n$
    so that the numerator converges to 0 in probability
    thanks to assumption
    $\langle \ba_0,\bh\rangle^2/\risk\to^\P 0$.
    The denominator is bounded from below
    by $C_*^2(\gamma,\mu)/16$ in $\Omega_n$
    by \eqref{eq:lower-bound-f-in-Omega_n}
    and $\P(\Omega_n)\to 1$.
    This proves $\widehat{V}(\theta)/\Var_0[\xi_0]\to^\P 1$ and
    $\check V(\ba_0)/\Var_0[\xi_0]\to^\P 1$ follows
    by \Cref{lemma:check-V-variance}.
    Slutsky's theorem completes the proof 
    as $V_0/\Var_0[\xi_0]\to^\P 1$ for all four possible choices
    for $V_0$. 
    As $\Phi(t)$ is continuous, convergence in Kolmogorov distance 
    is equivalent to convergence in distribution. 
\end{proof}

\restateIfEnabled{
    \theoremMainResult*
}

\begin{proof}[Proof of \Cref{thm:main-result}]
    We
    construct a subset $\overline{S}$
    of the sphere
    such that
    $\langle \ba_0,\bh\rangle^2/\risk
    \to^\P 0$
    uniformly over all
    $\ba_0\in\bSigma^{1/2}\overline{S}$.
    Let $\bv$ be uniformly distributed on the unit Euclidean sphere $S^{p-1}$,
    independently of $(\bX,\by)$,
    and denote by $\nu$ its probability measure.
    The vector $\sqrt p \bv$ is subgaussian in $\R^p$
    \cite[Theorem 3.4.6]{vershynin2018high}, in the sense
    that for any non-zero vector $\bu\in\R^p$,
    $\int \exp\{(\sqrt p \bv^\top\bu)^2/(C^*\|\bu\|^2)\} d\nu(\bv)\le 2$
    for some absolute constant $C^*>0$.
    By Jensen's inequality and Fubini's Theorem,
    \bes
    \int
    \exp
    \Bigl\{
        \E\Bigl[
            \frac{(\bv^\top \bSigma^{1/2}\bh)^2}{C^*\|\bSigma^{1/2}\bh\|^2}
        \Bigr]
    \Bigr\} \; d\nu(\bv)
    &\le&
    \E\Big[\int
        \exp \Bigl\{
            \frac{(\sqrt p \bv^\top \bh)^2}{C^* \|\bSigma^{1/2}\bh\|^2}
        \Bigr\} 
    \; d\nu(\bv)
    \Big]
    \\&\le& 2.
    \ees
    Hence by Markov's inequality, for any positive $x$,
    $\nu(\{\bv\in S^{p-1}: \E[(\bv^\top \bSigma^{1/2}\bh)^2/\|\bSigma^{1/2}\bh\|^2]>C^*x/p\})\le 2e^{-x}$.
    Setting $x=p/a_p$, we obtain that the subset $\overline{S}\subset S^{p-1}$
    defined by
    \eqref{set-tilde-S-}
    has relative volume at least $|\overline{S}|/|S^{p-1}| \ge 1 - 2e^{-p/a_p}$,
    and for all $\ba_0\in\bSigma^{1/2}\overline{S}$,
    \begin{equation}
        \label{bound-second-moment-ba0-bh}
        \E[\langle \ba_0,\bh\rangle^2/\|\bSigma^{1/2}\bh\|^2]
        \le C^*/a_p.
    \end{equation}
    Furthermore, the set $\overline{S}\cap\{\bSigma^{-1/2}\be_j/\|\bSigma^{-1/2}\be_j\|, j\in[p]\}$ has cardinality at least $p-\phi_{\rm cond}(\bSigma)a_p/C^*$ due to
    \bes
        \sum_{j=1}^p
        \frac{1}{\|\bSigma^{-1/2}\be_j\|^{2}}
        \E\Bigl[\frac{\langle \be_j,\bh\rangle^2}{\|\bSigma^{1/2}\bh\|^2}\Bigr]
          &\le&
          \|\bSigma\|_{op} 
          \E\Bigl[
          \frac{\|\bh\|^2}{\|\bSigma^{1/2}\bh\|^2}
          \Bigr]
        \le \phi_{\rm cond}(\bSigma).
    \ees
    To show that $\sup_{\ba_0\in\bSigma^{1/2}\overline{S}}
    \E[\delta_1^2(\ba_0)]\to 0$, thanks
    to \eqref{ineq-single-a0} it is enough to
    prove that $\E[\langle \ba_0,\bh\rangle^2/\risk]\to 0$
    uniformly over $\ba_0\in\bSigma^{1/2}\overline{S}$.
    By the Cauchy-Schwarz inequality,
    \bes
    \E[\langle \ba_0, \bh\rangle^2/\risk] 
          &=&
    \E[\langle \ba_0, \bh\rangle\|\bSigma^{1/2}\bh\|/\risk
    ~
    \langle \ba_0,\bh\rangle /\|\bSigma^{1/2}\bh\|
    ] 
        \\&\le&
    \E[\|\bSigma^{1/2}\bh\|^4/\risk^2]^{1/2}
    (C^*/a_p)^{1/2}
    \ees
    for any $\ba_0\in\bSigma^{1/2}\overline{S}$
    thanks to \eqref{bound-second-moment-ba0-bh},
    while
    $\E[\|\bSigma^{1/2}\bh\|^4/\risk{}^2]\le
    \E[F_+^2F^4]\le
    \C(\gamma,\mu)
    $
    by \eqref{new-bd-1} and \eqref{inequalities-moments}.
    This implies 
    $\sup_{\ba_0\in\bSigma^{1/2}\overline{S}}
    \E[\langle \ba_0,\bh\rangle^2/\risk]\to 0$
    and 
    $\sup_{\ba_0\in\bSigma^{1/2}\overline{S}}\E[\delta_1^2(\ba_0)]\to0$
    hold.

    By \Cref{thm:L2-distance-from-normal}
    this shows that $\xi_0/\Var_0[\xi_0]^{1/2}\to^d N(0,1)$
    uniformly over $\ba_0\in \bSigma^{1/2}\overline{S}$.
    Since the bounds
    \eqref{eq:extra-term-to-0},
    \eqref{eq:imply-variance-consistency}
    are all uniform over all $\ba_0$ with $\|\bSigma^{-1/2}\ba_0\|=1$,
    Slutsky's theorem implies
    $V_0/\Var_0[\xi_0]\to^\P 1$,
    $\xi_0/V_0^{1/2}\to^d N(0,1)$ and
    $\{(n-\df)\langle \ba_0,\bh\rangle + \bz_0^\top(\by-\bX\hbbeta)\}
    /V_0^{1/2}\to^d N(0,1)$
    uniformly over $\ba_0\in\bSigma^{1/2}\overline{S}$
    for all four possible choices for $V_0$, 
    and convergence in Kolmogorov distance follows from convergence in distribution.  
\end{proof}

\restateIfEnabled{
    \theoremGisANorm*
}

\begin{proof}[Proof of \Cref{thm:g-is-a-norm}]
    The first statement of the theorem follows
    from \Cref{lemma:uniformly-bounded}.
    Finally, if $g$ is a norm then the KKT conditions of $\hbbeta$,
    $|\bz_0^\top(\by-\bX\hbbeta)|
    = n | (\bSigma^{-1}\ba_0)^\top \partial g(\hbbeta) |
    \le n g( \bSigma^{-1}\ba_0)$
    since for a norm $g(\bu) = \sup_{\bv\in\partial g(\bu)}\langle \bu,\bv\rangle$.
\end{proof}

\appendix

\section{Integrability of $\phi_{\min}^{-1}(\bX\bSigma^{-1/2}/\sqrt n)$ when $p/n\to\gamma\in(0,1)$}
\label{appendix:integrability-edelman}
In our regression model with Gaussian covariates, the matrix $\bX\bSigma^{-1/2}$
has iid $N(0,1)$ entries, and the inverse of its smallest singular value enjoys
the following
integrability property as $n,p\to+\infty$ with $p/n\to\gamma\in(0,1)$.
\begin{proposition}
    \label{prop:bounded-negative-moments}
    Let $n>p$ and let $\bG$ be a matrix with $n$ rows, $p$ columns and iid $N(0,1)$ entries.
    Then $\bG^\top\bG$ is a Wishart matrix
    and if $n,p\to+\infty$ with $p/n\to\gamma\in(0,1)$ we have for any constant $k$ not growing with $n,p$,
    $$\lim_{p/n\to\gamma}
    \E[\phi_{\min}(\bG^\top\bG/n)^{-k}] = (1-\sqrt\gamma)^{-2k}
    $$
\end{proposition}
\begin{proof}
    Throughout the proof, $p=p_n$ is an implicit function of $n$; we omit the subscript for brevity.
    Since $S_n = \phi_{\min}(\bG^\top\bG/n)\to(1-\sqrt\gamma)^2$ almost surely (cf. \cite{silverstein1985smallest}),
    it is enough to show that the sequence of random variables
    $(S_n^{-k})_{n\ge n_0}$ is uniformly integrable for some $n_0>0$, i.e., that
    $\sup_{n\ge n_0} \E[S_n^{-k} I_{\{ S_n < \eps \}}] \to 0$ as $\eps\to 0.$
    For uniform integrability, we use the following argument from \cite[Section 5]{edelman1988eigenvalues}.
    The matrix $\bG^\top\bG$ is a Wishart matrix and the density of $L = \phi_{\min}(\bG^\top\bG)$ satisfies for $\lambda \ge0$,
    \bes
    f_L(\lambda)
    &\le& \frac{
    \sqrt\pi 2^{-(n-p+1)/2} \Gamma(\frac{n+1}{2})
}{\Gamma(\frac p 2)\Gamma(\frac{n-p+1}{2})\Gamma(\frac{n-p+2}{2})}
    \lambda^{(n-p-1)/2}
    e^{-\lambda/2}
    = \frac{
    \sqrt\pi \Gamma(\frac{n+1}{2})
}{\Gamma(\frac p 2)\Gamma(\frac{n-p+2}{2})}
    f_{\chi^2_{n-p+1}}(\lam)
    \ees
    cf. \cite[Section 5]{edelman1988eigenvalues}.
    The density of $S_n=L/n=\phi_{\min}(\bG^\top\bG/n)$ that we are interested in, is given by
    $f_{S_n}(x) = n f_L(nx)$ for $x\ge 0$. Hence if $0<\eps<(1-\gamma)/2$,
    \begin{equation*}
        \E[ S_n^{-k} I_{ \{S<\eps\} }]
        \le
        \left[\frac{
            \sqrt\pi \Gamma(\frac{n+1}{2}   )   (\frac n 2)^{(n-p+1)/2}
        }{
            \Gamma(\frac p 2)\Gamma(\frac{n-p+1}{2})\Gamma(\frac{n-p+2}{2})
        }
     \right]\int_0^\eps x^{(n-p-1)/2 - k}
        e^{-nx/2}dx.
    \end{equation*}
    The mode of the integrand over $[0,+\infty)$ is $x_n^*=1-p/n-1/n - 2k/n$.
    Thanks to $\eps<(1-\gamma)/2$,
    there exists some $n_1\ge 1$ such that for all $n\ge n_1$,
    \begin{equation}
        \label{smaller-than-mode}
        n-p-1 - 2k \ge n (1-\gamma)/2,
    \end{equation}
    $(1-\gamma)/2$ is smaller than the mode $x_n^*$ and the integral above is bounded by
    $\eps^{(n-p-k+1)/2} e^{-n\eps/2}$.
    Let $\Lambda_n$ denote the bracket of the previous display. 
    Then using Stirling's formula $\Gamma(x+1) \asymp \sqrt{2\pi x} e^{-x} x^x$,
    we have for some constants $n_2,C_2(\gamma)>0$ possibly depending on $\gamma$
    $$\sup_{n\ge n_2}\frac{\log(\Lambda_n)}{(n-p+1)/2} \le C_2(\gamma)$$
    because the main terms (coming from $x^x$ in Stirling's formula) cancel each other.
    Then for any $n\ge n_1\vee n_2$,
    \bel{eq-edelman}
        \E[ S_n^{-k} I_{ \{S_n<\eps\} }]
        &\le& \left(\exp( {C_2(\gamma)} ) \eps\right)^{(n-p+1)/2} \eps^{-k} e^{-n\eps/2} \cr
        &\le& \left(\exp( {C_2(\gamma)} ) \eps\right)^{(n-p+1)/2 - k}  e^{ k C_2(\gamma) - n\eps/2}.
    \eel
    For $n\ge n_1$, \eqref{smaller-than-mode} holds and if $\eps< (\exp C_2(\gamma))^{-1}$ we have
    $$
        \sup_{n\ge n_1\vee n_2}
        \E[ S_n^{-k} I_{ \{S_n<\eps\} }]
        \le
        \left(\exp( {C_2(\gamma)} ) \eps\right)^{(n_1 (1-\gamma)/4 )}
        e^{ k C_2(\gamma)} 
    $$
    which converges to $0$ as $\eps\to 0$.
    This shows uniform integrability of the sequence and proves the claim.
\end{proof}

\section{Proof: $p>n$ without strong convexity}
\label{sec:p-larger-n-without-strong-convexity}

\begin{lemma}
    \label{lemma:edelman}
    Let $\bbeta\in\R^p$ and assume that $p/n\le \gamma$.
    Then for any $\kappa<1$,
    $$
    \P\Bigl(
        \inf_{
            t\in\R, \bu\in\R^p:\|\bu\|_0\le \kappa n
        }
        \Bigl(
        \frac{ \|\bX(\bu-t \bbeta)\|^2 }{n\|\bSigma^{1/2}(\bu-t \bbeta)\|^2}
        \Bigr)
        > \varphi(\gamma, \kappa)^2
    \Bigr) \to 1,
    $$
    for some constant $\varphi(\gamma, \kappa)>0$ depending
    only on $\gamma,\kappa$.
\end{lemma}
\begin{proof}
If $V\subset \R^p$ is a subspace of dimension $d=\lfloor \kappa n \rfloor +1$
and $\bG = \bX\bSigma^{-1/2}$ then by \eqref{eq-edelman}
with $k=0$, $\eps\in(0,(1-d/n)/2)$ and $n$ large enough,
$$
\P(
\inf_{\bv\in \bSigma^{1/2} V: \|\bv\|=1} \|\bG\bv\|^2/n
< \eps
) \le 
\exp\bigl(C_2(\kappa')\log(\eps) (n-p+1)/2 - n\eps/2\bigr).
$$
for constant $\kappa'=(\kappa+1)/2$ thanks to
$1>\kappa'\ge d/n$.
Applying this bound to
the subspace $V_B= \{\bu -t\bbeta, (\bu,t)\in\R^{p+1}:\bu_{B^c}=\mathbf{0}\}$
for $B\subset[p]$ with $|B|\le \kappa n$ and using the union bound,
\begin{align*}
\P\Bigl(\inf_{
            t\in\R, \bu\in\R^p:\|\bu\|_0\le \kappa n
        }
        \Bigl(
        \frac{ \|\bX(\bu-t \bbeta)\|^2 }{n\|\bSigma^{1/2}(\bu-t \bbeta)\|^2}
        \Bigr)
        < \eps
\Bigr)
&\le \binom{p}{\lfloor \kappa n\rfloor}
e^{C_2(\kappa')\log(\eps)(n-d+1)/2 - n\eps/2}
\\&\le
e^{ n \log(e\gamma) + C_2(\kappa')\log(\eps)(n-d+1)/2 - n\eps/2 } 
\end{align*}
using $\binom{p}{q} \le e^{q \log(ep/q)}\le
e^{n \log(ep/n)}
$ with $q=\lfloor \kappa n\rfloor \le n$ and $p/n\le \gamma$.
Since $d\le \kappa n + 1$, choosing
$\eps = 1 \wedge \exp(C_2(\kappa')^{-1}(1-\kappa)^{-1}2 \log(e\gamma))$
the right-hand side of the previous display is bounded from above by $e^{-n\eps/2}$.
This value of $\eps$ provides $\varphi(\gamma,\kappa)^2$.
\end{proof}

\theoremGroupLassoPLargerN*

\begin{proof}[Proof of \Cref{thm:lasso-p-larger-than-n}]
    As in the rest of the paper, $f(\bz_0)=\by-\bX\hbbeta$ and we wish to apply
    \Cref{thm:L2-distance-from-normal} to $\bz_0$ conditionally on
    $(\bep,\bX\bQ_0)$. Instead of applying \Cref{thm:L2-distance-from-normal}
    to $f$,
    and in order to avoid certain events of small probability
    where the sparse eigenvalues of $\bX$ are not well behaved,
    we will apply it to a different function.
    Consider $F_+$ in \eqref{def-F_+} and the events
    \bes
    \Omega_L  &=& \big\{
    \|\hbbeta\|_0\le \kappa n/2
    \big\},
    \quad
    \Omega_\chi=
    \big\{
    F_+ < 2,
    \quad
    (F_+-1)_+ < 4 \sqrt{\log(n)/n} 
    \big\},
    \\
        \Omega_E&=&
        \Big\{\min_{t\in\R, \bu\in\R^p:\|\bu\|_0\le \kappa n}
            \frac{ \|\bX(\bu-t\bbeta)\|}{\|\bSigma^{1/2}(\bu-t\bbeta)\|}
        > 
        \varphi \sqrt n,
        \qquad
        \|\bX\bSigma^{-1/2}\|_{op} < \sqrt{n} (2 + \sqrt{\gamma})
    \Big\}
    \ees
    where $\varphi=\varphi(\gamma,\kappa)$ is the constant from
    \Cref{lemma:edelman}.
    Finally, let $\Omega_{KKT}$ be the event \eqref{eq:KKT-strict-group-lasso}
    that the KKT conditions of $\hbbeta$ hold strictly,
    and set
    $$
    \Omega \defas \Omega_L \cap \Omega_E \cap \Omega_{KKT} \cap \Omega_\chi.
    $$
    We have $\P(\Omega_L)\to 1$ by \eqref{eq:assum-lasso}
    and standard concentration bounds for $\chi^2_n$ random variables
    \cite[Lemma 1]{laurent2000adaptive}
    give $\P(\Omega_\chi) \to 1$.
    \Cref{lemma:edelman} and \cite[Theorem II.13]{DavidsonS01}
    provide $\P(\Omega_E)\to 1$ and
    \eqref{eq:KKT-strict-group-lasso} gives 
    $\P(\Omega_{KKT})=1$.
    These bounds imply $\P(\Omega)\to 1$ by the union bound.

    As the only randomness of the problem comes from
    $(\bep,\bX)$,
    we may choose the underlying probability space
    as $\R^n\times \R^{n\times p}$,
    so that $\Omega,\Omega_L,\Omega_E,\Omega_{KKT}$ are subsets
    of $\R^n\times \R^{n\times p}$.
    We next prove that $\Omega$ is open
    as a subset of $\R^n\times \R^{n\times p}$.
    Indeed, because the KKT conditions are strict in $\Omega$,
    $\Omega$ is a disjoint union of sets of the form
    \begin{multline}
        \label{eq:open-set-fixed-B}
    \Omega_L\cap  \Omega_E \cap \Omega_\chi \cap 
    \{ \|\hbbeta_{G_k}\| > 0, k\in B\} \cap \{
        \|\bX^\top(\by-\bX\hbbeta)\| < n \lambda_k,
        k \in B^c
    \}
    \end{multline}
    over all possible active groups $B\subset\{1,...,K\}$. 
    The sets $\Omega_E,\Omega_\chi$ are open 
    as 
    the inequalities
    are strict. In $\Omega_E$
    the function $(\bep,\bX)\mapsto \hbbeta$
    is locally Lipschitz by \Cref{lemma:lipschitzness},
    hence continuous.
    By continuity, the preimage of the open set $(0,+\infty)$
    by the function
    $\Omega_E\to\R$,
    $(\bep,\bX)\mapsto \|\hbbeta_{G_k}\|$ is open by continuity,
    and the preimage of the open set $(-\infty,n\lambda_k)$
    by the function
    $\Omega_E\to\R$,
    $(\bep,\bX)\mapsto \|\bX_{G_k}^\top(\by-\bX\hbbeta)\|$ is also open,
    again by continuity.
    This shows that the set \eqref{eq:open-set-fixed-B}
    is open for any fixed $B\subset\{1,...,K\}$
    so that $\Omega$ is open as the union of sets
    of the form \eqref{eq:open-set-fixed-B}
    over all $B\subset\{1,...,K\}$ satisfying
    $\sum_{k\in B}|G_k|\le \kappa n/2$.
    This proves that $\Omega\subset\R^{n} \times \R^{n\times p}$
    is open.

    For $F =  2 \max\{1,\|\bSigma^{1/2}\bh\|^2/(n\|\bX\bh\|^2) \}$
    in \Cref{lemma:moment-equivalence},
    \eqref{eq:condition-F} is satisfied
    so that \eqref{lm-F-1}-\eqref{lm-F-2} hold.
    In $\Omega$, we thus have
    $\|\bSigma^{1/2}\bh\|^2 \vee(\|\bX\bh\|^2/n)
    \le F_+F^2 \risk 
    \le 8 \varphi^{-2} \risk$
    and $\|\by-\bX\hbbeta\|/\sqrt n\le F_+^{1/2}\sigma + \sqrt{8}\varphi^{-1} \risk^{1/2}\le 3\sqrt{2} \varphi^{-1}\risk$. 
    Furthermore
    $\|\bw_0\|_2^2\le \varphi^{-1} /n$ in $\Omega_E$
    thanks to $|\Shat|\le \kappa n/2$ and
    the explicit expression for $\bw_0$ in \Cref{lemma:gradient-group-lasso}.
    In summary we have in $\Omega$
    \begin{equation}
    \|\bSigma^{1/2}\bh\|^2 \vee(\|\bX\bh\|^2/n)
    \le 8 \varphi^{-2}\risk,
    ~
    \|\by-\bX\hbbeta\|^2/n
    \le 18\varphi^{-2} \risk,
    ~
    \|\bw_0\|^2\le\varphi^{-2}/n
    \label{eq:new-bound-lasso-GL}
    \end{equation}
    which replace \eqref{new-bd-1}-\eqref{new-bd-2} in the present context.
    By the deterministic inequality \eqref{520},
    in $\Omega$ we have $\df \le |\Shat| \le \kappa n/2$ since $\hbH$ is
    rank at most $|\Shat|$ with operator norm at most one, so that
    \begin{equation}
    I_{\Omega}(1-\kappa/2)^2/8
    \le \|\by-\bX\hbbeta\|^2/(n\risk) + \Delta_n^a + \Delta_n^b + \Delta_n^c.
    \label{lower-bound-kappa-lasso-group-lasso}
    \end{equation}
    Let $(\bep,\bX),(\bep,\tbX)$ both in $\Omega$,
    let $\tbep=\bep$,
    and let $\bh,\tbh, \bff, \tbf$ be as in \Cref{lemma:lipschitzness}.
    Thanks to event $\Omega_E$ and the fact that $|\Shat|\le \kappa n/2$
    and similarly for $\tbbeta$ we have
    $\varphi^2\|\bSigma^{1/2}(\bh -\tbh)\|^2
    \le \|\bX(\bh-\tbh)\|^2/(2n)
    +  \|\tbX(\bh-\tbh)\|^2/(2n)$.
    Thus by \eqref{eq:lipschitz-3},
    $$n\varphi^2\|\bSigma^{1/2}(\bh -\tbh\|^2
    \le (\tbh-\bh)^\top(\bX-\tbX)^\top\bep
    + (\bh-\tbh)^\top(\bX^\top\bX - \tbX{}^\top\tbX)(\bh+\tbh)/2.
    $$
    Summing this inequality with 
    the first line in \eqref{lipschitz-eq-1} we find
    \bel{lipschitz-argument-GL}
&&
    n\varphi^2\|\bSigma^{1/2}(\bh -\tbh)\|^2
    + \|\bff-\tbf\|^2
  \cr&\le&
     (\tbh-\bh)^\top(\bX-\tbX)^\top\bep
    + (\bh-\tbh)^\top(\bX^\top\bX - \tbX{}^\top\tbX)(\bh+\tbh)/2
  \cr&& + (\tbh-\bh)^\top(\bX-\tbX)^\top\bff
  + \bh^\top(\bX-\tbX)^\top(\bff-\tbf).
    \eel
    Thanks to the bounds in \eqref{eq:new-bound-lasso-GL},
    this implies
    $\|\bff-\tbf\| \le 
    L \|(\bX-\tbX)\bSigma^{-1/2}\|_{op}$
    if $\{(\bep,\bX),(\bep,\tbX)\}\subset \Omega$,
    where $L=\C(\gamma,\kappa) \risk^{1/2}$.
    
    For a given $(\bep,\bX\bQ_0)$,
    we define $U_0=\{\bz_0\in\R^n: (\bep,\bX\bQ_0+\bz_0 \ba_0^\top)\in\Omega\}$.
    In $U_0$, the function $f(\bz_0) = \bX(\hbbeta-\bbeta)-\bep$
    is $L$-Lipschitz.
    By Kirszbraun's theorem, there exists a function $F:\R^n\to \R^n$
    that is an extension of $f$,
    i.e., $F(\bz_0) = f(\bz_0)$ for $\bz_0\in U_0$, and such that
    $F$ is $L$-Lipschitz in the whole $\R^n$.
    Note that both function $F$ and $f$
    implicitly depend on $(\bep,\bX\bQ_0)$.
    Since $\Omega$ is open, $U_0$ is also open,
    and thus conditionally on $(\bX\bQ_0,\bep)$,
    \begin{equation}
        \label{eq:equality-gradients-extension}
        \nabla f(\bz_0) = \nabla F(\bz_0),
        \qquad
        \text{ for all } \bz_0 \in U_0.
    \end{equation}
    (Without the openness of $\Omega$ established above,
    equality of the gradients would be unclear).

    Since $F:\R^n\to\R^n$ is such that
    $F(\bz_0)=f(\bz_0)$ in $\Omega$,
    by \eqref{lower-bound-kappa-lasso-group-lasso}
    \bel{eq:lower-boud-F-extension-Omega}
    (1-\kappa/2)^2 I_{\Omega}
       &\le& I_{\Omega}[
           \|\by-\bX\hbbeta\|^2/(n\risk)
           + \Delta_n^a + \Delta_n^b + \Delta_n^c
       ]
    \cr&=& I_{\Omega}[
        \|F(\bz_0)\|^2/(n\risk) 
        + \Delta_n^a + \Delta_n^b + \Delta_n^c
    ].
    \eel
    Taking conditional
    expectations and multiplying both sides by
    $\delta_1^2 \defas
    \E_0[\|\nabla F(\bz_0)\|_F^2]/
    \bigl\{
        \E_0[\|\nabla F(\bz_0)\|_F^2] + \E_0[\|F(\bz_0)\|^2]
    \bigr\}$
    we find
    \begin{equation*}
    \delta_1^2 ( 1-\kappa/2)^2
    \E_0[I_{\Omega}]
    \le 
    \E_0[\|\nabla F(\bz_0)\|_F^2/(n\risk)]
    + \E_0[I_\Omega(\Delta_n^a + \Delta_n^b + \Delta_n^c)],
    \end{equation*}
    due to
    $\delta_1^2\E_0[I_{\Omega}\|F(\bz_0)\|^2]\le \E_0[\|\nabla F(\bz_0)\|_F^2]$
    for the first term and $\delta_1^2\le 1$ for the second.
    Using $\delta_1^2\le 1$ and $1=I_\Omega+I_{\Omega^c}$,
    \bes
       &&\E[\delta_1^2]( 1-\kappa/2)^2
       \le 
    \E[I_{\Omega}\|\nabla F(\bz_0)\|_F^2/(n\risk)]
    + \E[I_\Omega(\Delta_n^a + \Delta_n^b + \Delta_n^c)]
    + \P[\Omega^c](1+L^2/\risk)
    \ees
    where we used that $\|\nabla F(\bz_0)\|_F^2\le n
    \|\nabla F(\bz_0)\|_{op}^2\le n L^2$ in $\Omega^c$
    since $F$ is $L$-Lipschitz.
    We now prove that the three terms
    on the right-hand side of converge to 0.
    For the third term, $L^2/\risk \le \C(\gamma,\kappa)$
    and $\P(\Omega^c)\to 0$ as $\Omega$
    has probability approaching one.
    For the first term,
    since $F$ is $L$-Lipschitz, $\|\nabla F(\bz_0)\|_F^2 \le n L^2$ almost surely
    so that the sequence of random variables 
    $I_{\Omega}\|\nabla F(\bz_0)\|_F^2/(\risk n)$ is uniformly integrable.
    Thanks to uniform integrability,
    if we can prove $\|\nabla F(\bz_0)\|_F^2/(\risk n)\to^\P 0$
    then $\E[I_{\Omega}\|\nabla F(\bz_0)\|_F^2/(\risk n)]\to 0$ holds.
    We use that $I_\Omega \nabla f(\bz_0) =I_\Omega \nabla F(\bz_0)$
    by \eqref{eq:equality-gradients-extension}, and that 
    in $\Omega$ the gradients of $f$ with respect to $\bz_0$ are given
    in \Cref{lemma:gradient-group-lasso} so that by 
    \eqref{eq:new-bound-lasso-GL}
    \bes
    I_{\Omega}\|\nabla F(\bz_0)\|_F/(\risk n)^{1/2}
       &=&
    I_{\Omega}\|\nabla f(\bz_0)\|_F/(\risk n)^{1/2}
    \\&\le& I_{\Omega}
        \big[
        \|\bw_0\|
        \|\by-\bX\hbbeta\| + \|\bI_n-\hbH\|_F |\langle \ba_0,\bh\rangle|
        \big]
        /(\risk n)^{1/2}
    \\&\le&\C(\gamma,\kappa)
    (
    n^{-1/2} + |\langle \ba_0,\bh\rangle|/\risk^{1/2}
    )
    \ees
    which converges to 0 in probability 
    thanks to assumption $\langle \ba_0,\hbbeta-\bbeta\rangle^2/\risk\overset{\P}{\to} 0$.
    Thanks to uniform integrability, this proves
    $\E[I_{\Omega}\|\nabla F(\bz_0)\|_F^2/(\sigma^2 n)] \to 0$.
    It remains to show 
    $\E[I_\Omega( \Delta_n^a+ \Delta_n^b + \Delta_n^c)]\to 0$.
    By definition of $\Delta_n^b$ in \eqref{eq:delta-n-b},
    thanks to $\Omega_\chi$ and \eqref{eq:new-bound-lasso-GL}
    we have $I_\Omega \Delta_n^b
    \le 18 \varphi^{-2}(F_+-1) \le \C(\gamma,\kappa) \sqrt{\log(n)/n}$.
    For $\Delta_n^a$ in \eqref{eq:def-delta-n-a},
    let $\bPi:\R^n\to\R^n$ be the convex projection onto the Euclidean ball
    of radius $\sqrt{18\varphi^{-2}\risk}$,
    then $\bPi(\by-\bX\hbbeta)=\by-\bX\hbbeta$ in $\Omega$ by \eqref{eq:new-bound-lasso-GL} so that
    \bel{bound-Delta_n-grouplasso}
    \E[I_{\Omega}\Delta_n^a]
    &=&\E\big[
        I_{\Omega}
        \big\{
            (1-\df/n) - \bep^\top\bPi(\by-\bX\hbbeta)/(n\sigma^2)
        \big\}^2
    \big] \sigma^2/\risk
      \cr&\le&
      \E[\|\bPi(\by-\bX\hbbeta)\|^2]/(n^2\risk) + \sigma^2/(n\risk)
      \cr&\le&
      18\varphi^{-2}/n + 1/n
      \eel
    by applying \Cref{prop:2nd-order-stein-Bellec-Zhang}
    to the function $\bep\mapsto\bPi(\by - \bX\hbbeta)$ which is 1-Lipschitz
    as the composition of two 1-Lipschitz functions (cf. \Cref{prop:hat-bH}(i)).
    For $\Delta_n^c$ in \eqref{eq:delta-n-c}, let $\bg,\ba_*$
    by as in \Cref{lemma:Deltas-a-b-c-d-strongly-convex}
    and set $\bu_*=\bSigma^{-1}\ba_*$,
    $\bQ_* = \bI_p - \bu_* \ba_*^\top$ and
    note that $(\bg,\bu_*,\bQ_*) = (\bz_0,\bu_0,\bQ_0)$ for $\ba_0=\ba_*$.
    Let also $\bw_*$ be the $\bw_0$ from \Cref{lemma:gradient-group-lasso}
    for $\ba_0=\ba_*$.
    As above for $\ba_0$, for a fixed $(\bep,\bX\bQ_*)$
    the function $\bg\mapsto \by-\bX\hbbeta$ is $L$-Lipschitz
    in $U_* =\{\bg\in\R^n: (\bep,\bX\bQ_* + \bg\ba_*^\top)\in\Omega\}$
    by \eqref{lipschitz-argument-GL} for the value of $L$ given after
    \eqref{lipschitz-argument-GL}.
    Furthermore $\|\by-\bX\hbbeta\|^2\le 18 n \varphi^{-2}\risk$ in $\Omega$.
    By Kirszbraun's theorem, there exists an extension $F_*:\R^n\to\R^n$
    implicitly depending on $(\bep,\bX\bQ_*)$
    such that $F_*(\bg) = \by-\bX\hbbeta$ in $\Omega$
    and $\|F_*(\bg)\|^2\le 18 n \varphi^{-2}\risk$ by
    composing the extension given by Kirszbraun's theorem by the convex
    projection onto the Euclidean ball of radius
    $(18 n \varphi^{-2}\risk)^{1/2}$.
    By \Cref{prop:2nd-order-stein-Bellec-Zhang}
    with respect to $\bg$
    conditionally on $(\bep,\bX\bQ_*)$
    \bes
    \E\bigl[I_\Omega\bigl((n-\df)\langle \ba_*,\bh\rangle + \langle \bw_* + \bg, \by-\bX\hbbeta\rangle\bigr)^2\bigr]
         &=&
    \E\bigl[I_\Omega\bigl(\dv F_*(\bg) + \langle \bg, F_*(\bg)\rangle\bigr)^2\bigr]
      \\ &\le& 18 n \varphi^{-2} \risk +  n L^2.
    \ees
    For the value of $L$ given after \eqref{lipschitz-argument-GL}
    and using the bound \eqref{eq:new-bound-lasso-GL}
    to control $\langle\bw_*, \by-\bX\hbbeta\rangle$ in $\Omega$,
    this gives
    $\E[I_\Omega \Delta_n^c] \le \C(\gamma,\kappa) n^{-1}$.

    This proves
    $(1-\kappa/2)^2 \E[\delta_1^2]\to 0$.
    Consequently $\Xi_0 = \bz_0^\top F(\bz_0) - \dv F(\bz_0)$
    satisfies
    $|\sup_t|\P(\Xi_0/\|F(\bz_0)\|\le t) - \Phi(t)|\to 0$
    by \eqref{eq:consistency-variance-2}. 
    Since $\xi_0=\bz_0^\top f(\bz_0) - \dv f(\bz_0)$ is equal to $\Xi_0$
    on the event $\Omega$ because $F$ is an extension of $f$, we have
    $|\P(\xi_0/\| f(\bz_0)\| \le t) -
    \P(\Xi_0/ \| F(\bz_0)\| \le t)
    |\le 2 \P(\Omega^c) \to 0$ 
    so that
    $\sup_t|\P(\xi_0/\|f(\bz_0)\|\le t) - \Phi(t)| \to 0$ as well.
    The conclusion \eqref{eq:thm_lasso-p-bigger-n}
    is obtained by controlling the term $\bw_0^\top(\by-\bX\hbbeta)/\|\by-\bX\hbbeta\|$
    by $\|\bw_0\|$ which is bounded as in \eqref{eq:new-bound-lasso-GL}
    in $\Omega$.

    It remains to show that \eqref{eq:thm_lasso-p-bigger-n}
    holds uniformly over all $\ba_0\in\bSigma^{1/2}\overline{S}$
    and to derive the properties of $\overline{S}$.
    The proof of the relative volume bound on $\overline{S}$
    and the lower bound on the cardinality
    of $\{j\in[p]: \be_j/\|\bSigma^{-1/2}\be_j\| \in \bSigma^{1/2}\overline{S}
    \}$ is the same as in the proof of \Cref{thm:main-result}
    given around \eqref{bound-second-moment-ba0-bh},
    and for $\ba_0\in\bSigma^{1/2}\overline{S}$
    inequality \eqref{bound-second-moment-ba0-bh} holds.
    For such $\ba_0$,
    $\E[I_\Omega \langle \ba_0,\bh\rangle^2/\risk]
    \le 8\varphi^{-2}\E[\langle \ba_0,\bh\rangle^2/\|\bSigma^{1/2}\bh\|^2]
    \le 8\varphi^{-2} C^*/a_p\to0$
    by \eqref{eq:new-bound-lasso-GL} for the first inequality
    and \eqref{bound-second-moment-ba0-bh} for the second.
\end{proof}

\section{Strict KKT conditions with probability one for the group Lasso}
\label{sec:proof-gradient-GL}

\begin{restatable}{lemma}{kktStrictLemma}
    \label{lemma:kkt-strict}
Consider a design matrix $\bX\in\mathbb R^{n\times p}$
and a response vector $\by\in\R^n$
for which the joint distribution of $(\bX,\by)$ admits a density
with respect to the Lebesgue measure.
Consider a partition of $\{1,...,p\}$ into groups $(G_1,...,G_K)$
and any minimizer
$$\hbbeta\in \argmin_{\bb\in\mathbb R^p} \frac{1}{2n}\|\bX \bb- \by\|^2 
+ \|\bb\|_{GL}
,
\qquad
\|\bb\|_{GL} \defas
\sum_{k=1,...,K} \lambda_k \|\bb_{G_k}\|_2$$
for some deterministic $\lambda_1,...,\lambda_K > 0$. 
There exists an open set $U\subset \R^{n\times (1+p)}$ such that
$\P((\by,\bX)\in U) = 1$ and
the KKT conditions are strict in $\{(\by,\bX)\in U\}$ in the sense that
\begin{equation}
    \label{eq:KKT-strict-group-lasso}
\Bigl\{(\by,\bX)\in U \Bigr\}
\subset
\Bigl\{\forall k=1,...,K, \quad \hbbeta_{G_k} = 0 \quad \Rightarrow\quad \| \bX_{G_k}^\top(\by - \bX \hbbeta) \|_2 < n \lambda_k  \Bigr\}.
\end{equation}
Finally, $\widehat{B}=\{k\in[K]:\|\hbbeta_{G_k}\|>0\}$ is constant
in a small neighborhood of any point in $U$.
\end{restatable}

\begin{proof}[Proof of \Cref{lemma:kkt-strict}]

Consider a fixed $B\subset \{1,...,K\}$ and its complementary set $B^c$, and consider
the Group-Lasso estimator $\hbbeta(B)$ with the additional constraint $\bb_{G_k}=0$ for every 
$k\in{B^c}$.
Now consider a group $k \in B^c$. Since the joint distribution of $(\bX
,\by)$ has a density with respect to the Lebesgue measure,
the conditional distribution of $\bX_{G_k}$ given
$(\by, (\bX\be_{j})_{j\notin G_k})$ also admits a density with respect to the Lebesgue measure.
Conditionally on $(\by, (\bX\be_{j})_{j\notin G_k})$, two cases may appear:
\begin{enumerate}
    \item
        If $\by - \bX \hbbeta(B) = 0$, the KKT condition for group $G_k$ hold strictly since $\lambda_k \ne 0$.
\item
    If $\by - \bX \hbbeta(B) \ne 0$, the distribution of $\bX_{G_k}$ given
$(\by, (\bX\be_{j})_{j\notin G_k})$ and the distribution of
$\bX_{G_k}^\top(\by - \bX \hbbeta(B))$ given 
$(\by, (\bX\be_{j})_{j\notin G_k})$ both have a density with respect to the Lebesgue measure.
The sphere of radius $n\lambda_k$ has measure 0 for any continuous distribution, hence
$$\P\left(
    \| \bX_{G_k}^\top(\by - \bX \hbbeta(B)) \|_2 \ne n\lambda_k
\Big| \by, (\bX\be_{j})_{j\notin G_k}\right) = 1.$$
\end{enumerate}
Finally, the unconditional probability
$\P( \| \bX_{G_k}^\top(\by - \bX \hbbeta(B)) \|_2 \ne n \lambda_k )$ is also one.
Let
$U= \cap_{B\subset\{1,...,K\}}
\cap_{k\notin B}
\{(\by,\bX):
    \|\bX_{G_k}^\top(\by-\bX\hbbeta(B))\|_2 \ne n\lambda_k
\}$.
Then $\P((\by,\bX)\in U) = 1$ as a finite intersection of
events of probability one
and \eqref{eq:KKT-strict-group-lasso} holds.
The set $U$ is open as a finite intersection of open sets,
since 
$\{(\by,\bX):
    \|\bX_{G_k}^\top(\by-\bX\hbbeta(B))\|_2 \ne n\lambda_k
\}$ is open by continuity
of $(\by,\bX)\mapsto \bX^\top(\by-\bX\hbbeta(B))$
by the claim following \eqref{lipschitz-eq-1}.

Next, to show that $\widehat{B}$ is constant in
a neighborhood of every point in $U$, 
set $U_\delta = 
\cap_{B\subset\{1,...,K\}}
\cap_{k\notin B}
\{(\by,\bX):
    |\|\bX_{G_k}^\top(\by-\bX\hbbeta(B))\|_2/(n\lambda_k) - 1|>\delta
\}$ for all $\delta>0$.
We have $U=\cup_{\delta>0}U_\delta$
and the set $U_\delta$ is open by continuity of 
$(\by,\bX)\mapsto|\|\bX_{G_k}^\top(\by-\bX\hbbeta(B))\|_2/(n\lambda_k) - 1|$,
which follows from the continuity 
of $(\by,\bX)\mapsto \bX^\top(\by-\bX\hbbeta(B))$
by the claim following \eqref{lipschitz-eq-1}.
For any $(\bybar,\bXbar)\in U$,
there exists some $\delta>0$ with $(\bybar,\bXbar)\in U_\delta$.
Let $\overline{B}=\{k\in[K]: \|\bbetabar_{G_k}\|>0\}$.
By continuity of $(\by,\bX)\mapsto \bX^\top(\by-\bX\hbbeta)$
thanks to the claim following \eqref{lipschitz-eq-1},
there exists a neighborhood $\mathcal N$ of $(\bybar,\bXbar)$
with $\mathcal N\subset U_\delta$ such that
for all $(\by,\bX)\in\mathcal N$,
$\|\bX_{G_k}^\top(\by-\bX\hbbeta)\|/(n\lambda_k)<1-\delta/2$
for $k\notin \overline{B}$ 
and
$\|\bX_{G_k}^\top(\by-\bX\hbbeta)\|/(n\lambda_k)>1-\delta/2$
for $k\in \overline{B}$.
Since $\mathcal N\subset U_\delta$,
$\|\bX_{G_k}^\top(\by-\bX\hbbeta)\|/(n\lambda_k)>1-\delta/2$
implies 
$\|\bX_{G_k}^\top(\by-\bX\hbbeta)\|/(n\lambda_k) =1$
so that 
that $\widehat{B}=\overline{B}$ in $\mathcal N$.
\end{proof}

\LemmaGradientGroupLasso*

\begin{proof}[Proof of \Cref{lemma:gradient-group-lasso}]
    By \Cref{lemma:kkt-strict}, $\widehat{B}$ and $\Shat$
    are constant in a sufficiently small neighborhood
    of almost every $(\bybar,\bXbar)$.
    The additional assumption that $\bXbar_{\Sbar}^\top\bXbar_{\Sbar}$
    is invertible provides that
    $\bX_{\Sbar}^\top\bX_{\Sbar}$ is invertible by continuity of the smallest
    eigenvalue in a small enough
    compact neighborhood of $(\bybar,\bXbar)$, and in this neighborhood
    $(\by,\bX)\mapsto \hbbeta$ is Lipschitz
    by the sentence following \eqref{lipschitz-eq-2}
    and thus almost everywhere differentiable by Rademacher's theorem.
    The formulae for $\nabla \hbbeta(\bz_0)$,
    $\hbH$ and $\bw_0$ involving the matrix $\bM$ in 
    \eqref{matrix-M-group-lasso}
    are then obtained by differentiating
    the KKT conditions restricted to $\Shat$
    in this neighborhood, that is,
    $\bX_{G_k}^\top(\by-\bX\hbbeta) = n\lambda_k \hbbeta_{G_k} / \|\hbbeta_{G_k}\|$ for all $k\in\widehat{B}$.
\end{proof}

\section{Proof of Theorem~\texorpdfstring{\ref{thm:consistency-variance}}{}}
\label{appendix:proof-variance-quadratic-case}

\begin{proof}[Proof of \Cref{thm:consistency-variance}] 
With $\E[\|\bmubar + \bAbar^\top \bz\|^2]=\|\bmubar\|^2+\|\bAbar\|_F^2$
in mind, consider
\bes
\widehat{\Var[\xi]}
&=&
\|f(\bz)-(\bmubar+\bAbar^\top\bz)\|^2 + 
\trace[\{\nabla f(\bz) - \bAbar\}^2]
\\&& + 2(f(\bz) - \bmubar-\bAbar^\top\bz)^\top(\bmubar+\bAbar^\top\bz)
+2\trace[\{\nabla f(\bz) - \bAbar\}\bAbar]
\\&&
+ \bigl(\|\bmubar+\bAbar^\top\bz\|^2 - \|\bmubar\|^2 - \|\bAbar\|_F^2 \bigr)
+\|\bmubar\|^2 + \|\bAbar\|_F^2 + \trace[\bAbar^2].
\ees
By the triangle and Cauchy-Schwarz inequality inequalities, 
\bes
&& \E\big[\big|\|f(\bz)\|^2+\trace[\{\nabla f(\bz)\}^2] - \Var[\xi]\big|\big]
\cr &\le & \E\big[\|f(\bz)-(\bmubar + \bAbar^\top \bz)\|^2\big]+\E[\|\nabla f(\bz)-\bAbar\|_F^2] 
\cr && + 2\big\{\E\big[\|f(\bz)-(\bmubar + \bAbar^\top \bz)\|^2\big]+\E[\|\nabla f(\bz)-\bAbar\|_F^2]\big\}^{1/2}\big\{\|\bmubar\|^2+2\|\bAbar\|_F^2\big\}^{1/2}
\cr && + \E\big[\big|\|\bmubar + \bAbar^\top \bz\|^2 - \|\bmubar\|^2 - \|\bAbar\|_F^2\big|\big]
+ \big|\|\bmubar\|^2 + \|\bAbar\|_F^2 + \trace(\bAbar^2) - \Var[\xi]\big|. 
\ees
We have $\E\big[\|f(\bz)-(\bmubar + \bAbar^\top \bz)\|^2\big]
\le \E[\|\nabla f(\bz)-\bAbar\|_F^2] \le \epsdoublebar_{1,2}^2\Var[\xi]/2$ 
by the Gaussian Poincar\'e inequality, 
$\E\big[\big|\|\bmubar + \bAbar^\top \bz\|^2 - \|\bmubar\|^2 - \|\bAbar\|_F^2\big|^2\big]
= \|\bAbar\bmubar\|^2+2\trace\{(\bAbar\bAbar^\top)^2\}
\le \|\bAbar\|_{op}^2C_0^2\Var[\xi]$, and 
$0\le 1 - \{\|\bmubar\|^2 + \|\bAbar\|_F^2 + \trace(\bAbar^2)\}/\Var[\xi] \le \epsdoublebar_{1,2}^2$ 
as in \eqref{quadratic-approximation-3}. Thus
\eqref{thm:consistency-variance-two-terms-1} holds 
and the conclusions follow. 
\end{proof}

\bibliographystyle{plainnat}
\bibliography{asymptotic-normality.bib}

\end{document}